%% file: elliptic.tex
\documentclass[11pt]{amsart}
\pdfoutput=1 
\usepackage{macros}
\date{}

\newcommand{\const}{ (i \pi)^{-1}}
\newcommand{\constinv}{ i \pi }
\newcommand{\ubar}{\br{u}}
\newcommand{\I}{\mathscr{I}}

\newcommand{\Atiyah}{\op{At}}

\newcommand{\Id}{\op{Id}}

\newcommand{\wbar}{\br{w}}

\newcommand{\Fields}{\E}
\newcommand{\F}{\mathcal{F}}

\newcommand{\jhol}{\mscr{J}^{hol}}
\theoremstyle{thm}

\begin{document}
\title{A geometric construction of the Witten genus, II}
\author{Kevin Costello}
\thanks{Partially supported by NSF grants DMS 0706945 and DMS 1007168, and by a Sloan fellowship.}

\address{Department of Mathematics, Northwestern University.}
\email{costello@math.northwestern.edu}

\maketitle

\input{a_restate_theorem.txt}
\input{b_a_derived.txt}

\input{c_Geometry_using_lie_algebras.txt}

\input{c_z_mapping_spaces.txt}
\input{d_Atiyah_class.txt}
\input{d_z_volume_form.txt}
\input{e_holomorphic_CS.txt}
\input{e_z_hcs2.txt}

\input{f_RG_flow.txt}

\input{g_symmetries.txt}
\input{h_statement_of_the_main_theorem.txt}

\input{j_counterterms.txt}
\input{l_obstruction_group.txt}

\input{m_obstruction_graph.txt}

\input{p_eisenstein.txt}

\input{r_appendix.txt}

\newpage
\def\cprime{$'$}


\end{document}

%% file: a_restate_theorem.txt
\section*{Introduction}
Let $X$ be a compact complex manifold with vanishing second Chern character.  The \emph{Witten genus} of $X$ was introduced in the physics literature by Witten \cite{Wit87} and Alvarez, Killingback, Mangano and Windey \cite{AlvKilMan87} as the partition function of a $2$-dimensional quantum field theory built from maps from an elliptic curve to $X$.   

The Witten genus has an expression in terms of the characteristic numbers of $X$, as follows.   If $E$ is an elliptic curve with a holomorphic volume form $\omega$,  let $E_{2k}(E,\omega)$ be the Eisenstein series evaluated on $(E,\omega)$.  If $E = \C / \Lambda$ with volume form $\d z$, then the Eisenstein series is given by the formula
$$
E_{2k}(E,\omega) = \sum_{\lambda \in \Lambda \setminus \{0\}} \lambda^{-2k}.  
$$
The Witten class 
$$\op{Wit}(X,E,\omega) \in \oplus H^i(X, \Omega^i_X)$$ 
is defined by the formula
$$
\log \op{Wit}(X,E,\omega) = \sum_{k \ge 2} \frac{(2k-1)!}{ (2 \pi i)^{2 k}}  E_{2k} (E,\omega) \op{ch}_{2k} (T X ).
$$
The Witten genus of $X$ is the integral over $X$ of the component of the Witten class lying in $H^{n}(X,\Omega^n_X)$. 

In this paper, I give a rigorous justification of the original physics definition of the Witten genus.  I define a $2$-dimensional quantum field theory built from maps from an elliptic curve to $X$, and show that the partition function of this theory is the Witten genus of $X$. 

\subsection{}
Let me describe how the classical field theory I consider is constructed, and how the partition function is defined.  

In \cite{Cos11b}, I define a class of classical field theories which I call \emph{cotangent field theories}.  Given any system of elliptic differential equations\footnote{A formal definition of the kind of elliptic differential equations of interested is given in \cite{Cos11b}: I call them \emph{elliptic moduli problems}.} on a manifold $M$ one can construct the corresponding cotangent theory.  In this paper, we are interested in a field theory on an elliptic curve $E$, defined as the cotangent theory to the (derived) moduli space of degree $0$ holomorphic maps from $E$ to a complex manifold $X$. 

In \cite{Cos11} I develop a definition of \emph{quantization} of a classical field theory.  This definition has the property that the space of possible quantizations of a classical field theory on a manifold $M$ is the global sections of a sheaf on $M$.  Further, quantizations have a descent property: if a group $G$ acts properly discontinuously on $M$, then a $G$-equivariant quantization of an equivariant classical theory on $M$ descends to a quantization of the corresponding theory on $M / G$.

The Witten genus will arise for us from quantizing the cotangent theory of degree $0$ holomorphic maps from an elliptic curve $E$ to $X$.  In order to construct such a quantization for every $E$, it suffices to construct a quantization of the corresponding theory on the complex plane $\C$. 
\begin{theorem*}
Consider the cotangent theory to holomorphic maps from $\C$ to a complex manifold $X$ (where we work in the formal neighbourhood of constant maps).  

There is a natural bijection between
\begin{enumerate}
\item Quantizations of this theory, invariant under $\op{Aff}(\C)$ and an additional $\C^\times$ action (to be discussed later).
\item Trivializations of the second Chern character $ch_2(T X ) \in H^2(X, \Omega^2_{cl}(X))$. 
\end{enumerate}
\end{theorem*}
Here $\Omega^{2}_{cl}(X)$ denotes the sheaf of closed holomorphic $2$-forms on $X$.
\begin{corollary}
A trivialization of $ch_2(TX)$ leads to a quantization of the cotangent theory to the moduli space of degree $0$ holomorphic maps $E \to X$ for every elliptic curve $E$. 
\end{corollary}

\subsection{}
In \cite{Cos11b}, I show that such a quantization of the cotangent theory associated to any system of elliptic equations leads, roughly, to a volume form on the derived moduli space of solutions.  More precisely, we find a right $D$-module structure on the sheaf of functions: this is equivalent to a flat connection on the canonical bundle, or to a trivialization of the sheaf of volume forms up to scalar multiplication.  I call this structure a \emph{projective volume form}.  

The quantization of the cotangent theory to the space of degree $0$ holomorphic maps $E \to X$ thus leads to a projective volume form on this derived mapping space.  

\subsection{}
We will let $X^E$ denote the derived space of degree $0$ maps.  Choosing a holomorphic volume form on $E$ leads to an isomorphism $X^E \iso T[-1] X$. Note that there is an isomorphism of sheaves of algebras on $X$
$$
\Oo_{T^[-1] X} \iso \Omega^{-\ast}_X,
$$
where $\Omega^{-\ast}_X$ indicates the sheaf of holomorphic forms on $X$ where $\Omega^{i}_X$ is placed in cohomological degree $-i$. 
 
Thus,
$$
H^0(X, \Oo_{T[-1] X} ) = \oplus_{i} H^i(X, \Omega^i_X).
$$ 
Let $\d Vol_E$ denote the projective volume form on $X^E$ coming from the quantization described above.  Integrating against $\d Vol_E$ give a linear map
\begin{align*}
H^0(X, \Oo_{T[-1] X} ) = \oplus H^i(X,\Omega^i_X) &\to \C \\
\alpha & \mapsto \int_{T[-1] X} \alpha \d Vol_E,
\end{align*}
defined up to a scalar factor.  We can normalize this scalar factor so that, if $\alpha \in H^n(X, \Omega^n_X)$ is the class Serre dual to $1 \in H^0(X,\Oo_X)$, 
$$\int_{T[-1] X} \alpha \d Vol_E = 1.$$

The second main theorem calculates this linear map. 
\begin{theorem*}
For every trivialization of $ch_2(T X)$, the corresponding projective volume form $\d Vol_E$ on $X^E \iso T[-1] X$ has the property that the integration map
$$\alpha \to \int_{T[-1] X} \alpha \d Vol_E$$ 
is the map sending $\alpha \in \oplus_i H^i(X,\Omega^i_X)$ to
$$
\int_{X} \left[ \op{Wit}(X,E,\omega) \alpha \right]_n.
$$
Here $[-]_n$ indicates the projection onto the component in $H^n(X,\Omega^n_X)$, and 
$$
\op{Wit}(X,E,\omega) \in \oplus H^i(X,\Omega^i_X)
$$
is the Witten class of $X$.
\end{theorem*}

\subsection{}
One can restate this theorem as follows.  The space $T[-1] X$ is equipped with a projective volume form $\d Vol_0$ characterized by the property that the integration map
$$
\alpha \mapsto \int_{T[-1]X} \alpha \d Vol_0
$$
sends $\alpha \in \oplus H^i(X,\Omega^i_X)$ to $\int_X [\alpha]_n$. 

Any two projective volume forms differ by a function (defined up to a scalar).  The theorem states that
$$\d Vol_E = \op{Wit}(X,E, \omega) \d Vol_0$$
where $\op{Wit}(X,E,\omega)$ is viewed as a function on $T[-1] X$.  

\subsection*{Factorization algebra formulation of the results}
In \cite{Cos10a} I announced the results proved here in a slightly different form, using the language of factorization algebras.    In order to connect the results proved here with the statement of \cite{Cos10a}, let me explain a little about the results of the work in progress \cite{CosGwi11}.   Let $M$ be a manifold, and suppose we have a quantum field theory on $M$ in the sense of \cite{Cos11}, Chapter 5.  Let me briefly recall what this is.  We have a graded vector bundle $E$ on $M$, over $\C$, say.  The space of fields will be $\E = \Gamma(M,E)$.  The space $\E$ of fields is equipped with a symplectic pairing of cohomological degree $-1$, arising from a map of vector bundles $E \otimes E \to \op{Dens}(M)$, of cohomological degree $-1$.    

The space $\E$ is equipped with a differential $Q : \E \to \E$, which is a differential operator compatible with the symplectic pairing.  We will let 
$$\Oo(\E) = \prod_{n \ge 0} \op{Hom} ( \E^{\otimes n}, \C)_{S_n}$$
denote the algebra of formal power series on $\E$. In this expression, $\op{Hom}$ denotes continuous linear maps, $\otimes$ denotes the completed projective tensor product, and the subscript $S_n$ denotes coinvariants.  We will let 
$$\Oo_{red}(\E) = \Oo(\E)/  \C$$
denote the algebra of functions on $\E$ modulo constants. 

The essential part of the data of a quantum field theory is a collection
$$
I[L] \in \Oo_{red}(\E) [[\hbar]]
$$
of effective interactions, defined for all $L > 0$.  These must satisfy three axioms: a \emph{renormalization group equation}, expressing $I[L]$ in terms of $I[\eps]$ if $\eps < L$; a \emph{locality axiom}, saying that as $L \to 0$,  $I[L]$ becomes more and more local; and the \emph{quantum master equation}, saying that for all $L$,
$$
Q I[L] + \tfrac{1}{2} \{ I[L], I[L] \}_L +  \hbar \Delta_L I[L] = 0.
$$
Here, $\Delta_L : \Oo(\E) \to \Oo(\E)$ is a Batalin-Vilkovisky operator, which depends on the scale $L$.  The quantum master equation implies that for all $L$, the operator
\begin{align*}
\what{Q}_L : \Oo(\E)[[\hbar]] & \to \Oo(\E)[[\hbar]] \\ 
\Phi & \mapsto Q \Phi + \{I[L], \Phi\}_L + \hbar \Delta_L \Phi 
\end{align*}
is of square zero.  

The renormalization group equation implies that the complexes
$$
\left( \Oo(\E)[[\hbar]], \what{Q}_L \right)
$$
are chain homotopic for different values of $L$.  
\begin{definition}
The complex of \emph{global observables} of the quantum field theory is the complex 
$\left( \Oo(\E)[[\hbar]], \what{Q}_L \right)$, for any $L > 0$. 
\end{definition}
\subsection{}
The results of \cite{CosGwi11} imply the following.
\begin{theorem}
\begin{enumerate}
\item Any quantum field theory on a manifold $M$, in the sense of \cite{Cos11}, yields a factorization algebra $\F$ on $M$, over the ring $\C[[\hbar]]$.  
\item There is a quasi-isomorphism
$$
\F(M) \iso \left( \Oo(\E)[[\hbar]], \what{Q}_L \right)
$$
between the complex of global sections of the factorization algebra, and the complex of global observables of the quantum field theory.
\item Quantum field theories in the sense of \cite{Cos11} on open subsets of $M$ form a sheaf, as do factorization algebras defined on open subsets of $M$.  The map from quantum field theories to factorization algebras is a map of sheaves.
\item  If a discrete group $G$ acts properly discontinuously on $M$, then any quantum field theory on $M$ invariant under the  $G$ action descends to one on $M / G$. Factorization algebras satisfy the same descent property, and the map from quantum field theories to factorization algebras is compatible with descent.
\end{enumerate}
\end{theorem}
In this paper, we will see that the Witten genus of $X$ is encoded in the scale-infinity effective action $I[\infty]$ constructed from a certain quantum field theory of maps to $X$.  Thus, the main theorem of this paper implies the results announced in \cite{Cos10a}.    Indeed, the construction of the quantum field theory on $\C$ presented in this paper yields a translation invariant factorization algebra on $\C$.  The global sections of this factorization algebra on an elliptic curve $E$ is quasi-isomorphic to the complex of global observables of the quantum field theory on $E$.  This complex of global observables is computed in this paper by explicitly evaluating the scale $\infty$ effective interaction $I[\infty]$ (and seeing that it is the logarithm of the Witten class).

\subsection{}
I should briefly compare the construction of the Witten genus in this paper to the work of Gorbounov, Malikov and Schechtman \cite{GorMalSch00}.  These authors show that the Witten genus of $X$ is the character of a sheaf of vertex algebras on $X$ called the \emph{chiral differential operators} of $X$.  Conjecturally, the factorization algebra of observables of the field theory constructed in this paper is an analytic incarnation of the chiral differential operators.  If this is the case, then one could view the results of this paper as being the Lagrangian counterpart of the results of \cite{GorMalSch00}.  Indeed, they show the Witten genus arises as the character of an operator on the Hilbert space of the theory, whereas in this paper we find the Witten genus directly from the functional integral.  

\subsection{}
Other recent related work is that of Grady and Gwilliam \cite{GraGwi11}.  These authors consider a $1$-dimensional field theory related to the $2$-dimensional field theory considered here.  They find that the partition function of the theory (on $S^1$) is the $\what{A}$-class of the manifold.  

\subsection{}
Nick Rozenblyum's MIT thesis \cite{Roz11} contains some exciting developments related to the results presented here.  The techniques Rozenblyum develops allow one to give a purely algebro-geometric construction of a projective volume form on the derived mapping space $X^E$ from an elliptic curve $E$ to an algebraic variety $X$.  It is natural to conjecture that the projective volume form Rozenblyum constructs coincides with the projective volume form constructed here.  

\subsection{}
The plan of the paper is as follows.  Part 1 develops general techniques which allow us to treat field theories with non-linear targets using the techniques of \cite{Cos11}.   In order to do this, I introduce some formalism related to formal geometry for describing a certain class of ``derived manifolds''.  In formal derived geometry \cite{Lur09, Hin01}, every formal derived space (``formal moduli problem'') can be represented by a dg Lie algebra (or an $L_\infty$ algebra).  For the purposes of this paper, a global derived space is a functor from a certain category of dg ringed manifolds to the category of simplicial sets, satisfying a sheaf property.  I show how to construct such a functor from an ``$L_\infty$ space'', which is a manifold equipped with a sheaf of curved $L_\infty$ algebras. I show that every complex manifold can be represented by an $L_\infty$ space. Further, certain derived mapping spaces are representable in the category of $L_\infty$ spaces.  This allows to talk about the derived space of maps from an elliptic curve to a complex manifold; this derived space of maps is the space of fields of our theory. 

From the point of view of field theory, this approach to derived geometry has the great advantage that it allows us to write a $\sigma$-model (a field theory based on maps) as a ``gauge theory'', where the fields are sections of some bundle of Lie algebras on the space-time manifold.  The perturbative renormalization techniques of \cite{Cos11} are well adapted to working with gauge theories.  Hopefully this point of view will be useful for treating other $\sigma$-models.  

In section \ref{section_volume_form} we develop the concept of projective volume form on $L_\infty$ spaces, and show that, under suitable hypothesis, one can integrate against a projective volume form.  

Part 2 focuses on holomorphic Chern-Simons theory and the proofs of the main results.    We start in section \ref{section_hcs}  by introducing the classical fields of the holomorphic Chern-Simons theory as the derived space of maps from an elliptic curve $E$ to $T^\ast X$, where $X$ is a complex manifold. 

Section \ref{section_rgflow} contains a description of the Wilsonian renormalization group flow, which is a key part of the approach to quantum field theory developed in \cite{Cos11}.  Section \ref{section_symmetries} describes some symmetries of classical holomorphic Chern-Simons.  These symmetries will constrain the possible quantizations of the theory.

In section \ref{section_main_theorem} we can finally give a precise statement of the main theorem, describing the quantizations of holomorphic Chern-Simons theory compatible with the symmetries at the classical level.   

I then proceed to prove the main theorem.  Section \ref{section_counterterms} analyzes the counter-terms that appear in quantizing the theory (it turns out that the theory is finite, so all counter-terms vanish).  Sections \ref{section_obstruction} and \ref{section_obstruction_class} analyze the cohomological obstructions to quantization, and show that the obstruction is precisely the second Chern character of $X$.

Finally, in section \ref{section_witten_genus}, I compute the scale $\infty$ effective interaction $I[\infty]$ of holomorphic Chern-Simons theory, and show that it can be identified with the Witten class of $X$ evaluated at the elliptic curve $E$.

\section*{Acknowledgments}
Conversations with many people have contributed to this paper.  I'm particularly grateful to Dan Berwick-Evans, Damien Calaque, Ryan Grady, Owen Gwilliam, Si Li, Nick Rozenblyum, Josh Shadlen and Yuan Shen, for helpful correspondence and conversations about this material and for their close reading of earlier drafts.  

%% file: b_a_derived.txt
\section*{\Large Part I: Derived geometry and $L_\infty$ spaces}

In this section, we will introduce some ideas from derived geometry which we will use throughout the paper.   I will develop only the minimum amount of theory required for the application. 

In derived algebraic geometry \cite{Toe06,Lur09b}, one defines the notion of derived scheme using the functor of points.    To give a derived stack over a field of characteristic zero is to give a functor from the category of commutative dgas (concentrated in degrees $\le 0$) to the category of simplicial sets; satisfying appropriate sheaf conditions.  

We will adopt this idea to our setting.  For us, a \emph{derived space} will be a functor from a category of manifolds equipped with a sheaf of dg rings, to the category of simplicial sets, satisfying a descent condition.  

In deformation theory \cite{Hin01,Lur10} formal derived spaces can be represented by dg Lie algebras (or $L_\infty$ algebras).    The derived spaces of interest in this paper will be represented by what I call \emph{$L_\infty$ spaces}; an $L_\infty$ space will be a manifold equipped with a sheaf of curved $L_\infty$ algebras.  

\subsection{}
The reason we need to use the language of derived geometry is that the  space of solutions to the equations of motion of the field theory we consider can be interpreted as the derived version of the space of maps from an elliptic curve $E$ to a complex manifold $X$.   

We would like to study this field theory using the approach to renormalization developed in \cite{Cos11}.   There, however, the spaces of fields are always assumed to be the sections of some vector bundle on the space-time manifold $\Sigma$. 

If we restrict our attention to those maps $f : \Sigma \to X$ which are in the formal neighbourhood of a constant map with value $x \in X$, then the techniques of \cite{Cos11} apply.  In that case, we can linearize $X$ near $x$, and then our space of fields becomes the sections of a trivial vector bundle on $X$.    We then have to ensure that any quantization is independent of the choice of linearization, but the homological techniques of \cite{Cos11} allow one to analyze this question.

Some more work is needed, however, if we want to consider fields $f : \Sigma \to X$ which are near \emph{some} constant map.  The language of $L_\infty$ spaces allows us to solve this problem, by representing the space of maps $\Sigma \to X$ which are near some constant map as the space of Maurer-Cartan elements in a sheaf of $L_\infty$ algebras on $X$.  

The techniques developed here allow one to study a wide class of field theories where the fields are spaces of maps to some non-linear target.  Although I emphasized above the problem of perturbing around constant maps, the same techniques allow one to analyze (in principle) the contributions of non-constant maps.

\section{Differential graded ringed manifolds}
On a manifold $M$, let $\Omega^\ast_M$ denote the de Rham complex of $M$, viewed as a sheaf of commutative dgas on $M$. 
\begin{definition}
A dg ringed manifold (over $\R$) is a manifold $M$, together with a sheaf $\A$ of differential graded unital $\Omega^\ast_M$-algebras, with the following properties.
\begin{enumerate}
\item As a sheaf of graded $\Omega^0_M$-algebras, $\A$ is locally free of finite total rank. 
\item  $\A$ is equipped with a map of sheaves of $\Omega^\ast_M$-algebras $\A \to \cinfty_M$; the kernel of this map must be a sheaf of nilpotent ideals
\item For sufficiently small open subsets $U$ of $M$, the cohomology of $\A(U)$ must be concentrated in non-positive degrees.   
\end{enumerate}
\end{definition} 
If we work over $\C$, we should use the complexified de Rham algebra $\Omega^\ast_M \otimes_{\R} \C$, but otherwise the definition is the same.    

Note that the axioms imply that the graded $\Omega^0_M$-module $\A$ is given by the sections of a graded vector bundle of finite total rank on $M$. 

Here are some examples of dg ringed manifold.
\begin{enumerate}
\item Let $M$ be any manifold.  Then letting $\A = \Omega^\ast_M$, equipped with the de Rham differential, gives a dg ringed manifold which we refer to as $M_{dR}$.  
\item Setting $\A = \cinfty_M$ gives a dg ringed manifold which we just call $M$. 
\item Let $M$ be a complex manifold.  Then there is a complex dg ringed space $M_{\dbar}$ with $\A = \Omega^{0,\ast}(M)$, where the differential is the operator $\dbar$. 
\item Let $M$ be a complex manifold, and let $R$ be 
any finite rank graded commutative algebra in the category of holomorphic bundles on $M$, concentrated in degrees $\le 0$, and equipped with a bundle $I \subset R$ of graded ideals such that $R / I = \underline{\C}$.  Then, $\Omega^{0,\ast}(M,R)$ defines a dg ringed manifold. 
\end{enumerate}

\begin{definition}
A map of dg ringed manifolds $(M,\A) \to (N, \mscr{B})$ is a smooth map $f : M \to N$, together with a map of sheaves of dg $f^{-1} \Omega^\ast_N$-algebras $f^{-1} \mscr{B} \to \mscr{A}$, such that the diagram
$$
\xymatrix{ f^{-1} \mscr{B}\ar[r] \ar[d]  & \mscr{A} \ar[d] \\ 
f^{-1} \cinfty_N \ar[r] & \cinfty_M } 
$$
commutes.  
\end{definition}
Here, $f^{-1}$ refers to the inverse image of a sheaf: so that if $F$ is a sheaf on $N$, then 
$$(f^{-1}F)(U) = \colim_{V \subset f(U)} F(V).$$ 

If $(M,\A)$ is a dg ringed manifold, then the sheaf $\A$ acquires a finite filtration by powers of the nilpotent dg ideal $\mscr{I} \subset \A$, which is the kernel of the map $\A \to \cinfty_M$.  We will let $\op{Gr} \A$ denote the associated graded sheaf; note that $\op{Gr} \A$ is a sheaf of dg algebras over the graded ring $\Omega^\sharp_M$, consisting of the de Rham algebra of $M$ with zero differential. 
\begin{definition}
A map $(M,\A) \to (N,\mscr{B})$ of dg ringed manifolds is an equivalence, if the map of smooth manifolds $M \to N$ is a diffeomorphism, and the map of sheaves
$$
\op{Gr} \A \to \op{Gr} \mscr{B} 
$$
is a quasi-isomorphism. 
\end{definition}

Let $\op{sSets}$ be the category of simplicial sets. 
\begin{definition}
A \emph{derived space} is a functor $\Phi$ from the category of dg ringed manifolds to the category of simplicial sets, which satisfies the following two properties.
\begin{enumerate}
\item $\Phi$ takes equivalences of dg ringed manifolds to weak equivalences of simplicial sets.
\item $\Phi$ satisfies a sheaf property, as follows.  Let $(M,\A)$ be any derived space; then, assigning to an open subset $U \subset M$ the simplicial set $\Phi(U,\A)$ defines a simplicial presheaf on $M$, which we call $\Phi\mid_M$.   We require that this simplicial presheaf satisfy \v{C}ech descent: for every open cover $\mf{U}$ of $M$, the natural map of simplicial sets
$$
\Phi ( M,\A) \to \check{C}(\mf{U}, \Phi\mid_M) 
$$
is a weak homotopy equivalence.
\end{enumerate}
\end{definition}

%% file: c_Geometry_using_lie_algebras.txt
\section{Derived geometry with curved $L_\infty$ algebras}
\label{section_geometry}

The main theorem of deformation theory asserts that every formal derived space can be represented by an $L_\infty$ algebra. Let us recall briefly how this works.  
\begin{definition}
An Artinian dg algebra over $k$ is a dg algebra $R$ with a nilpotent differential ideal $m$, such that $R /  m  = k$,  such that $R$ is concentrated in degrees $\le 0$ of finite total dimension over $k$. 
\end{definition}
Note that an Artinian dg algebras over $\R$ are (essentially) the same as dg ringed manifolds over $\R$ where the underlying manifold is a point.
\begin{definition}
A formal moduli problem is a functor $F$ from the category of nilpotent dg algebras over $k$ to the category of simplicial sets, such that $F(k)$ is contractible, and $F$ preserves certain homotopy limits.  
\end{definition}
See \cite{Lur10} for more details.

We want to explain briefly how every $L_\infty$ algebra $\g$ gives rise to a formal moduli problem.  Let us work over $\R$ for simplicity.   Let $R$ be an Artinian dg algebra with maximal ideal $m$.  The formal moduli problem associated to $\g$ assigns to $R$ the simplicial set $\op{MC}(\g \otimes m)$ of Maurer-Cartan elements of the nilpotent $L_\infty$ algebra $\g \otimes m$.  An $n$-simplex of this simplicial set is a Maurer-Cartan element in $\g \otimes m \otimes \Omega^\ast(\tr^n)$, where $\Omega^\ast(\tr^n)$ refers to the commutative dg algebra of differential forms on the $n$-simplex. 

\subsection{}
We will introduce a global version of the Maurer-Cartan functor associated to an $L_\infty$ algebra.  This construction will associate a derived space to a manifold $X$ equipped with a certain sheaf of $L_\infty$ algebras.

We will start by giving a general definition of curved $L_\infty$ algebra.   Let $A$ be a differential graded commutative algebra, and let $I \subset A$ be a nilpotent ideal in $A$. 

We will let $A^\sharp$ denote the underlying graded algebra, with zero differential. 
\begin{definition}
A \emph{curved $L_\infty$ algebra over $A$} consists of a locally free finitely generated graded $A^\sharp$-module $V$,  together with a derivation
$$\d :\what{\Sym}^\ast( V [1] ^\vee ) \to \what{\Sym}^\ast ( V [1] ^\vee )$$
of cohomological degree $1$ and square zero. In this expression, all tensors and duals are over the graded algebra $A^\sharp$.

The derivation $\d$ must make the completed symmetric algebra $\what{\Sym}^\ast( V [1] ^\vee )$ into a differential graded algebra over the differential graded algebra $A$.  

Further, when we reduce modulo the nilpotent ideal $I$, the derivation $\d$ must preserve the ideal in $\what{\Sym}^\ast ( V[1]^\vee)$ generated by $V$.
\end{definition}

The Taylor components of the derivation $\d$ are maps
$$
l_k : \wedge^k ( V) \to V
$$
of cohomological degree $2-k$, satisfying a version of the standard $L_\infty$ identities which also incorporates the differential on $A$.  The first operator $l_0$ defines an element of $V$; our axioms imply that the operator $l_0$ lies in the subspace $V \otimes_{A} I$.

Note that if $l_0$ is not zero, then $V$ will not have the structure of a differential graded module over $A$.   However,  the fact that we have an ordinary $L_\infty$ structure when we reduce modulo $I$ implies that $V / I $ is a differential module over the dg algebra $A / I$. 

If $V$ is a curved $L_\infty$ algebra over $A$, we will let $C^\ast(V)$ denote the differential graded $A$-algebras $\what{\Sym}^\ast(V^\vee[1])$, equipped with the differential which appears in the definition of the curved $L_\infty$ structure.

\begin{definition}
Let $X$ be a manifold.   A curved $L_\infty$ algebra over $\Omega^\ast_X$ is a sheaf $\g$ of graded $\Omega^\#_X$ modules on $X$, which is locally free of finite total rank, equipped with the structure of curved $L_\infty$ algebra over $\Omega^\ast_X$, as described earlier; where the nilpotent ideal is $\Omega^{> 0}_X$.  

We let $\g_{red} = \g/ \Omega^{>0}_X$.  

An \emph{$L_\infty$ space} is a manifold $X$ equipped with a curved $L_\infty$ algebra $\g$ over the sheaf $\Omega^\ast_X$. 
\end{definition}

We will think of an $L_\infty$ space as a kind of ``derived space''.  Note that if $(X,\g)$ is an $L_\infty$ space, then $C^\ast(\g)$ (where cochains are taken over $\Omega^\ast_X$) is a sheaf of pro-nilpotent differential graded algebras over $\Omega^\ast_X$.   Let $I \subset C^\ast(\g)$ denote the ideal generated by $\g^\vee$ and by $\Omega^1_X$.  Then, for each $k$, $C^\ast(\g) / I^k$ defines a dg ringed manifold in the sense defined above.  Thus, we should think of $(X,C^\ast(\g))$ as an inverse limit of dg ringed manifolds.

\subsection{}
If $(X,\g)$ is an $L_\infty$ space, $Y$ is a manifold, and $\phi : Y \to X$ is a smooth map, then we can form a curved $L_\infty$ algebra $\phi^\ast \g$ over $\Omega^\ast_Y$, defined by
$$
\phi^\ast \g = \phi^{-1} \g \otimes_{\phi^{-1}\Omega^\ast_X} \Omega^\ast_Y. 
$$
(Here $\phi^{-1} \g$ refers to the sheaf pull back). 

\begin{definition}
Let $(X,\g)$ be an $L_\infty$ space.  Let us define a functor $\op{MC}_{ (X,\g)}$ from dg ringed manifolds to simplicial sets, by saying that  $\op{MC}_{(X,\g)}(M, \A)$ is the simplicial set consisting of smooth maps $f : M \to X$, together with a Maurer-Cartan element
$$
\alpha \in f^\ast \g \otimes_{\Omega^\ast_M} \A
$$
which vanishes modulo the ideal $\mscr{I} \subset \A$. 
\end{definition}
Recall that the ideal $\mscr{I}$ is the kernel of the map of sheaves of algebras $\A \to \cinfty_M$.  

To give a Maurer-Cartan element as above is the same as to give a map of sheaves of pro-$\Omega^\ast_M$-algebras
$$
C^\ast( f^\ast \g) \to \A
$$
such that the diagram
$$
\xymatrix{C^\ast(f^\ast\g) \ar[r] \ar[d] & \A \ar[dl] \\
\cinfty_M & 
}
$$
commutes.
\begin{theorem}
The functor $\op{MC}_{(X,\g)}$ associated to an $L_\infty$ space $(X,\g)$ defines a derived space: that is, it takes equivalences of dg ringed manifolds to weak equivalences of simplicial sets, and it satisfies the \v{C}ech descent property.  
\end{theorem}
The proof of this theorem is provided in the appendix. 

\section{Complex manifolds as $L_\infty$ spaces}

Let $X$ be a complex manifold.  In this section we will construct a curved $L_\infty$ algebra over $\Omega^\ast_X$, the $\cinfty$ de Rham complex of $X$.  This curved $L_\infty$ algebra will encode the holomorphic geometry of $X$, and is well-defined up to contractible choice. 

Let $\mscr{J}^{hol}$ denote the infinite-rank vector bundle on $X$ whose fibre at $x \in X$ is the space of infinite jets of holomorphic functions at $x$.  Although $\mscr{J}^{hol}$ has a natural structure of an infinite-rank holomorphic vector bundle, we will only consider $\mscr{J}^{hol}$ has a $\cinfty$ vector bundle.  There is a natural flat connection on $\mscr{J}^{hol}$, and a flat section of $\mscr{J}^{hol}$ over an open subset $U \subset M$ is precisely a holomorphic function on $U$. 

If $p \in X$ and $z_1,\ldots, z_n$ are holomorphic coordinates at $p$, then we can identify the fibre of $\mscr{J}^{hol}$ at $p$ with the algebra $\C[[z_1,\dots,z_n]]$ of formal power series in the variables $z_i$.  Now suppose that $U \subset X$ is an open subset, and $z_1,\dots,z_n$ are holomorphic coordinates on $U$.  For each $p \in U$, the functions $z_i - z_i(p)$ define holomorphic coordinates centered at $p$. Let $y_i$ be the section of $\mscr{J}^{hol}$ whose value at $p$ is the jet of $z_i - z_i(p)$. We thus find an identification between the space $\Gamma(U, \mscr{J}^{hol})$ of smooth sections of $\mscr{J}^{hol}$ on $U$ with the space $\cinfty(U,\C)[[y_1,\ldots,y_n]]$ of formal power series in $n$ variables, with coefficients in the algebra $\cinfty(U,\C)$ of complex-valued smooth functions on $U$.   In these coordinates, the flat connection
$$
\nabla:  \Gamma(U, \mscr{J}^{hol}) \to \Omega^1(U,\mscr{J}^{hol})
$$
takes the form
$$
\nabla = \sum \d \zbar_i \dpa{\zbar_i} + \d z_i \dpa{z_i} - \d z_i \dpa{y_i}.
$$
From this expression, it is clear that section of $\mscr{J}^{hol}$ is flat if and only if it is the jet of a holomorphic function on $U$.  

Since $\mscr{J}^{hol}$ has a flat connection, we can define the de Rham algebra 
$$\Omega^\ast_X(\mscr{J}^{hol}) = \Omega^\ast_X \otimes_{\cinfty_X} \mscr{J}^{hol}$$
with coefficients in $\mscr{J}^{hol}$.  I should emphasize that $\Omega^\ast_X$ denotes the $\cinfty$ de Rham complex, viewed as a sheaf of differential graded algebras on $X$.  

Since the natural algebra structure on $\mscr{J}^{hol}$ is compatible with the flat connection, $\Omega^\ast_X(\mscr{J}^{hol})$ is a differential graded algebra over $\Omega^\ast_X$. If $\Oo_X^{hol}$ denotes the sheaf of holomorphic functions on $X$, there is a natural quasi-isomorphism (of sheaves of dg algebras on $X$)
$$
\Oo_X^{hol} \simeq \Omega^\ast_X( \mscr{J}^{hol} ) .
$$

\begin{lemma}
Then there is a canonical, up to contractible choice,  curved $L_\infty$ algebra $\g_{X_{\dbar}}$ over $\Omega^\ast_X$ (with the nilpotent ideal $\Omega^{> 0}_X$), with the following properties. 
\begin{enumerate}
\item As an $\Omega^\sharp_X$-module, $\g_{X_{\dbar}}$ is isomorphic to $T^{1,0}_X [-1] \otimes_{\cinfty_X} \Omega^\sharp_X$ (where $T^{1,0}_X$ denotes the holomorphic tangent bundle of $X$).
\item There is an isomorphism
$$
C^\ast(\mf g_X ) \simeq \mscr{J}^{hol}
$$
of differential graded $\Omega^\ast_X$-algebras. 
\end{enumerate}
\end{lemma}
\begin{proof} 
There is a natural decreasing filtration on $\mscr{J}^{hol}$ by subbundles, where $F^k \mscr{J}^{hol}$ is the subbundle whose fibre at $x \in X$ is the space of jets of holomorphic functions at $x$ which vanish to order $k$.  These subbundles are not preserved by the flat connection: rather, a kind of Griffiths transversality condition holds.  The connection gives a map
$$
F^k \mscr{J}^{hol} \to  F^{k-1} \mscr{J}^{hol} \otimes \Omega^1_X.
$$
Further, one can identify $F^1\mscr{J}^{hol} / F^2 \mscr{J}^{hol}$ with $(T_X^{1,0})^\vee$. 

Let us choose a splitting of the map 
$$
F^1 \mscr{J}^{hol} \to (T_X^{1,0})^\vee
$$
of $\cinfty$ vector bundles.   This leads to an isomorphism of $\Omega^\sharp_X$ modules
$$
\Omega^\sharp_X (\mscr{J}^{hol}) \iso  \what{\Sym}^\ast (T^{1,0}_X)^\vee) \otimes_{\cinfty_X}  \Omega^\sharp_X.
$$
Since the left hand side is a differential graded algebra over $\Omega^\ast_X$, this isomorphism leads to a curved $L_\infty$ algebra over $\Omega^\ast_X$, which is easily seen to have all the desired properties.  

Next, we need to verify that the resulting curved $L_\infty$ structure on $T^{1,0}_X[-1]$ is independent, up to contractible choice, of the splitting of the bundle map $F^1 \mscr{J}^{hol} \to (T^{1,0})^\ast X$.  The space of such splittings is contractible.  Thus, it suffices to verify that if we have a family of such splittings, parameterized by an $n$-simplex $\tr^n$, then we get a family of curved $L_\infty$ structures on $T^{1,0}_X[-1]$ over $\Omega^\ast(\tr^n)$. 

To give such a family of curved $L_\infty$ structure is to give a differential on the completed symmetric algebra 
$$\what{\Sym}^\ast (T^{1,0}_X )^\vee \otimes_{\cinfty_X} \Omega^\sharp_X \otimes_{\C} \Omega^\ast(\tr^n),$$
making it into a sheaf on $X$ of differential graded algebras over $\Omega^\ast_X \otimes \Omega^\ast(\tr^n)$.  As above, the differential must preserve the ideal generated by $\what{\Sym}^{>0} (T^{1,0}_X )^\vee$ and by $\Omega^{> 0}_X$.   

The choice of our splitting gives an isomorphism of sheaves of $\cinfty(\tr^n)$-algebras
$$\what{\Sym}^\ast (T^{1,0}_X )^\vee \otimes_{\C} \cinfty(\tr^n) \iso \mscr{J}^{hol}_X \otimes_{\C} \cinfty(\tr^n) .
$$ 
This isomorphism can be extended, by linearity, to an isomorphism of graded algebras
$$\what{\Sym}^\ast (T^{1,0}_X )^\vee \otimes_{\cinfty_X} \Omega^\sharp_X \otimes_{\C} \Omega^\ast(\tr^n) \iso \mscr{J}^{hol}_X \otimes_{\cinfty_X} \Omega^\sharp_X \otimes_{\C} \Omega^\ast(\tr^n).$$
The right hand side of this equation has a differential coming from the flat connection on $\mscr{J}^{hol}_X$, and this gives the desired family of curved $L_\infty$ structures. 

\end{proof}

The curved $L_\infty$ algebra $\g_{X_{\dbar}}$ -- or rather its restriction to a flat $L_\infty$ algebra over the sheaf of holomorphic functions on $X$ --  was discussed by Kapranov in \cite{Kap97}.  Of course, from a formal point of view \cite{Qui69, Hin01, Lur10} the existence of a Lie algebra structure on $TX[-1]$ is no surprise: it is just defined to be the Koszul dual of the holomorphic bundle of complete augmented commutative algebras given by $\mscr{J}^{hol}$.  

This way of encoding a complex manifold by an $L_\infty$ space is a version of formal geometry.  A general approach to formal geometry, related to the approach used here, was developed in \cite{CalBer10}. 

\subsection{}
Now that we have defined the $L_\infty$ space $(X,\g_{X_{\dbar}})$ associated to a complex manifold $X$, we need to verify that the associated Maurer-Cartan functor represents the problem of holomorphic maps into $X$.
\begin{lemma}
Let $M$ and $X$ be complex manifolds, and let $(X,\g_{X})$ be the $L_\infty$ space encoding the complex structure on $X$. Then, 
\begin{enumerate}
\item The simplicial set $\op{MC}_{(X,\g_{X})}(M,\Omega^{0,\ast}_M)$ is discrete, that is, all higher simplices are constant.
\item Zero simplices of $\op{MC}_{(X,\g_{X})}(M,\Omega^{0,\ast}_M)$ are in 
bijection with holomorphic maps $M \to X$.
\end{enumerate}
\end{lemma}
\begin{proof}
Let $\phi : M \to X$ be a smooth map.  We are interested in Maurer-Cartan elements in the nilpotent curved $L_\infty$ algebra $\Omega^{0,>0}_M \otimes_{\Omega^\ast_M} \phi^\ast \g_{X_{\dbar}}$.  Note that $\phi^\ast \g_{X_{\dbar}} = \Omega^\sharp_M \otimes \phi^\ast T^{1,0}_X[-1]$, as a graded module over $\Omega^\sharp_M$ (which denotes the graded algebra of forms on $M$ with no differential). 

Thus, the $L_\infty$ algebra of interest is concentrated in cohomological degrees $\ge 2$.  It follows immediately that if there is a Maurer-Cartan element, it is unique, and that the simplicial set of Maurer-Cartan elements is discrete.  The existence of the Maurer-Cartan element is equivalent to the vanishing of the curving $l_0 \in \Omega^{0,2}_M \otimes T^{1,0}_X$. 

It remains to show that the curving vanishes if and only if $\phi$ is holomorphic.  Firstly, suppose that $l_0$ vanishes.  Then, the Maurer-Cartan element gives us a map of differential graded $\Omega^\ast_M$-algebras 
$$C^\ast( \phi^\ast \g_{X_{\dbar}}) \to \Omega^{0,\ast}_M,$$
such that the diagram
$$
\xymatrix{ C^\ast(\phi^\ast \g_{X_{\dbar}})  \ar[r] \ar[d]  & \Omega^{0,\ast}_M \ar[dl] \\ 
\cinfty_M & } 
$$
commutes.  

Since the $\Omega^\ast_X$-module $C^\ast (\phi^\ast \g_{X_{\dbar}})$ is the de Rham complex of $X$ with coefficients in jets of holomorphic functions, this commutative diagram implies that the pull back of a holomorphic function on $X$ is a holomorphic function on $M$, so that $\phi$ is a holomorphic map.

Conversely, suppose that $\phi$ is a holomorphic map.  Then, $\phi$ induces a map of $\phi^{-1} \Omega^{\ast}_X$-algebras
$$
\phi^{-1} \Omega^{0,\ast}_X \to \Omega^{0,\ast}_M.
$$
Now, there is a natural map of $\Omega^\ast_X$-algebras
$$
\Omega^\ast_X ( J (\Oo_X)) \to \Omega^{0,\ast}_X,
$$
where $\Omega^\ast_X (J(\Oo_X))$ indicates the $\cinfty$ de Rham complex of $X$ with coefficients in the $\cinfty$ bundle of jets of holomorphic functions on $X$.  It follows that a holomorphic map $M \to X$ induces a map of $\Omega^\ast_M$-algebras $C^\ast(\phi^\ast \g_{X_{\dbar}}) \to \Omega^{0,\ast}_M$, as desired. 
\end{proof}
 
\subsection{}
We can define curved $L_\infty$ algebras encoding other kinds of geometry on $X$.  For example, it is straightforward to modify the above definition to produce a curved $L_\infty$ algebra which encodes the $\cinfty$ geometry of a smooth manifold $X$. 

\section{Geometric constructions with curved $L_\infty$ algebras}
If $(X,\g)$ is an $L_\infty$ space, we will let $B \g$ denote the ringed space $(X, C^\ast(\g))$.  As always, the Chevalley-Eilenberg cochain complex $C^\ast(\g)$ is defined over the de Rham algebra $\Omega^\ast_X$.   We will often write $\Oo_{B \g}$ to denote the structure sheaf of $B \g$, that is, the sheaf $C^\ast (\g)$ of rings on $X$.  Note that $\Oo_{B \g}$ is a sheaf of commutative dg algebras over $\Omega^\ast_X$.     We will sometimes also use the notation $\Oo(X,\g)$ to denote the sheaf $C^\ast(\g)$.

\subsection{}
According to the standard dictionary between commutative algebras and dg Lie algebras, modules over commutative algebras correspond to modules over dg Lie algebras.    This suggests the following definition.  If $\g$ is an ordinary dg Lie algebra, then a dg module over $\g$ is the same thing as a split square zero extension of $\g$. 
\begin{definition}
A vector bundle $V$ on an $L_\infty$ space $(X,\g)$ is a locally free sheaf of $\Omega^\sharp_X$-modules on $X$, such that $V \oplus \g$ has the structure of curved $L_\infty$ algebra over $\Omega^\ast_X$, with the following properties:
\begin{enumerate}
\item The maps $\g \into \g \oplus V$ and $\g \oplus V \to \g$ are maps of $L_\infty$ algebras.
\item  Any higher product $l_n$ involving two or more sections of $V$ is zero. 
\end{enumerate}
\end{definition}
The $L_\infty$ space $(X , \g \oplus V)$ should be thought of as the total space of the vector bundle $V[1]$, formally completed along the zero section.  

In usual geometry, there are two equivalent languages for discussing vector bundles: we can think of a vector bundle in terms of its total space, or we can think of it in terms of its sheaf of sections.  
\begin{definition}
If $V$ is a vector bundle on $(X,\g)$ let $C^\ast(\g,V)$ be the sheaf of dg modules over $C^\ast(\g)$.  We call $C^\ast(\g,V)$ the sheaf of sections of $V$. 
\end{definition}

Familiar geometric constructions, such as the tangent bundle and the theory of differential forms, make sense on an $L_\infty$ space $(X,\g)$.  For example, the tangent bundle $T (X,\g)$ is given by the module $T(X,\g) = \g[1]$, with its natural structure of module over $\g$.    Thus, sections of the tangent bundle -- that is, vector fields -- are given by the sheaf $C^\ast(\g, \g[1])$. 

 The cotangent bundle is $T^\ast(X,\g)$ is then defined to be  the dual module $\g^\ast[-1]$ to the tangent bundle $\g[1]$.  

Similarly, the exterior powers of the cotangent bundle of  $(X,\g)$ are defined by
$$\wedge^k T^\ast (X,\g) = \wedge^k (\g^\vee[-1] ) =  \Sym^k (\g^\vee) [-k] ).$$
Thus, a $k$-form on $(X,\g)$ is a section of the sheaf $C^\ast(\g, \Sym^k (\g^\vee ) [-k] )$. 

\subsection{}
Recall that the total space of a vector bundle $V$ on $(X,\g)$ is the $L_\infty$ space $(X,\g \oplus V)$.  For example, the total space of the cotangent bundle to $(X,\g)$ can be described as 
$$
T^\ast (X ,\g  ) = (X,  \g \oplus \g^\vee [-1] ) .
$$
Thus, the algebra of functions on $T^\ast (X,\g)$ can be written as 
$$
\Oo(T^\ast (X,\g) ) = C^\ast ( \g, \what{\Sym}^\ast ( \g[1] ) ),
$$
where the completed symmetric algebra $\what{\Sym}^\ast ( \g[1] ) $ is endowed with its natural $\g$ action.  

For example, if $\g = \g_{X_{\dbar}}$ is the curved $L_\infty$ algebra arising from a complex structure on $X$, then there is a weak equivalence of sheaves of dgas on $X$ between $\Oo_{T^\ast B \g_{X_{\dbar}}}$ and the formal completion along the zero section $X \into T^\ast X$ of the sheaf of holomorphic functions on $T^\ast X$.

%% file: c_z_mapping_spaces.txt
\section{Derived mapping spaces}
\label{section_mapping}
We have seen that that the Maurer-Cartan functor associated to an $L_\infty$ space is a derived space. If $(X,\g)$ is an $L_\infty$ space, we will view this derived space as representing the functor of maps to $(X,\g)$.   In this section we will show that, if $(M, \A)$ is a dg ringed manifold, a subset of the space of maps $(M,\A) \to (X,\g)$ is itself represented by an $L_\infty$ space. 

Let us define a functor $\op{MC}((M,\A); (X,\g))$, from the category of dg ringed manifolds to the category of simplicial sets, by saying that
$$
\op{MC}((M,\A); (X,\g)) (N,\mscr{B} )  = \op{MC}_{(X,\g) } ( (M \times N, \A \boxtimes B ) ).
$$
Recall that the Maurer-Cartan functor associated to $(X,\g)$ satisfies the axioms of a derived space: it takes equivalences of dg ringed manifolds to weak equivalences of simplicial sets, and satisfies \v{C}ech descent with respect to open covers of dg ringed manifolds.  It follows that the functor $\op{MC} ( (M, \A) ; (X, \g)) $ satisfies these same axioms. 

We will let
$$
\what{\op{MC}}((M,\A); (X,\g) ) \subset \op{MC} ((M,\A); (X,\g))
$$
be the subsheaf consisting of maps such that the map of underlying manifolds $M \to X$ is constant. More precisely, if $(N,\mscr{B})$ is an auxiliary dg ringed manifold, we set 
$$
\what{\op{MC}}((M,\A) ; (X,\g) ) ( N, \mscr{B} ) \subset \op{MC} _{(X,\g)} ( (M \times N,\A \otimes \mscr{B}) 
$$
to be the sub-simplicial set consisting of Maurer-Cartan elements where the underlying map $M \times N \to X$ of smooth manifolds factors through the projection $M \times N \to N$. 

\begin{proposition}
Let $(M, \A)$ be a dg ringed manifold with the property that, if $F^i \A(M)$ denotes the filtration on $\A(M)$ by the powers of the ideal $\mscr{I}(M)$, then the cohomology each $\op{Gr}^i \A(M)$ for $i \ge 1$ is concentrated in degrees $\ge 1$.  

Let $(X,\g)$ be an $L_\infty$ space such that the cohomology of the sheaf of $L_\infty$ algebras $\g^{red} = \g / \Omega^{> 0}_X$ is concentrated in degrees $\ge 1$. 

Then, the restricted Maurer-Cartan functor $\what{\op{MC}}((M,\A); (X,\g))$ is equivalent to the functor represented by the $L_\infty$ space $(X, \g \otimes \A(M))$, where $\A(M)$ is the global sections of the sheaf of commutative dgas $\A$ on $M$.
\end{proposition}
Note that, as always, when we are dealing with topological vector spaces such as $\A(M)$ we take the completed projective tensor product. 
\begin{proof}
Indeed, let $(N, \mscr{B})$ be another dg ringed manifold.  An $n$-simplex of  
$$\what{\op{MC}}((M,\A); (X,\g))(N, \mscr{B})$$ 
is given by a smooth map $\phi : N \to X$ together with a Maurer-Cartan element in the sheaf of curved $L_\infty$ algebras over $\mscr{B}$,
$$\phi^\ast \g \otimes_{\Omega^\ast_N} \mscr{B} \otimes_{\R} \A(M) \otimes_\R \Omega^\ast(\tr^n),$$
where the Maurer-Cartan element vanishes modulo the ideal 
$$\mscr{I}_1[n] = \left( \mscr{I}_\mscr{B} \otimes \A(M)  \otimes \Omega^\ast(\tr^n) \right) + \left( \mscr{B} \otimes \mscr{I}(M) \otimes \Omega^\ast(\tr^n) \right).$$  

We will let $\mscr{C}[n]$ denote the sheaf of $\Omega^\ast_N$-algebras
$$
\mscr{C}[n]=\mscr{B} \otimes_{\R} \A(M) \otimes_\R \Omega^\ast(\tr^n).
$$ 
Thus, $\mscr{I}_1[n]$ is an ideal in $\mscr{C}[n]$.

We will let $\mscr{C}$ and $\mscr{I}_1$ denote $\mscr{C}[0]$ and $\mscr{I}_1[0]$. 

An $n$-simplex of $\op{MC}_{(X, \g \otimes \A(M) )} (N,\mscr{B} ) $ is given by a smooth map $\phi : N \to X$ and a Maurer-Cartan element of 
$$\phi^\ast \g \otimes_{\Omega^\ast_N} \mscr{C}[n]$$
as above; except now we require that it vanishes modulo the ideal 
$$
\mscr{I}_2[n] = \mscr{I}_\mscr{B} \otimes \A(M)  \otimes_{\R} \Omega^\ast(\tr^n).
$$
As above, we will let $\mscr{I}_2$ denote $\mscr{I}_2[0]$. 

Thus, the two simplicial sets are almost identical, except that in the second simplicial set we require a stronger vanishing condition.  (The condition is stronger because $\mscr{I}_1[n] \subset \mscr{I}_2[n]$). 

It follows immediately that there is a natural transformation of functors
$$
\op{MC}_{(X, \g \otimes \mscr{A}(M) )} \to \what{\op{MC}} ( (M,\A); (X,\g ) ).  
$$
It remains to verify that this natural transformation yields a weak equivalence of simplicial sets when evaluated on all $(N,\mscr{B})$.  

Note that both simplicial sets decompose as a disjoint union over the set of all smooth maps $\phi : N \to X$.  Thus, we will fix such a $\phi$ and analyze the components of both simplicial sets corresponding to $\phi$. 

Given $\phi$, the corresponding component of the first (respectively, second) simplicial set is the Maurer-Cartan simplicial set associated to the nilpotent curved $L_\infty$ algebra 
$$
\g_{i,\phi} = \Gamma(N, \phi^\ast \g \otimes_{\Omega^\ast_N} \mscr{I}_i )
$$
where $i=1,2$. 

Both of these nilpotent $L_\infty$ algebras are equipped with finite bifiltrations, induced by the filtrations on $\A(M)$ and $\mscr{B}(N)$ by the powers of the ideals $\mscr{I}_\A(M)$ and $\mscr{I}_{\mscr{B}}(N)$.  The map $\g_{1,\phi} \to \g_{2,\phi}$ is filtration preserving, and the associated graded is Abelian.  

The associated graded of $\g_{i,\phi}$ is $\phi^\ast \g \otimes_{\Omega^\ast_N} \op{Gr} \mscr{I}_i$.  Note that 
$$
\op{Gr}^{i,j} (\mscr{I}_1) = 
\begin{cases}
\op{Gr}^i(\mscr{A}) \otimes \op{Gr}^j (\mscr{B}) & \text{ if } i \ge 0 \text{ or } j \ge 0 \\
0 & \text{ if } i = j = 0.
\end{cases}
$$
Similarly,
$$
\op{Gr}^{i,j} (\mscr{I}_2) = 
\begin{cases}
\op{Gr}^i(\mscr{A}) \otimes \op{Gr}^j (\mscr{B}) & \text{ if } j >0  \\
0 & \text{ if } j = 0.
\end{cases}.
$$
It follows from this observation, and the assumptions in the statement of the proposition, that the induced map on the cohomology of the associated graded 
$$
H^\ast (\op{Gr}^{\ast,\ast}( \g_1) ) \to 
H^\ast (\op{Gr}^{\ast,\ast}( \g_2) )  
$$
is an isomorphism on $H^{\le 1}$.  It follows (by a standard argument using induction on the filtration) that the induced map of Maurer-Cartan simplicial sets is a weak equivalence. 
\end{proof}

\subsection{}
In our study of holomorphic Chern-Simons theory, we are interested in the space of holomorphic maps $E \to X$, where $E$ is a Riemann surface and $X$ is a complex manifold.  We are only interested in those maps which are infinitesimally near to the constant map.    This proposition shows that this mapping space is represented by the $L_\infty$ space $\Omega^{0,\ast}(E) \otimes \g_{X_{\dbar}}$ over $\Omega^\ast_X$. 

The main reason for developing the theory of $L_\infty$ spaces as much as I did is to be able to represent the mapping space in this way.  The quantum field theory techniques developed in \cite{Cos11} apply when our space of fields is linear.  This presentation of the mapping space allows us to apply the techniques of \cite{Cos11} directly to theories where the space of fields is some space of maps.  In the example of interest in this paper, the space of classical fields will be the sheaf of $\Omega^\ast_X$-modules
$$
\Omega^{0,\ast}(E) \otimes \left( \g_{X_{\dbar}} [1] \oplus \g^\vee_{X_{\dbar}} [-1] \right).
$$
The classical action is constructed from the curved $L_\infty$ structure on $\g_{X_{\dbar}}$, and has the property that the derived moduli space of solutions to the equations of motion is the same as the space solutions to the Maurer-Cartan equation in the curved $L_\infty$ algebra $\Omega^{0,\ast}(E) \otimes \left( \g_{X_{\dbar}}  \oplus \g^\vee_{X_{\dbar}} [-2] \right).$  The above proposition allows us to identify this space of Maurer-Cartan elements with the derived space of maps from the elliptic curve $E$ to the cotangent bundle $T^\ast X$ (completed near constant maps to $X$).

Recall that the particular $L_\infty$ algebra $\g_{X_{\dbar}}$ associated to the complex manifold $X$ depends on a choice: namely, the choice of a $\cinfty$ splitting of the vector bundle map
$$
F^1 \jhol \to T^{1,0} X. 
$$
However, we have seen that if we choose a different splitting, then we get a homotopy-equivalent $L_\infty$ algebra structure on $T^{1,0} X \otimes_{\cinfty_X} \Omega^\sharp_X$.  (By definition, a homotopy of $L_\infty$ structures on a graded vector space $V$ is a family of $L_\infty$ structures over the base dg ring $\Omega^\ast([0,1])$). 

Thus, the $L_\infty$ space $\Omega^{0,\ast}(E) \otimes \g_{X_{\dbar}}$ is well-defined up to homotopy.  Since the quantum field theory formalism developed in \cite{Cos11} works relative to an arbitrary dg base ring, we see that the field theory constructions we perform will we independent, up to homotopy, of the choice of splitting $T^{1,0} X \to F^1 \jhol$. 

This point illustrates a general philosophy in perturbative quantum field theory, as developed in \cite{Cos11}.  Although the Feynman diagram expansion of a field theory depends on a linear structure on the space of fields, by talking about homotopy equivalences of classical field theories one can access non-linear local isomorphisms of the space of fields, and so remove this dependence.

%% file: d_Atiyah_class.txt
\section{Curvature and characteristic classes}
\label{section_atiyah}
In this section,  I will describe how one can construct the Chern classes of a vector bundle on an $L_\infty$ space $(X,\g)$.  

\subsection{}
Let us recall the definition of the Atiyah class. Let $V$ be a holomorphic vector bundle on a complex manifold $Y$.  The Atiyah class is the element
$$
\alpha(V) \in H^1(Y, \Omega^1_Y \otimes_{\Oo_Y} \op{End}(V) )
$$
which is the obstruction to the existence of a holomorphic connection on $V$. Another way to phrase the definition is to say that $\alpha(V)$ classifies the $\Omega^1_Y \otimes_{\Oo_Y} \op{End}(V)$-torsor of holomorphic connections on $V$.

We can find a cochain representative of the Atiyah class as follows.
\begin{definition}
Let 
$$
\nabla_V : \Omega_X^{0,\ast}(V) \to \Omega_X^{1,\ast}(V)
$$
be a connection over $\Omega^{0,\ast}_X$. Thus, $\nabla_V$ must satisfy the Liebniz rule
$$
\nabla_V (\alpha v ) = (\partial \alpha) v + (-1)^{\abs{\alpha} } \alpha \nabla_V v,
$$
for all $\alpha \in \Omega^{0,\ast}_X$ and $v \in V$.  However, we do not assume that $\nabla_V$ is compatible with the $\dbar$ operator. 

Then, the Atiyah class of $\nabla_V$ is
$$
\alpha(\nabla_V) = [\dbar, \nabla_V] \in \Omega^{1,1}_X(\op{End}(V)).
$$
\end{definition}
It is easy to see that the cohomology class of $\alpha(\nabla_V)$ in $H^1(X, \Omega^1 \otimes \op{End}(V) )$ is the usual Atiyah class.

This definition of the Atiyah class of a connection generalizes immediately.  Let $R$ be a differential graded algebra, and let $R^\sharp$ be the underlying graded algebra.  Let $V$ be an $R$-module, which is projective as an $R^\sharp$-module.  Let $\Omega^1_R$ denote the $R$-module of K\"{a}hler differentials of $R$.
\begin{definition}
A \emph{connection} on $V$ is a map (of graded vector spaces)
$$
\nabla_V : V \to \Omega^1_R \otimes_R V
$$
satisfying the Leibniz rule
$$
\nabla_V (r v) = \d_{dR}(r) v + r (-1)^{\abs{r} } \nabla_V v.
$$
The \emph{Atiyah class} of $\nabla_V$ is the class
$$
\Atiyah(\nabla_V) = [\nabla_V, \d] \in \Omega^1_R \otimes_R \op{End}_R(V) 
$$
which measures the failure of $\nabla_V$ to be a cochain map.
\end{definition} 
 Note that $\Atiyah(\nabla_V)$ is a closed element of $\Omega^1_R \otimes_R \op{End}_R(V)$ of cohomological degree $1$.  If we change the connection $\nabla_V$ on $V$, then the Atiyah class $\Atiyah(\nabla_V)$ chains by an exact element. 

\subsection{}
We are interested in the Atiyah class of the tangent bundle to $(X,\g)$ for any $L_\infty$ space $X$. 

The tangent bundle to $(X,\g)$ is represented by the sheaf of $\g$-modules $\g[1]$.  In order to implement the algorithm above, we need to pass to the language of ringed spaces.  As before, the ringed space associated to the $L_\infty$ space $(X,\g)$ is denoted by $B \g$.  The sheaf of rings is $C^\ast(\g)$.    Recall that this is a sheaf of rings over the sheaf $\Omega^\ast_X$.

We will freely pass back and forth between a curved $L_\infty$ space $(X,\g)$ and the associated ringed space $B \g$.  Thus, the tangent bundle $T(X,\g)$ can be viewed as the $\g$-module $\g[1]$; in the language of ringed spaces, the tangent bundle $T_{B \g}$ is the sheaf of $C^\ast(\g)$-modules $C^\ast(\g, \g[1])$.    Similarly, $\Omega^1_{B \g}$ is the sheaf of $C^\ast(\g)$-modules $C^\ast(\g, \g^\vee[-1] )$.  

As before, let $C^\sharp(\g)$ denote the graded algebra underlying $C^\ast (\g)$, equipped with zero differential. 

Note that, as a $C^\sharp(\g)$ module,  $C^\sharp(\g, \g[1])$ is naturally trivialized.  That is, 
$$C^\sharp(\g,\g[1]) \iso C^\sharp(\g) \otimes_{\Omega^\sharp_X} \g[1].$$
This trivialization gives us a natural connection on the tangent bundle $T_{B \g}$: we will let
\begin{align*}
\Atiyah ( T_{B \g} ) &\in H^1(X,\Omega^1_{B \g} (\op{End} T_{B \g} ))\\
&= H^1( X, C^\ast(\g, \g^\vee[-1] \otimes \g^\vee \otimes \g) )
\end{align*} 
be the Atiyah class of this connection.

We can describe this Atiyah class explicitly in terms of the curved $L_\infty$ algebra structure on $\g$.  The expression is entirely local on $X$, and is defined for any curved $L_\infty$ algebras.  

If $\chi$ is a local section on $X$ of $\g[1]$, thought of as a covariant-constant section of the tangent bundle $T _{B \g}$, then the Atiyah class applied to $\chi$ is an element of $\Atiyah_\chi ( T_{B \g} )$ of 
$$
C^\ast(\g, \op{End}(\g) ) = \what{\Sym}^\ast(\g^\vee[-1] ) \otimes \op{End}(\g).
$$ 
Given further elements $U_1,\ldots,U_n \in \g$, the Taylor expansion of $\Atiyah_\chi( T_{B \g})$ is constructed from the sequence of elements
$$
\frac{\partial}{\partial U_1} \dots \frac{\partial}{\partial U_n} \Atiyah_\chi (T_{B \g}) (0) \in \op{End}(\g).
$$ 
\begin{lemma}
\label{lemma_atiyah_hcs}
If $W$ is an element of $\g$, then
$$
\frac{\partial}{\partial U_1} \dots \frac{\partial}{\partial U_n} \Atiyah_\chi (T_{B \g} ) (0)  (W) = l_{n+2} (U_1,\ldots,U_n, \chi, W) \in \g. 
$$
\end{lemma}
\begin{proof}
This is a straightforward local calculation. 
\end{proof}

%% file: d_z_volume_form.txt
\section{Volume forms on elliptic $L_\infty$ spaces}
\label{section_volume_form}
The main result of this paper is that we can understand the Witten class of a complex manifold $X$ in terms of a certain natural volume form on the derived mapping space from an elliptic curve to $X$.  The derived mapping space will be represented as an $L_\infty$ space.  In this section I will give a definition (following \cite{Cos11b}) of a \emph{projective volume form} on an $L_\infty$ space.   I will also explain how, under certain hypothesis, one can integrate functions on an $L_\infty$ space against such a projective volume form. 

\subsection{}
To start with, I will give the definition of a projective volume form on an ordinary (non-derived) manifold.  
\begin{definition}
Let $X$ be a complex manifold.  A \emph{projective volume form} on $X$ is a flat connection on the canonical bundle $K_X$.  Equivalently, it is a trivialization of the $\Oo_X^\times / \C^\times$-torsor associated to  $K_X$.
\end{definition}
Note that what we call a projective volume form is \emph{not} the same as a volume form on $X$ up to scalar multiplication.  Locally, the two notions coincide.  Globally, however, the flat connection on $K_X$ may have non-trivial monodromy: this provides an obstruction to lifting a projective volume form to a volume form. 
\begin{lemma}
A projective volume form on $X$ is the same as a right $D_X$-module structure on $\Oo_X$. 
\end{lemma}
\begin{proof}
If $M$ is a right $D_X$-module, then $M \otimes K_X^{-1}$ is a left $D_X$-module.  Thus, a right $D_X$-module structure on $\Oo_X$ induces a left $D_X$-module structure on $K_X^{-1}$, that is, a flat connection on $K_X^{-1}$; and so a flat connection on $K_X$.    The converse is immediate.   
\end{proof}

\subsection{}

We are interested in projective volume forms on $L_\infty$ spaces.     I will follow a very helpful suggestion of Nick Rozenblyum, and \emph{define} a projective volume form on a formal moduli problem to be a right $D$-module structure on the structure sheaf.  The reason for this approach is that I don't know how to define the canonical sheaf in derived geometry; presumably, the correct definition would involve some version of Grothendieck-Serre duality.  

Let $(X,\g)$ be an $L_\infty$ space.   Let 
$$\op{Vect}(X,\g) = C^\ast(\g,\g[1] )$$ be the sheaf on $X$ of sections of the tangent bundle $T(X,\g)$.   This is a sheaf of dg $C^\ast(\g)$-modules, equipped with an $\Omega^\ast_X$-linear Lie bracket.  We can identify it with the sheaf of $\Omega^\ast_X$-linear derivations of $C^\ast(\g)$, or equivalently with $C^\ast(\g, \g[1])$.

Let us define the associative algebra of differential operators $D (X, \g)$ to be the free associative algebra generated over $\Oo(X, \g)$ by $X \in \op{Vect}( B\g)$ subject to the usual relations:
\begin{align*}
X \cdot f - f \cdot X = (X f) \\
f \cdot X = f X 
\end{align*}
where $\cdot$ denotes the associative product in $D (X, \g)$, and juxtaposition indicates the action of $\op{Vect}(X, \g)$ on $\Oo(X, \g)$ or the $\Oo(X, \g)$-module structure on $\op{Vect}( X, \g)$. 

\begin{definition}
A projective volume form on the $L_\infty$ space $(X,\g)$ is a right $D(B\g)$-module structure on $\Oo( X, \g)$. 
\end{definition}

\subsection{}
The main result of this paper will be the identification of a projective volume form arising from quantizing a certain field theory with the Witten class of a complex manifold $X$.    In order to state this theorem precisely, we need to know how a quantization leads to a projective volume form. 

Let $(X,\g)$ be an $L_\infty$ space, and let $T^\ast[-1] (X,\g)$ be the $L_\infty$ space $(X, \g \oplus \g^\vee [-3] )$.   (As always, $\g^\vee$ denotes the $\Omega^\sharp_X$-linear dual, and we equip $\g \oplus \g^\vee[-3]$ with the natural semi-direct product $L_\infty$ structure).  Note that the invariant pairing of degree $-3$ on $\g \oplus \g^\vee[-3]$ induces a Poisson bracket of degree $+1$ on the Chevalley-Eilenberg cochain complex $C^\ast( \g \oplus \g^\vee[-3])$.  Further, there is a $\C^\times$ action on $(X, \g \oplus \g^\vee[-3])$ given by scaling $\g^\vee[-3]$.  The Poisson bracket has weight $1$ with respect to this $\C^\times$ action. 

\begin{definition}
A $P_0$ algebra is a commutative differential graded algebra with a Poisson bracket of degree $1$.  A $\C^\times$-equivariant $P_0$ algebra is a commutative dga $A$, with a $\C^\times$ action, such that the Poisson bracket has weight $1$. 
\end{definition}
Thus, the sheaf of functions on the $L_\infty$ space $T^\ast[-1] (X,\g)$ has the structure of $\C^\times$-equivariant $P_0$ algebra.

The derived space of solutions to the equations of motion of a classical field theory always has a $P_0$ structure.  For more about these ideas, see \cite{Cos11b,CosGwi11}.  

\subsection{}
Part of the data of our quantum field theory will be a quantization of the $P_0$ algebra describing the classical field theory.  The notion of quantization we use has an operadic definition, developed in more detail in \cite{Cos11b, CosGwi11}.
\begin{definition}
A \emph{BD-algebra} is a cochain complex $A$ flat over $\C[[\hbar]]$, equipped with a commutative product and a Poisson bracket of cohomological degree $1$.  We require that the Poisson bracket satisfies the Leibniz and Jacobi identities, and that the bracket $\{-,-\}$, product $\star$ and differential $\d$ are related by the equation
$$
\d ( a \star b ) = (\d a ) \star b \pm a \star \d b + \hbar \{a,b\}. 
$$ 
\end{definition}
Note that, if $A$ is a BD algebra, then $A$ reduces to a $P_0$ algebra modulo $\hbar$.
\begin{definition}
A quantization of a $P_0$ algebra $A$ is a lift of $A$ to a BD algebra $\til{A}$. 
\end{definition}
We are interested in particular in quantizations of $P_0$ algebras equipped with a $\C^\times$ action.
\begin{definition}
Let $A$ be a $P_0$ algebra with a $\C^\times$ where the Poisson bracket has weight $1$ and the differential and product are preserved.  A $\C^\times$-equivariant quantization of $A$ is a lift to a BD algebra $\til{A}$, where $\til{A}$ has a $\C^\times$ action with the same compatibility with the product, differential, and bracket, and where parameter $\hbar$ has weight $-1$. 
\end{definition}
In this paper, we are only interested in $\C^\times$-invariant quantizations of sheaves of $P_0$ algebras on a manifold $X$ of the form $C^\ast(\g \oplus \g^\vee[-3])$, where $\g$ is a curved $L_\infty$ algebra over $\Omega^\ast_X$ (or, more generally, a sheaf of curved $L_\infty$ algebras).  In this case, there is no loss in generality in assuming that our quantization is of the form $C^\ast(\g \oplus \g^\vee)[[\hbar]]$, with the same product and Poisson bracket, but with a differential of the form $\d + \hbar \tr$.  Here $\d$ is the given differential on $C^\ast(\g \oplus \g^\vee[-3])$, and 
$$\tr: C^\ast(\g \oplus \g^\vee[-3]) \to C^\ast(\g \oplus\g^\vee[-3])$$
is an order $2$ differential operator satisfying the following properties. 
\begin{enumerate}
\item $\tr$ is $\Omega^\ast_X$-linear. 
\item $\tr^2 = 0$ and $[\d,\tr] = 0$, where $\d$ is the differential on $C^\ast(\g \oplus \g^\vee[-3] )$. 
\item The failure of $\tr$ to be a derivation is the Poisson bracket $\{-,-\}$ on $C^\ast(\g \oplus \g^\vee[-3] )$.  That is,
$$
\tr( \alpha \beta ) - (\tr \alpha) \beta - (-1)^{\abs{\alpha}} \alpha \tr \beta = \{\alpha,\beta\}.
$$
\item $\tr$ is of weight $1$ under the $\C^\times$ action on $C^\ast(\g \oplus \g^\vee[-3])$.
\end{enumerate}

\begin{lemma}
There is a natural bijection between $\C^\times$-equivariant quantizations of $T^\ast[-1] (X,\g)$, and projective volume forms on $(X,\g)$. 
\end{lemma}
For the proof, see \cite{Kos85,Cos11b}.  For the purposes of this paper, the proof is not so important: we can take such a quantization to be the definition of a projective volume form.  We will see shortly how, under certain additional hypothesis, the choice of such a quantization allows one to integrate functions on $(X,\g)$.

\subsection{}
In order to get a geometric understanding of the relationship between quantizations and volume forms, let us consider how this works for ordinary  manifolds.  Thus, let $X$ be a manifold, and let $T^\ast[-1] X$ be the graded manifold whose algebra of functions are polyvector fields on $X$.  Then, a volume form $\omega$ on $X$ induces a divergence operator
$$
\op{Div}_\omega : \op{Vect}(X) \to \cinfty_X
$$  
characterized by the property that, for all vector fields $V$, 
$$
\mscr{L}_V \omega = \left( \op{Div}_\omega V \right) \omega.
$$
This divergence operator extends to a map
$$
\op{Div}_\omega : \Gamma(X, \wedge^k T X ) \to \Gamma(X, \wedge^{k-1} TX) 
$$
characterized by the property that
$$
\left( \op{Div}_\omega (\Phi) \right) \vee \omega = \d_{dR} ( \Phi \vee \omega ),
$$
where, for $\Phi \in \Gamma(X, \wedge^k T X)$, $\Phi \vee \omega \in \Omega^{n-k}(X)$ is the form obtained by contracting $\Phi$ against the volume form. 

An easy calculation shows that the operator $\op{Div}_\omega$ satisfies the properties required to define a $\C^\times$-invariant quantization of the $P_0$ algebra of polyvector fields on $X$.

For a general $L_\infty$ space $(X,\g)$, and a $\C^\times$-invariant quantization of $T[-1] (X,\g)$, we should think of the operator
$$
\tr : \op{Vect}(X,\g) \to \Oo(X,\g)
$$
arising from the quantization as being given by the divergence with respect to the corresponding projective volume form.  

\subsection{}
The next question we want to answer is: under what circumstances can one integrate a projective volume form on an $L_\infty$ space?  

To motivate the answer, let us again consider the case of an ordinary smooth manifold $X$.  We have seen that a volume form $\omega$ on $X$ leads to a quantization of $T^\ast[-1] X$, with the operator $\tr$ given by the divergence $\op{Div}_\omega$ of the volume form.  

By construction, there is an isomorphism of sheaves of cochain complexes on $X$
$$
\left(  \Oo (T^\ast[-1] X),  \tr \right) \iso \left(\Omega^\ast(X)[n],  \d_{dR} \right) .
$$
The isomorphism comes from the map
$$
\Gamma(X, \wedge^i TX) \to \Gamma(X,  \Omega^{n-i} X)
$$
given by contracting with $\omega$. 

In particular, we see that the cohomology of the sheaf of complexes $\left(  \Oo (T^\ast[-1] X),  \tr \right)$ is just the constant sheaf $\C$ concentrated in degree $-n$.  

Now, there is a map of sheaves of cochain complexes
$$
\cinfty_X \to \left(  \Oo (T^\ast[-1] X),  \tr \right) \iso \C[n].
$$
Passing to compactly support sections we get a map
$$
\cinfty_c(X) \to H^n_c(X,\C) = \C
$$
from compactly supported smooth functions on $X$ to $H^n_c(X,\C)$, which, since $X$ is oriented, is $\C$. 

This map is, of course, the integral.  

Now, if we are just given the operator $\tr = \op{Div}_\omega$, then we do not get a canonical isomorphism between the cohomology of $\left( \Oo(T^\ast[-1] X), \tr \right)$ and the constant sheaf $\C[n]$.  All we know is that this cohomology is a constant sheaf of rank one concentrated in degree $-n$. 

Even so, this is enough to give us the integral map, up to a constant factor. Indeed, we get a map
$$
\cinfty_c(X) \to H^n_c( \Oo(T^\ast[-1] X), \tr )
$$
and the right hand side is isomorphic (but not canonically) to $\C$. 

\subsection{}
We would like to generalize this story to provide a definition of an integral associated to a projective volume form on an $L_\infty$ space $(X,\g)$.  

Suppose that we have such a projective volume form, which, as above, we view as a $\C^\times$-equivariant quantization of the $\Omega^\ast_X$-linear sheaf of $P_0$ algebras 
$$\Oo(T^\ast[-1] (X,\g)) = C^\ast(\g \oplus \g^\vee[-3] ).$$
The quantization is encoded in the cochain complex $\Oo( T^\ast[-1] (X,\g) )[[\hbar]]$ with differential $\d + \hbar \tr$.

The integral map we are trying to construct will be encoded in the map of $\Omega^\ast_X$-modules 
$$
\Oo(X,\g) \to \Oo ( T^\ast[-1] ( X, \g )) ((\hbar))
$$
In order to show that we get a reasonably well-behaved integral, we need to know something about the cohomology sheaves of $\Oo( T^\ast[-1] (X,\g) ) ((\hbar))$.   In general we won't have much control over these cohomology sheaves.  However, there are some reasonable assumptions on $(X,\g)$ that we can impose which will guarantee that these cohomology sheaves have some nice properties. 
\begin{definition}
An $L_\infty$ space $(X,\g)$ is \emph{locally trivial} the $\cinfty_X$-linear sheaf of $L_\infty$ algebras $\g^{red}$ is locally quasi-isomorphic to the sheaf of sections of a graded vector bundle $V$, with trivial differential and $L_\infty$ structure.   

We say that $(X,\g)$ is \emph{quasi-smooth} if the cohomology sheaves of $\g^{red}$ are concentrated in degrees $1$ and $2$. 

Finally we say that $(X,\g)$ is \emph{nice} if $(X,\g)$ is both quasi-smooth and locally trivial.  
\end{definition}
If $(X,\g)$ is a locally trivial $L_\infty$ space, then the cohomology sheaves $H^i(\g^{red})$ are locally free sheaves of $\cinfty_X$-modules.  In that case, we let $d_i$ denote the rank of $H^i(\g^{red})$.   

Note that if $(X,\g)$ is nice, then so is 
$$T^\ast[-1](X,\g) = (X, \g \oplus \g^\vee[-3]).$$

\subsection{}
Before I state the lemma which tells us that we can integrate on nice $L_\infty$ spaces, I need to introduce some notation.

Let $(X,\g)$ be any $L_\infty$ space, and let $\omega$ be a projective volume form on $(X,\g)$.  Let $\tr_\omega$ be the corresponding operator on $\Oo(T^\ast[-1] ( X,\g))$.   This operator allows us to define a sheaf of cochain complexes on $X$, which we call the divergence complex associated to $\omega$.  It is defined by the formula
$$
\op{Div}^\ast(\omega) = \left( \Oo(T^\ast[-1] ( X,\g)) ((\hbar)), \d + \hbar \tr_\omega \right) .
$$
We let $\mc{H}^i ( \op{Div}^\ast(\omega))$ denote the $i$'th cohomology sheaf of this complex.  

Note that this is a sheaf of $\C((\hbar))$-modules. Further, this complex has a $\C^\times$-action lifting that on $\C((\hbar))$ under which $\hbar$ has weight $-1$. 
\begin{lemma}
Let $(X,\g)$ be a nice $L_\infty$ space, and let $d_i$ denote the rank of $H^i(\g^{red})$. Then, for any projective volume form $\omega$ on $(X,\g)$, the cohomology sheaves $\mc{H}^i( \op{Div}^\ast(\omega))$ are zero except for $i = - d_1 - d_2$.  Further, $\mc{H}^{-d_1 - d_2}(\op{Div}^\ast(\omega))$ is a locally constant rank one sheaf of $\C((\hbar))$ vector spaces. 
\end{lemma}
\begin{remark}
The result holds without the $\C^\times$-equivariant assumption, and with the same proof. 
\end{remark}
\begin{proof}
We need to compute, locally on $X$, the cohomology of $\Oo(T^\ast[-1] (X,\g))[[\hbar]]$ with differential $\d + \hbar \tr$.  Since this cohomology does not change if we replace $\g$ by something quasi-isomorphic, we can assume without loss of generality that the differential and all $L_\infty$ structures on $\g^{red}$ vanish.    Recall that $T^\ast[-1] (X,\g)$ refers to the $L_\infty$ space $(X, \g \oplus \g^\vee [-3])$.  Thus, $\Oo( T^\ast[-1] ( X,\g))$ is the sheaf of $\Omega^\ast_X$-algebras $C^\ast(\g \oplus \g^\vee[-3])$.  Let us filter this by defining $F^i$ to be the image of multiplication by  $\Omega^i_X$.  We can compute the cohomology by the spectral sequence associated to this filtration.  The first term is $\oplus^i \Omega^i[-i] \otimes_{\cinfty_X} C^\ast( \g^{red} \oplus (\g^{red})^\vee[-3])$. 

Thus, to prove the lemma, we have to verify that locally, the cohomology of $C^\ast( \g^{red} \oplus (\g^{red})^\vee[-3])[[\hbar]]$ is a copy of $\cinfty_X$ concentrated in degree $-d_1 -d_2$. 
    
We will get the same answer if we replace $\g^{red}$ by a quasi-isomorphic $L_\infty$ algebra.  We have assumed that $(X,\g)$ is nice, and in particular locally trivial. We can thus assume, with out loss of generality, that $\g^{red}$ has trivial differential and $L_\infty$ structure. Further, by working locally, we can assume that $\g^{red}$ is a free $\cinfty_X$ module.  Let $V$ be the graded vector space which is $\C^{d_1}$ in degree $0$ and $\C^{d_2}$ in degree $-1$.   Locally there is an isomorphism $\g^{red} \iso V^\vee[-1] \otimes_{\C} \cinfty_X$.   The Lie algebra cochains $C^\ast(\g^{red} \oplus(\g^{red})^\vee[-3] ) $ can be identified with $\cinfty_X \otimes \what{\Sym}^\ast ( V \oplus V^\vee[1] )$.  

Note that $\what{\Sym} ( V \oplus V^\vee[1])$ has a $P_0$ structure where the bracket, on generators, is given by the pairing between $V$ and $V^\vee[1]$.  Let us thus assume that we are given a $\C^\times$-equivariant quantization of this $P_0$ algebra, described by a differential operator $\tr$.  We need to verify that the cohomology is, after inverting $\hbar$, concentrated in degree $-d_1 -d_2$.  

Let $x_i, \alpha_j$ refer to a basis of $V^0$ and $V^{-1}$ respectively, and let $\beta_i, y_j$ refer to a dual basis of $V^\vee[1]$.  Thus, $\alpha_j$ and $\beta_i$ are in degree $-1$, whereas $x_i$ and $y_j$ are in degree $0$.  

Then, $\what{\Sym}(V \oplus V^\vee[1])$ is the algebra $\C[[x_i,\eps_j, \delta_i, y_j]]$.   The Poisson bracket is defined by
\begin{align*}
\{x_i, \delta_i\} & = 1\\
\{\eps_j, y_j\} &= 1
\end{align*}
and all other brackets being $0$.  

Let us choose a $\C^\times$ equivariant quantization of this $P_0$ algebra. Such a quantization is defined by an operator $\tr$, which is necessarily of the form $\tr = \tr_0 + \{S,-\}$, where
$$
\tr_0 = \sum_i \dpa{x_i}\dpa{\beta_i}  + \sum_j \dpa{\alpha_j}\dpa{y_j},
$$
and where $S \in \C[[x_i]]$.

We need to compute the cohomology of $\C[[x_i, \alpha_j, \beta_i, y_j, \hbar]]$ with differential 
$$\hbar \tr_0 + \hbar \{S,-\}.$$  
Let us grade $\C[[x_i, \alpha_j, \beta_i, y_j,\hbar]]$ by giving the generators $x_i,\alpha_j, \beta_i, y_j$ all weight $1$, and giving $\hbar$ weight $2$.  Let us filter our complex by defining $F^k \C[[x_i, \alpha_j, \beta_i, y_j,\hbar]]$ to be the subspace of elements of weight $\ge k$.  The differential $\hbar \tr$ preserves this filtration.   Therefore, we can compute cohomology by using a spectral sequence.  

The operator $\hbar \tr_0$ preserves weight, whereas the operator $\hbar \{S_0,-\}$ strictly increases weight.   It follows that the first term in the spectral sequence is given by the cohomology with respect to the operator $\hbar \tr_0$.

Next, observe that we have an isomorphism 
$$
\C[[x_i, \alpha_j, \beta_i, y_j,\hbar]] \iso \C[[x_i,y_j,  \partial_{x_i}, \partial_{y_i} , \hbar]] ,
$$
where, as usual, $\partial_{x_i}$ and $\partial_{y_j}$ are put in degree $-1$.  This isomorphism sends $\alpha_j$ to $\partial_{y_j}$ and $\beta_i$ to $\partial_{x_i}$.  Under this isomorphism, the operator $\tr_0$ corresponds to divergence with respect to the translation invariant volume form 
$$\d Vol = \d x_1 \wedge \dots \wedge \d x_{d_1} \wedge \d y_{1} \dots \wedge \d y_{d_2}.$$
As usual, by contracting with $\d Vol$ we can turn polyvector fields into forms, so that we get an isomorphism
$$
\C[[x_i, \alpha_j, \beta_i, y_j,\hbar]] \iso \C[[x_i,y_j, \d x_i, \d y_j, \hbar]][d_1 + d_2]
$$
where the right hand side is equipped with the differential $\hbar \d_{dR}$.  

When we invert $\hbar$, the formal Poincar\'e lemma gives the desired result.  
\end{proof}

\subsection{}
The lemma shows that the divergence complex $\op{Div}^\ast(\omega)$ is quasi-isomorphic to a local system of $\C((\hbar))$-lines, with a shift.  Further, this local system has an action of $\C^\times$, compatible with the action on $\C((\hbar))$ under which $\hbar$ has weight $-1$.  Thus, we can take the $\C^\times$ invariants, to get a local system of $\C$-lines.

We will let 
$$\mc{D}(\omega) = \mc{H}^{-d_1 -d_2} (\op{Div}^\ast(\omega))^{\C^\times}$$ denote this local system of $\C$-lines. Thus, we have a quasi-isomorphism of sheaves of $\C((\hbar))$-modules
$$
\mc{D}(\omega)((\hbar)) [d_1 + d_2] \iso \op{Div}^\ast(\omega). 
$$

\subsection{}
This lemma is nearly enough to show that we can integrate on a nice $L_\infty$ space.  We need one more condition.
\begin{definition}
A projective volume form $\omega$ on a nice $L_\infty$ space $(X,\g)$ is \emph{integrable} if the local system $\mc{D}(\omega)$ on the manifold $X$ is isomorphic to the orientation local system on $X$.  
\end{definition}

Now suppose that $(X,\g)$ is such an $L_\infty$ space and if $\omega$ is an integrable projective volume form on $(X,\g)$. Then, there is a map of sheaves  
$$
\Oo(X,\g) \to \op{Div}^\ast(\omega)
$$
coming from the natural pull-back map $\Oo(X,\g) \to \Oo(T^\ast[-1](X,\g))$. 

Passing to compactly cohomology, and taking $\C^\times$-invariants on the right hand side, we get a map
$$
H^i_c( X, \Oo(X,\g)) \to H^{i+d_1 + d_2}_c ( X, \mc{D}(\omega) ) .
$$
Since $\mc{D}(\omega)$ is isomorphic to the orientation local system on $X$, Poincar\'e duality tells us that $H^{i+d_1 + d_2}_c ( X, \mc{D}(\omega) )$ is one-dimensional if $i + d_1 + d_2$ is the (real) dimension of the smooth manifold $X$. 
\begin{definition}
Let $(X,\g)$ be as above, and let $n$ denote the real dimension of $X$. The \emph{integral} of an integrable projective volume form on $(X,\g)$ is the map just constructed 
$$
H^{n - d_1 - d_2}_c( X, \Oo(X,\g)) \to H^n_c(X, \mc{D}(\omega)) \iso \C. 
$$
\end{definition}
Because the isomorphism between $\mc{D}(\omega)$ and the orientation local system on $X$ is non-canonical, this integral map is only defined up to a scalar factor.  

\section{Volume forms on the shifted tangent bundle}
Let $X$ be a complex manifold, and let $\g_{X_{\dbar}}$ denote the curved $L_\infty$ algebra encoding the complex structure on $X$. In our study of the Witten genus, projective volume forms on $T[-1] (X,\g_{X_{\dbar}})$ will play an important role.   Note that the sheaf $\Oo(T[-1] (X,\g_{X_{\dbar}}))$ is quasi-isomorphic to the sheaf of complexes 
$$\Omega^{-\ast,\ast}_X = \oplus \Omega^{p,q}_X [p-q]$$
with differential $\dbar$.   Thus,
$$
H^0(X, \Oo( T[-1] (X,\g_{X_{\dbar}} )) = \oplus H^i(X, \Omega^i_{X,\op{hol}}).
$$
If $X$ is compact, there is a natural integration map on this space, which is zero on $H^i(X, \Omega^i_{X, \op{hol}})$ if $i < n$, and which is usual integration on $H^n(X,\Omega^n_{X,\op{hol}})$.  

In this section we will see that this integration map is realized by a canonically-defined projective volume form on the $L_\infty$ space $T[-1] (X,\g_{X_{\dbar}})$.
\begin{theorem}
Let $(X,\g_{X_{\dbar}})$ be the $L_\infty$ space encoding the complex structure on a compact complex manifold $X$.  Then, there is a unique projective volume form $\omega_0$ on $T[-1] (X,\g_{X_{\dbar}})$ which is integrable, and where the integral map
$$
\int : H^0 ( X,  \Oo( T[-1] (X,\g_{X_{\dbar}})) \to \C
$$
(defined up to a multiplicative constant) coincides with the map described above.
\end{theorem}
Although the theorem is morally completely obvious, I will give a detailed proof which will occupy the rest of this section. 

The first step is to construct the volume form.  We will do this explicitly.  Observe that we can represent $T[-1] (X,\g_{X_{\dbar}})$ as the $L_\infty$ space $(X,\g_{X_{\dbar}}[\eps])$ where $\eps$ is a parameter of degree $1$.  We are interested in projective volume forms on this space.  Thus, we need to understand 
$$T^\ast[-1] T[-1] (X,\g_{X_{\dbar}}) = T^\ast[-1] (X,\g_{X_{\dbar}}[\eps] ).$$
In general, $T^\ast[-1](X,\g) = (X,\g \oplus \g^\vee[-3])$.  Thus, we see that 
\begin{align*}
T^\ast[-1] T[-1] (X,\g_{X_{\dbar}}) &= \left(X, \g_{X_{\dbar}}[\eps] \oplus \left(\g_{X_{\dbar}}[\eps]\right)^\vee[-3] \right) \\
& = \left(X, \g_{X_{\dbar}}[\eps] \oplus \g_{X_{\dbar}}^\vee[\eps][-2] \right) \\
&=\left(X, \left( \g_{X_{\dbar}} \oplus \g_{X_{\dbar}}^\vee[-2]\right)[\eps] \right).
\end{align*}
The sheaf $\g_{X_{\dbar}}[\eps] \oplus \g_{X_{\dbar}}^\vee[\eps][-2] )$ is given the $L_\infty$ structure arising from the natural $\g_{X_{\dbar}}[\eps]$ action on $\g_{X_{\dbar}}^\vee[\eps]$, which is the $\eps$-linear extension of the $\g_{X_{\dbar}}$ action on $\g_{X_{\dbar}}^\vee$.  The invariant pairing of degree $-3$ on $(\g_{X_{\dbar}} \oplus \g_{X_{\dbar}}^\vee[-2])[\eps]$ is the composition of the natural $\C[\eps]$-valued pairing of degree $-2$ with the degree $-1$ map $\C[\eps] \to \C$, sending $\eps$ to $1$. 

Note that $(X,\g_{X_{\dbar}} \oplus \g_{X_{\dbar}}^\vee[-2])$ is the $L_\infty$ space $T^\ast(X,\g_{X_{\dbar}})$.  Thus, we have constructed a natural isomorphism
$$
T^\ast[-1] (T[-1] (X,\g_{X_{\dbar}})) = T[-1] T^\ast(X,\g_{X_{\dbar}})).  
$$
We will use this isomorphism extensively shortly.  

Now, to construct a projective volume form, we need to produce an operator $\tr_0$ on $C^\ast( \g_{X_{\dbar}}[\eps] \oplus \g_{X_{\dbar}}^\vee[\eps] [-2])$.  I will give three descriptions of this operator: one as a formula, and two more conceptual interpretations.

Let $K \in \g_{X_{\dbar}} \otimes_{\Omega^\sharp_X} \g_{X_{\dbar}}^\vee$ denote the inverse of the pairing between $\g_{X_{\dbar}}$ and $\g_{X_{\dbar}}^{\vee}$. (In what follows, tensor products will always be taken over $\Omega^\sharp_X$ unless otherwise specified). 

From $K$ we construct an anti-symmetric tensor
$$
\til{K} = (\eps \otimes 1 + 1 \otimes \eps) K \in (\g_{X_{\dbar}} [\eps] \oplus \g_{X_{\dbar}}^\vee[\eps][-2] ) ^{\otimes 2}.
$$
We define the operator
$$
\tr_0 : C^\ast(\g_{X_{\dbar}}[\eps] \oplus \g_{X_{\dbar}}^\vee[\eps][-2]  ) \to C^\ast(\g_{X_{\dbar}}[\eps] \oplus \g_{X_{\dbar}}^\vee[\eps][-2]  )
$$
to be the operator of contracting with $\til{K}$.   In other words, $\tr_0$ is the unique order $2$ differential operator which is zero when restricted to constant and linear elements of 
$$C^\ast(\g_{X_{\dbar}}[\eps] \oplus \g_{X_{\dbar}}^\vee[\eps][-2]  ) = \Sym^\ast \left( \g_{X_{\dbar}}[\eps][1] \oplus \g_{X_{\dbar}}^\vee[\eps][-1]  \right)^\vee 
$$
and which, on quadratic elements, is defined by $\til{K}$. 

One needs to verify that $\tr_0$ is a cochain map, that $\tr_0^2 = 0$, and that the failure of $\tr_0$ to be a derivation is measured by the Poisson bracket on $C^\ast(\g_{X_{\dbar}}[\eps] \oplus \g_{X_{\dbar}}^\vee[\eps][-2]  )$.  All of these properties are simple computations.  Thus, we have constructed our projective volume form. 

\subsection{}
Let us now give the more conceptual construction of the projective volume form.  We have constructed an isomorphism
$$
T^\ast[-1] T[-1] (X,\g_{X_{\dbar}} ) \iso T[-1] T^\ast(X,\g_{X_{\dbar}}). 
$$
The symplectic form on $T^\ast(X,\g_{X_{\dbar}})$ gives an isomorphism
$$
T[-1] T^\ast(X,\g_{X_{\dbar}}) \iso T^\ast[-1] T^\ast(X,\g_{X_{\dbar}}).
$$
Composing, we get an isomorphism
$$
T^\ast[-1] T[-1] (X,\g_{X_{\dbar}} ) \iso T^\ast[-1] T^\ast(X,\g_{X_{\dbar}}).
$$
This isomorphism respects the natural Poisson brackets on both sides.

We want to construct an operator $\tr_0$ on $\Oo(T^\ast[-1] T[-1] (X,\g_{X_{\dbar}}))$.  By this isomorphism, it suffices to construct such an operator on $T^\ast[-1] T^\ast(X,\g_{X_{\dbar}})$. This will be given by a projective volume form on $T^\ast(X,\g_{X_{\dbar}})$. Now, $T^\ast(X,\g_{X_{\dbar}})$ has a canonically-defined volume form. We take our projective volume form to be that associated to this actual volume form.

A simple computation verifies the equivalence between the two definitions of the operator $\tr_0$ we have given so far.  

\subsection{}
The third description actually works on a general $L_\infty$ space $(X,\g)$ and not just one arising from a complex manifold.  As above, we have a canonical isomorphism 
$$T[-1] T^\ast(X,\g) \iso T^\ast[-1] T[-1] (X,\g).$$ 
Now, functions on $T[-1] T^\ast(X,\g)$ can be identified with forms on $T^\ast(X,\g)$.  Thus, one has a de Rham operator
$$
\d_{dR} : \Oo(T[-1] T^\ast(X,\g)) \to \Oo(T[-1] T^\ast(X,\g)) 
$$
of cohomological degree $-1$.   We also have an operator
$$
\iota_\pi : \Oo(T[-1] T^\ast(X,\g)) \to \Oo(T[-1] T^\ast(X,\g)) 
$$
given by contracting with the Poisson tensor $\pi$ on $T^\ast(X,\g)$.  In the language of forms, $\iota_\pi$ maps $\Omega^{i}(T^\ast(X,\g))$ to $\Omega^{i-1}(T^\ast(X,\g))$. 

Let $L_\pi$ denote the Lie derivative with respect to $\pi$, defined by
$$
L_\pi  = [\d_{dR}, \iota_\pi].
$$
Then, the third description of the quantization of $T^\ast[-1] T[-1] (X,\g)$ is that the associated BD algebra is
$$
\left( \Oo(T[-1] T^\ast(X,\g) ), \d + \hbar L+\pi \right)
$$
where $\d$ is the standard differential on $\Oo(T[-1] T^\ast(X,\g))$.

\subsection{}
The next thing to verify is that this volume form is integrable.  Integrability is a property of the divergence complex associated to our projective volume form; that is, of the complex $\Oo(T^\ast[-1] T[-1] (X,\g_{X_{\dbar}}))[[\hbar]]$ with differential $\d + \hbar \tr_0$.

Now, by the second construction, this complex is isomorphic to the divergence complex for the canonical volume form on $T^\ast(X,\g_{X_{\dbar}})$.  As usual, we can identify this divergence complex with the complex of $\Omega^{\ast,\ast}(\what{T}^\ast X) ((\hbar))[2 \op{dim} X]$, where $\what{T}^\ast X$ denotes the formal completion of the cotangent bundle of $X$ along the zero section.  The cohomology sheaves of this complex are zero except in degree $-2 \op{dim} X$, and in this degree is $\C((\hbar))$.  Since this is a trivial local system, and since $X$ is a complex manifold and therefore orientable, we see that our projective volume form is integrable.  

\subsection{}
Next, we have to verify that the integral map for this projective volume form is as claimed.   In order to calculate this, we will translate the integral map into the language of ordinary complex geometry.

Let us use the notation
$$
\PV^{i,\ast}(\what{T}^\ast X) = \Omega^{0,\ast}(\what{T}^\ast X, \wedge^i T (\what{T}^\ast X)).
$$
The notation $\PV^{-\ast,\ast}(\what{T}^\ast X)$ will refer to $\oplus \PV^{i,\ast}(\what{T}^\ast X)[i]$.

The symplectic form $\omega$ on $\what{T}^\ast X$ induces an isomorphism 
$$
\Phi : \PV^{i,\ast}(\what{T}^\ast X) \iso \Omega^{i,\ast}(\what{T}^\ast X).
$$
Let $\tr : \PV^{i,\ast}(\what{T}^\ast X) \to \PV^{i-1,\ast}(\what{T}^\ast X)$ denote the divergence operator for the canonical volume form.  The divergence complex for this volume form is then $\PV^{-\ast,\ast}(\what{T}^\ast X)((\hbar))$ with differential $\dbar + \hbar \tr$. 

The integral map arises from the cochain map 
\begin{equation*}
\Omega^{-\ast,\ast}(X) \xto{\Phi \circ \pi^\ast} \left( \PV^{-\ast,\ast}(X)((\hbar)), \dbar + \hbar \tr \right).
\tag{$\dagger$}
\label{eqn_dagger_integral}
\end{equation*}
As we have seen, the cohomology sheaf on the right hand side is a copy of the constant sheaf $\C((\hbar))$ in degree $-2\op{dim}_{\C} X$. 

The first thing to check is that this integral map is zero on $\Omega^{-i,\ast}(X)$.  Thus, we need to verify that the map $\Phi \circ \pi^\ast$ in equation (\ref{eqn_dagger_integral}) is homotopically trivial.  To see this, let $\eta$ denote the Liouville vector field on $\what{T}^\ast X$.  Let $m_\eta$ denote the operator of wedging with $\eta$ on $\PV^{-\ast,\ast}(\what{T}^\ast X)$.  Then, 
\begin{lemma}
The map $\hbar^{-1} m_\eta \circ \Phi \circ \pi^\ast$ is a cochain homotopy between $(n-i) \Phi \circ \pi^\ast$ and $0$, where $n = \op{dim}_{\C} X$.
\end{lemma}
\begin{proof}
Let $L_\eta$ denote the Lie derivative of $\eta$ acting on $\PV^{-\ast,\ast}(\what{T}^\ast X)$.  Note that $L_\eta$ acts is $-i$ on the image of $\Omega^{-i,\ast}(X)$.  Further, the divergence of $\eta$ is the constant $n$: that is $\tr \eta = n$. It is a standard identity that (for any holomorphic vector field $\chi$ on $\what{T}^\ast X$)
$$
[\tr, m_\chi] = L_\chi + (\tr \chi).
$$
It follows that, for all $\alpha \in \Omega^{-i,\ast}(X)$, 
$$
\tr  m_\chi \Phi \circ \pi^\ast \alpha = (n-i)\Phi \circ \pi^\ast \alpha,
$$
as desired.
\end{proof}

Next, we need to verify that the map
$$
H^n(X,\Omega^n_{X,\op{hol}} ) \to H^0( \PV^{-\ast,\ast}(\what{T}^\ast X)((\hbar)) = H^{2n}( X, \C((\hbar)) ) 
$$
is proportional to the usual integral map (where we identify $H^{2n}(X,\C((\hbar))$ with $\C((\hbar))$).  It suffices to verify that this map is non-zero. We can do this by working locally: if $D$ is a disc in $X$, we need to verify that the map
$$
H^n_c(D, \Omega^n_{D,\op{hol}}) \to H^{2n}_c(D, \C((\hbar)) ) 
$$
is proportional to the usual integral (where the subscript $c$ indicates compactly supported cohomology).  This computation can be performed explicitly in coordinates, and is left to the reader.

\subsection{}

The final part of the theorem was the uniqueness claim.  In fact, we will prove something a little stronger (which we will use later).

Let $(X,\g_{X_{\dbar}})$ be, as above, the $L_\infty$ space associated to a compact complex manifold, and let $\d Vol_0$ denote the canonical projective volume form on $T[-1] (X,\g_{X_{\dbar}})$.  Let $\tr_0$ be the associated divergence operator.   

Recall that projective volume forms on any $L_\infty$ space $(X,\g)$ form a torsor for $H^0(X, C^\ast_{red}(\g))$, where $C^\ast_{red}(\g)$ refers to the reduced Lie algebra cochains of $\g$.  Indeed, if $\tr$ is the divergence operator describing any such projective volume form, and if $f \in H^0(X,C^\ast_{red}(\g))$ is a function, then $\tr + \{f,-\}$ defines a new projective volume form (where $\{-,-\}$ denotes the Poisson bracket on functions on $T[-1](X,\g)$). 

Recall also that $C^\ast_{red}(\g)$ corresponds to functions on $(X,\g)$ modulo constants:
$$
C^\ast_{red}(\g) = \Oo(X,\g) / \C
$$
where $\C$ is the constant sheaf.

In the case of interest, there is a natural isomorphism$$H^0(X, \Oo(T[-1](X,\g)) / \C) \iso \oplus_{i > 0} H^i(X, \Omega^i_{X,\op{hol}}) \oplus H^0(X, \Oo_X / \C).$$ 
Further, note that $H^0(X, \Oo_X / \C)$ is isomorphic (via the de Rham differential) to $H^0(X, \Omega^1_{X,\op{hol}})$.   Here, $\Omega^1_{X,\op{hol}}$ refers to the sheaf of holomorphic $1$-forms on $X$. 

Thus, any projective volume from $\omega$ on $T[-1](X,\g_{X_{\dbar}})$ has divergence operator $\tr_\omega$ of the form
$$
\tr_\omega = \tr_0 + \{S_\omega,-\} + \{O_\omega,-\} 
$$
where 
\begin{align*}
S_\omega &\in \oplus_{i > 0}H^i(X,\Omega^i_{X,\op{hol}}) \\
O_\omega & \in H^0(X,\Omega^1_{X,\op{hol}}). 
\end{align*}
Note that we can view $O_\omega$ as an element of $H^1(X,\C)$. 
\begin{proposition}
A projective volume form $\omega$ on $T[-1](X,\g_{X_{\dbar}})$ is integrable if and only if 
$$O_\omega \in H^1(X,\Z 2 \pi i ) \subset H^1(X,\C);$$ 
that is, if $O_\omega/ 2 \pi i $ has integral periods. 

If $\omega$ is integrable, then the integration map
$$
\int - \omega : H^0(X, \Oo(T[-1] (X,\g_{X_{\dbar}}))  ) = \oplus H^i(X,\Omega^i_{X,\op{hol}}) \to \C
$$ 
sends
$$
\alpha \to \int_X \left[ e^{S_\omega} \alpha \right]_n.
$$
where $[-]_n$ indicates projection onto the component in $H^n(X,\Omega^n_{X,\op{hol}})$. 
\end{proposition}
\begin{proof}
First, suppose that $O_\omega$ is zero.  Note that the function $S_\omega$, which \emph{a priori} is a function on $T[-1](X,\g_{X_{\dbar}})$ modulo constants, lifts to an actual function. Our convention is that the constant term of the lift (that is, the term in $H^0(X,\Oo_X)$) is zero.  We will refer to this lift as $S_\omega \in H^0(X,\Oo(T[-1](X,\g_{X_{\dbar}})))$.  

Then, a standard formula shows that
$$
e^{-S_\omega} \tr_0 e^{S_\omega} = \tr_0 + \{S_\omega,-\} = \tr_\omega.
$$
Let $\omega_0$ denote the standard projective volume form corresponding to $\tr_0$.  Recall that the divergence complex for $\omega_0$ is the complex
$$
\op{Div}^\ast(\omega_0) = \left(  \Oo ( \what{T}^\ast[-1] (T[-1] (X,\g_{X_{\dbar}}))) ((\hbar)) , \d  + \hbar \tr_0\right).
$$
The identity above shows that multiplying by $e^{S_\omega/\hbar}$ gives a cochain isomorphism 
$$
e^{S_\omega} : \op{Div}^\ast(\omega) \to \op{Div}^\ast(\omega_0).
$$
Recall that integrability of $\omega$ amounts to the statement that the cohomology sheaves of the divergence complex $\op{Div}^\ast(\omega)$ form a trivial $\C((\hbar))$-local system.  It follows that integrability of $\omega_0$ implies integrability of $\omega$.

The integral against $\omega$ is defined by the map
$$
\Oo(T[-1](X,\g_{X_{\dbar}})) \to \op{Div}^\ast(\omega).
$$
The fact that $e^{S_\omega}$ provides a cochain map from $\op{Div}^\ast(\omega)$ to $\op{Div}^\ast(\omega_0)$ immediately implies that, for $\alpha \in H^0(X, \Oo(T[-1](X,\g_{X_{\dbar}})))$, we have
$$
\int_{T[-1](X,\g_{X_{\dbar}})} \alpha \omega = \int_{T[-1](X,\g_{X_{\dbar}}}) e^{S_\omega} \alpha \omega_0.
$$
We have already shown that
$$
\int_{T[-1](X,\g_{X_{\dbar}})} \alpha \omega_0 = \int_X [\alpha],
$$
leading to the desired formula for the integral against $\omega$. 

Finally, we need to verify that $\omega$ is integrable if and only if $O_\omega \in H^0(X, \Omega^1_{X,\op{hol}})$ vanishes.  We have already seen that the divergence complex is independent, up to quasi-isomorphism, of $S_\omega$.  Thus, we will assume that $S_\omega = 0$.   

Recall that we can identify the divergence complex for $\omega_0$ as
$$
\op{Div}^\ast( \omega_0 ) = \left( \Omega^{\ast,\ast}(\what{T}^\ast X) ) ((hbar)) [2 n ] , \dbar + \hbar \partial \right),
$$
where, as before, $\what{T}^\ast X$ indicates the formal completion of the cotangent bundle of $X$ along the zero section.  

Modifying the operator $\tr_0$ by adding on $\{O_\omega,-\}$ for some $O_\omega \in H^0(X, \Omega^1_{X,\op{hol}})$ amounts to adding the operator $O_\omega \wedge$ to the complex of forms on $\what{T}^\ast X$.  The cohomology sheaf of this is the non-trivial local system obtained from viewing the closed $1$-form $O_\omega$ as a connection on the trivial line bundle.  

This local system is trivial if and only if $O_\omega/2 \pi i$ has integral periods.  
\end{proof}

\subsection{}
The main theorem of this paper states that we can identify the volume of the derived space of degree $0$ maps from an elliptic curve $E$ to a complex manifold $X$ with the Witten genus of $X$.  The volume form on this mapping space arises from quantum field theory, as we will see shortly.

So far, we have developed a language to discuss such derived mapping spaces and volume forms on them.  As a pay-off, we can state a more precise version theorem. 

Let $E$ be an elliptic curve, and let $\mc{H}(E) \subset \Omega^{0,\ast}(E))$ be the harmonic part of the Dolbeaut complex of $E$.  Let $X$ be a complex manifold, and let $(X,\g_{X_{\dbar}})$ denote the corresponding $L_\infty$ space.  We have seen that the derived space of degree $0$ maps from $E$ to $X$ can be described by the $L_\infty$ space $(X,\g_{X_{\dbar}} \otimes\Omega^{0,\ast}(E))$.  

This is equivalent to the $L_\infty$ space $(X, \g_{X_{\dbar}} \otimes \mc{H}(E))$.    We will let $(X, \g_{X^E})$ denote this $L_\infty$ space. 

The choice of a holomorphic volume form $\omega$ on $E$ leads to an isomorphism $\mc{H}(E) \iso \C[\eps]$, under which the element $\alpha \in \mc{H}(E)$ with $\int_E \alpha \wedge \omega = 1$ goes to $\eps$.

Thus, if we choose such an $\omega$, we find an isomorphism of $L_\infty$ spaces
$$
(X,\g_{X^E}) \iso (X, \g_{X_{\dbar}}[\eps] ) = T[-1] (X,\g_{X_{\dbar}}). 
$$

As we will see shortly, a quantization of a classical field theory in the sense of \cite{Cos11} leads to the quantization of a $P_0$ algebra associated to the classical field theory.  

In the case of interest, the field theory is the cotangent theory \cite{Cos11b} associated to the space of holomorphic maps from $E$ to $X$.  It turns out that this leads to a $\C^\times$-equivariant quantization of the $P_0$ algebra $T^\ast[-1](X, \g_{X^E})$, and so to a volume form on $(X,\g_{X^E})$.

We have seen that $(X,\g_{X^E})$ can be identified with $T[-1] (X,\g_{X_{\dbar}})$ (once we have chosen a holomorphic volume form on $E$).  Our main theorem asserts that, after making this identification, the volume form $\d Vol_E$ on $(X,\g_{X^E})$ corresponds to the volume form $\op{Wit}(X,E) \d Vol_0$, where $\op{Wit}(X,E) \in H^0(X, \Oo(T[-1](X,\g_{X_{\dbar}}))$ is the Witten genus of $X$, and $\d Vol_0$ refers to the trivial projective volume form on $T[-1](X,\g_{X_{\dbar}})$ constructed above.  

%% file: e_holomorphic_CS.txt
\newpage
\section*{\Large Part II: Holomorphic Chern-Simons theory}
\section{Introduction}
\label{section_hcs}
In this section we will describe the classical space of fields of the theory I call holomorphic Chern-Simons theory, and then consider a generalization of this theory which we will use throughout the rest of the paper. 

Let $X$ be a complex manifold, and let $(X,\g_{X_{\dbar}})$ denote the corresponding $L_\infty$ space. 
We have seen that we can encode the space of maps from an elliptic curve $E$ to $X$ (which are infinitesimally near the constant map) in terms of the $L_\infty$ space  $\Omega^{0,\ast}(E) \otimes \g_{X_{\dbar}}$. This will allow us to describe the space of fields of holomorphic Chern-Simons theory in a way amenable to the perturbative renormalization techniques of \cite{Cos11}.

Let $E$ be a Riemann surface equipped with a never-vanishing holomorphic volume element $\omega$.  We will be concerned with maps $E_{\dbar} \to T^\ast X$ which are infinitesimally near a constant map to $X$.  Thus, the space of fields of our field theory is 
$$
\Omega^{0,\ast}(E) \otimes (\mf g_X[1] \oplus \mf g_X^\vee[-1]). 
$$
The summand $\g_{X}^\vee$ is introduced because $\g_{X} \oplus \g_{X}[-2]$ is the curved $L_\infty$ algebra encoding $T^\ast X$. Since we are only interested in the formal completion $\what{T}^\ast X$ of $T^\ast X$ near $X$, we have restricted this curved $L_\infty$ algebra to $X$.

In the language of \cite{Cos11b}, the theory we are constructing is the cotangent theory to the elliptic moduli problem of degree $0$ holomorphic maps from $E$ to $X$.  

\subsection{}
Our main theorem will be stated in a more general situation, where our target is any $L_\infty$ space $(X, \g)$.  The most general situation is as follows.  

In the general situation, the space of maps $E_{\dbar} \to \what{T}^\ast( X, \g)$ which are infinitesimally close to a constant map to $X$ is represented by the $L_\infty$ space
$$
\left( X, \Omega^{0,\ast}(E) \otimes (\mf g \oplus \mf g^\vee[-2])    \right).
$$
The space of fields is the sheaf of $\Omega^\ast_X$-modules
$$
\Fields = \Omega^{0,\ast}(E) \otimes (\mf g[1] \oplus \mf g^\vee[-1]),
$$ 
The results of \cite{Cos11} apply in this context, where the fields are sheaves of modules over a general Fr\'echet base ring.  

\subsection{}
The classical action on our space of fields can be described as follows.  

Let 
$\alpha \in \Omega^{0,\ast}(E) \otimes \mf g[1]$, and $\beta \in \Omega^{0,\ast}(E) \otimes \mf g^\vee[-1]$.  (We are abusing notation here: $\alpha,\beta$ are local sections of a sheaf on $X$).  If $E$ is non-compact we must assume that both $\alpha$ and $\beta$ have compact support on $E$. 

Then, the classical action $S$ is given by the formula
$$
S(\alpha + \beta) = \int_\C \omega \wedge \left( \ip{l_0, \beta}  + \ip{\dbar \alpha, \beta  } + \sum_{k \ge 1} \frac{1}{k!} \ip{l_k(\alpha^{\otimes k}), \beta }   \right).
$$
The fact that the $l_i$ define an $L_\infty$ structure implies that the action $S$ satisfies the classical master equation
$$
\{S,S\} = 0.
$$

\subsection{}
Now we have described classical holomorphic Chern-Simons theory in the generality we need.  Next, I will restate the main theorem of this paper in this generality. 
\begin{theorem}
Let $(X,\g)$ be an $L_\infty$ space.  
\begin{enumerate}
\item The simplicial set of quantizations of the holomorphic Chern-Simons theory of maps $\C \to \what{T}^\ast(X,\g)$\footnote{where, as always, we only consider maps infinitesimally close to a constant map to $X$} invariant under the symmetry group $\op{Aff}(\C) \times \C^\times$ (where the $\C^\times$ symmetry arises from rescaling the cotangent fibres of $\what{T}^\ast(X,\g)$ is weakly equivalent to the simplicial set of trivializations of the class
$$
\op{ch}_2(T (X,\g) ) \in \R \Gamma( X, \Omega^2_{cl}(X,\g) [2] ) 
$$
where $\Omega^2_{cl}(X,\g)$ is the sheaf on $X$ of closed two-forms on $(X,\g)$. In the case when $\g = \g_{X}$ is the $L_\infty$ algebra associated to a complex manifold $X$, this is quasi-isomorphic to the sheaf of closed holomorphic $2$-forms on $X$.
\item
Invariance under $\op{Aff}(\C)$ implies that any such quantization yields a quantization of holomorphic Chern-Simons theory on any elliptic curve $E$.   If we choose a volume form $\omega$ on $E$, then we find an quasi-isomorphism of $BD$ algebras  between the global observables of the theory on $E$ and the complex
$$
\left( \Omega^{-\ast}(\what{T}^\ast (X,\g)), \hbar L_\pi + \hbar\{ \op{log} \op{Wit}((X,\g),E,\omega), - \} \right).
$$
\end{enumerate}
\end{theorem}
Note that $\Omega^{-\ast}(\what{T}^\ast(X,\g) )$ is the same as $\Oo(T[-1] \what{T}^\ast (X,\g) )$.    

\subsection{}
Let us specialize to the case of the $L_\infty$ space $(X,\g_{X_{\dbar}})$ associated to a complex manifold.  Then, the quantization provides a projective volume form $\d Vol_E$ on $T[-1] (X,\g_{X_{\dbar}})$.   An immediate corollary of this result, and of the results of section \ref{section_volume_form}, is the following.
\begin{corollary}
This projective volume form is integrable, and the integration map 
\begin{align*}
 H^0(X, \Oo(T[-1](X,\g_{X_{\dbar}})) &\to \C\\
\alpha & \mapsto \int_{T[-1] (X,\g_{X_{\dbar}} )}  \alpha \d Vol_E 
\end{align*}
(which is defined up to a projective factor) is the map which sends 
$$\alpha \in \oplus H^i(X,\Omega^i_X) = H^0(X,\Oo(T[-1](X,\g_{X_{\dbar}})))$$ 
to
$$
 \int_X [ \op{Wit}(X,E) \alpha]_n ,
$$
where $[-]_n$ is the projection onto $H^n(X,\Omega^n_X)$. 
\end{corollary}
This is the form of the theorem presented in the introduction. 
												
\subsection{}
The rest of the paper is devoted to proving this theorem.   Thus, throughout the rest of the paper, we will omit all mention of the space $X$, and instead consider a curved $L_\infty$ algebra $\g$ as above.   The results stated in the introduction arise when we specialize to the case $\g = \g_{X_{\dbar}}$. 

Note that if we change the the $L_\infty$ space $(X,\g)$ by a homotopy -- that is, by a family of $L_\infty$ structures over $\Omega^\ast([0,1])$ -- then the classical field theory is also changed by a homotopy.  One can treat such homotopies at the quantum level, by quantizing the theory relative to the base ring $\Omega^\ast([0,1])$.  The $L_\infty$ algebra $\g_{X_{\dbar}}$ associated to a complex manifold $X$ is only well-defined up to homotopy (indeed, up to a contractible choice).  In order to ensure that the quantum theory we construct behaves well with respect to these homotopy equivalences, we will  always work relative to an arbitrary nuclear Fr\'echet dg base ring $A$.  

%% file: e_z_hcs2.txt
\section{Holomorphic Chern-Simons theory in more detail}
\label{section_hcs2}
In this section, we will describe, in more detail, the action functional and the propagator of our holomorphic Chern-Simons theory, and give a more precise statement of the main theorem.   In this section (and throughout) we will often avoid mention of the manifold $X$; of course, everything is a sheaf of $\Omega^\ast_X$-modules.  

There are only two classes of Riemann surface $E$ of interest to us. Either $E$ is compact, and therefore (because of the existence of a holomorphic volume element) an elliptic curve.  Or, $E = \C$ with $\omega = \d z$. In what follows, we will assume that we are in one of these two situations.

Let us split $S$ up into kinetic and interacting parts by 
$$
S(\alpha,\beta) = \int_E \omega \vee \left( \ip{\dbar \alpha, \beta} + \ip{l_1 \alpha,\beta} \right) + I_{hCS}(\alpha,\beta)
$$
where
$$
I_{hCS}(\alpha,\beta) = \int_E \omega \wedge \left( \ip{l_0, \beta}  +  \sum_{k \ge 2} \frac{1}{k!} \ip{l_k(\alpha^{\otimes k}), \beta }   \right).
$$
The holomorphic Chern-Simons interaction $I_{hCS}$ will be a key object throughout this paper.

We will let
$$
Q = \dbar + l_1 :\Fields \to \Fields.
$$

\subsection{}
In order to apply the renormalization techniques of \cite{Cos11}, we need to choose a gauge-fixing operator.  The natural gauge-fixing operator on our situation is the operator
$$
\dbar^\ast : \Fields \to \Fields,
$$
defined using the flat metric on $E$ associated to the holomorphic volume element $\omega$. 

A key part of the approach to quantum field theory of \cite{Cos11} is the Laplacian
$$
D = [\dbar^\ast, Q] : \Fields \to \Fields.
$$
Note that $[\dbar^\ast, l_1] = 0$, so that 
$$
D = [\dbar^\ast,\dbar]
$$
is the usual Laplacian acting on the Dolbeaut complex.

\subsection{}
As a first step in the analysis of holomorphic Chern-Simons theory on $E$, we will write an explicit expression for the propagator of the theory. 

Let $\op{Id}_{\mf g}$ be the element of cohomological degree $0$ of $\g[1] \otimes \g^\vee[-1]$ corresponding to the identity element of $\op{End}(\g ) = \g \otimes \g^\vee$.  Let 
$$
C_\g = \op{Id}_{\g} + \Id_{\g^\vee} \in \g[1] \otimes \g^\vee[-1] \oplus \g^\vee[-1] \otimes \g[1].
$$
Note that $C_\g$ is an anti-symmetric element.

For clarity, let me explain what this looks like in the special case when $\g$ is Lie algebra of cohomological degree $0$, and we work over $\C$ instead of over $\Omega^\ast_X \otimes A$.  In that case, let $V_i$ be a basis of $\g$, and let $V_i^\vee$ be the dual basis of $\g^\vee$.  Then,
$$
C_\g = \sum V_i \otimes V_i^\vee + \sum V_i^\vee \otimes V_i.  
$$

In the case when our Riemann surface $E$ is $\C$, the heat kernel for the operator $D$ is given, up to factors of $\pi$, by the expression
$$
K_t  = t^{-1} e^{-\norm{z-w}^2 / t} (\d \br{z} \otimes 1 - 1 \otimes \d \br{w} ) \otimes C_\g \in \Fields \otimes \Fields.
$$
Note that $K_t$ is a symmetric element of cohomological degree $1$ of $\Fields \otimes \Fields$.

If we work in the simple situation where $\g$ is a purely even Lie algebra over $\C$, and we choose a basis $V_i$ for $\g$, then we can write
$$
K_t = t^{-1} e^{-\norm{z-w}^2 / t} \left ( \sum \d \zbar X_i \otimes X_i^\vee + \d \zbar X_i^\vee \otimes X_i + X_i \otimes \d \wbar X_i^\vee + X_i^\vee \otimes \d \wbar X_i \right).
$$
The propagator for the theory is
$$
P(\eps,L) = \int_{\eps}^L \dbar^\ast K_t \d t.
$$
If our source Riemann surface is $\C$, we can write the propagator as
$$
P(\eps,L) = \int_{\eps}^L t^{-2} (\zbar - \wbar) e^{-\norm{z-w}^2/t}  C_\g \d t
$$
(up to factors of $\pi$).  Here, $\eps$ is an ultra-violet regulating parameter, and $L$ is an infra-red regulating parameter.  Sending $\eps \to 0$ and $L \to \infty$ amounts to turning off these regulating parameters.   When working on $\C$, we will always keep $L < \infty$, but we will send $\eps \to 0$.

\subsection{}
Before I can give a precise statement of the main theorem of this paper, I need to recall the definition of quantum field theory used in \cite{Cos11}, adapted to our particular situation.

We will let
$$
\Oo(\Fields) = \prod_{n \ge 1} \Hom_{\Omega^\sharp_X )} (\Fields^{\otimes n}, \Omega^\sharp(X) )
$$
denote the algebra of formal power series on the $\Omega^\sharp_X$-module $\Fields$, modulo constants.   In this expression, everything is taken relative to our base ring, $\Omega^\sharp_X$.  In addition, all tensor products are completed tensor products of sheaves of nuclear Fr\'{e}chet spaces, and $\op{Hom}$ refers to the space of continuous $\Omega^\sharp_X$-linear maps.   (See \cite{Cos11} for further details). 

Thus, $\Oo(\Fields)$ is a sheaf of graded commutative algebra over $\Omega^\sharp_X$. 

There is a subsheaf of $\Oo(\Fields)$ of particular interest, consisting of those functions on $\Fields$ which are \emph{local action functionals}.  A local action functional is a function on $\Fields$ which arises by integral of a Lagrangian density (see \cite{Cos11} for further details).  We will let
$$
\Ool(\Fields) \subset \Oo(\Fields)
$$
be the subsheaf of $\Omega^\sharp_X$-modules consisting of local action functionals.

\subsection{}
As we have seen, the classical field theory is described by a classical interaction functional 
$$I \in \Ool(\Fields)$$ 
which satisfies the classical master equation
$$
Q I + \tfrac{1}{2} \{I,I\} = 0.
$$
(Of course, by writing $I \in \Ool(\Fields)$ I mean that $I$ is a global section of the sheaf $\Ool(\Fields)$ on $X$.  I will often abuse notation in this way). 

Naively, one could say that a quantization of the classical theory is described by a quantum interaction functional $I \in \Ool(\Fields)[[\hbar]]$ satisfying the quantum master equation
$$
Q I + \tfrac{1}{2} \{I,I\} + \hbar \Delta I = 0.
$$
However, the quantum master equation is ill-defined; the expression $\Delta I$ is defined by the multiplication of two distributions which have coincident singularities.

\subsection{}
In \cite{Cos11}, I gave a definition of quantum field theory in the Batalin-Vilkovisky formalism which resolves this difficulty.  The idea of the definition is roughly as follows.  A quantization of the classical field theory described by the classical interaction $I \in \Ool(\Fields)$ will be given by a collection of functionals 
$$
I[L] \in \Oo(\Fields)[[\hbar]],
$$
one for each $L \in (0,\infty)$.   The functional $I[L]$ is called the scale $L$ effective interaction.  If one knows $I[L]$, one can deduce the behaviour of physical phenomena occurring at scales larger than $L$. 

If $\eps < L$, then the functional $I[L]$ can be expressed in terms of $I[\eps]$ by the \emph{renormalization group equation}.  Informally, the renormalization group equation tells us that $I[L]$ is obtained as an average over all fluctuations of the field of with wavelength between $\eps$ and $L$, each counted by $e^{I[\eps] / \hbar}$.  Formally, the renormalization group equation is an expression writing $I[L]$ as a sum over Feynman graphs, with vertices labelled by $I[\eps]$ and edges by the propagator $P(\eps,L)$.   An extensive treatment of this renormalization group flow is given in \cite{Cos11}; a precise definition of the renormalization group flow is reproduced in section \ref{section_rgflow} of this paper. 

In this axiomatic framework, the locality axiom of quantum field theory takes the following form. We require that the functionals $I[L]$, as $L \to 0$, must become more and more local. More precisely, we require that there is a small $L$ asymptotic expansion of the functionals $I[L]$ in terms of local action functionals:
$$
I[L] \simeq \sum f_i(L) I_i
$$
where the $f_i$ are smooth functions of $L$, and $I_i \in \Ool(\Fields)$.    The $L \to 0$ limit won't exist, however, except modulo $\hbar$.

The effective interactions $I[L]$ provide a quantization of the classical field theory described by the classical interaction $I$ if
$$
\lim_{L \to 0} I[L] = I \mod \hbar.
$$

\subsection{}
When we work with gauge theories, or other theories with a homological component, it is essential that our action functionals satisfy the quantum master equation.

In the definition of \cite{Cos11}, the quantum master equation is implemented as follows.  For every scale $L > 0$, there is a scale $L$ BV operator $\Delta_L$, constructed using the heat kernel $K_L$.   Associated to this BV operator is, as usual, a BV bracket $\{-,-\}_L$.   We require that the scale $L$ effective interaction $I[L]$ satisfies the scale $L$ quantum master equation:
$$
Q I[L] + \tfrac{1}{2} \{I[L], I[L]\}_L + \hbar \Delta_L I[L] = 0.
$$ 
Unlike the naive quantum master equation, this equation is well-defined.

The reason this definition works is that the renormalization group equation and the quantum master equation are intimately connected.  If the functionals $\{I[L]\}$ satisfy the renormalization group equation, then $I[L]$ satisfies the scale $L$ quantum master equation if and only if $I[\eps]$ satisfies the scale $\eps$ quantum master equation.  

\subsection{}
Thus, we can summarize the definition of quantum field theory of \cite{Cos11} as follows.
\begin{definition}
Suppose we have a classical action functional $I \in \Ool(\Fields)$ satisfying the classical master equation 
$$
Q I + \tfrac{1}{2} \{ I , I\} = 0.
$$
A \emph{quantization} of the classical field theory defined by $I$ consists of a collection $I[L]$ of effective interactions, satisfying the following properties.
\begin{enumerate}
\item The renormalization group equation expressing $I[L]$ in terms of $I[\eps]$ must hold, whenever $\eps < L$.
\item The $I[L]$ must satisfy a \emph{locality} axiom, saying that as $L \to 0$ the functional $I[L]$ becomes more and more local.
\item The functional $I[L]$ must satisfy the scale $L$ quantum master equation.
\item Modulo $\hbar$, the $L \to 0$ limit of $I[L]$ agrees with the classical action functional $I$.
\end{enumerate}
\end{definition}

One of the main results of \cite{Cos11} is an obstruction-theoretic framework for constructing quantum field theories in this sense.
\begin{theorem}
Let us equip the space $\Ool(\E)$ with the differential $Q + \{I,-\}$.

Suppose we have a quantization $\{I[L]\}$ of our classical field theory, defined modulo $\hbar^{n+1}$.

Then, the obstruction to lifting to a theory defined modulo $\hbar^{n+2}$ is a closed element
$$
O_{n+1} \in \Ool(\E)  
$$
of cohomological degree $1$. 

The simplicial set of lifts of the theory to  one defined modulo $\hbar^{n+2}$ coincides with the simplicial set of cochain homotopies between $O_{n+1}$ and $0 \in \Ool(\E)$.
\end{theorem}
This theorem is proved using the techniques of perturbative renormalization. 

\subsection{}
There is a related result which computes the obstruction-deformation group for translation invariant quantizations on $\R^n$.   Any translation-invariant field theory on $\R^n$ yields a field theory on any $n$-manifold equipped with an affine structure (i.e.\ an atlas where the transition functions are translations). 

Thus, a quantization of Chern-Simons theory on any elliptic curve is determined by a translation-invariant quantization on $\C$.

%% file: f_RG_flow.txt
\section{Feynman graphs and the renormalization group flow}
\label{section_rgflow}
In this section, I will reproduce the precise definition of the renormalization group flow and the weights of Feynman graphs.

If $\gamma$ is a graph, let $T(\gamma)$ denote the set of tails, or external edges, of $\gamma$.  Let $E(\gamma)$ denote the set of internal edges of $\gamma$.  Let $H(\gamma)$ denote the set of half-edges (or germs of edges) of $\gamma$.  Let $V(\gamma)$ denote the set of vertices of $\gamma$.  The vertices of our graphs are labelled by an element $g(v) \in \Z_{\ge 0}$, called the genus of the vertex.

We will view a tail as being a half-edge, so that there is an inclusion $T(\gamma) \into H(\gamma)$. Similarly, if $E_{or}(\gamma)$ denotes the set of internal edges $e \in E(\gamma)$ equipped with an orientation, there is a map $E_{or}(\gamma) \to H(\gamma)$, sending an oriented edge $e$ to the half-edge at the start.    We can identify $H(\gamma)$ as the disjoint union of $T(\gamma)$ with $E_{or}(\gamma)$. 

There is a map $H(\gamma) \to V(\gamma)$, which sends a half-edge to the vertex where it is located.  The fibre over $v \in V(\gamma)$ is the set $H(v)$ of half-edges incident to $v$.   

We will let $g(\gamma)$, the genus of $\gamma$, be the sum
$$
g(\gamma) = b_1(\gamma) + \sum_{v \in V(\gamma)} g(v). 
$$

\subsection{}
For any $I \in \Oo(\Fields)[[\hbar]]$, and any graph $\gamma$, I will describe the Feynman graph weight 
$$
W_\gamma( P(\eps,L) , I) : \Fields^{\otimes T(\gamma)} \to \C.
$$
The renormalization group flow will be defined by summing the weights of graphs. 

Let us expand
$$
I = \sum \hbar^i  I_{i,k},
$$
where $I_{i,k}$ is homogenous of degree $k$ as a functional on $\Fields$.  

In general, for any vector space $V$, we will identify the space of homogeneous polynomials of degree $k$ on $V$ with the space of $S_k$-invariant linear maps $V^{\otimes k} \to \C$,  by the map which sends a polynomial $f$ to the linear map
\begin{align*}
V^{\otimes k} &\to \C \\
v_1 \otimes \cdots \otimes v_k &\mapsto \left( \frac{\partial}{\partial v_1} \dots \frac{\partial}{\partial v_k} f \right) (0).
\end{align*}

By this identification we will view $I_{i,k}$ as an $S_k$-invariant linear map
$$
I_{i,k} : \Fields^{\otimes k} \to \C.
$$

When we define $W_\gamma( P(\eps,L), I)$, every vertex $v$ of genus $g(v)$ and valency $k$ is labelled by $I_{g(v), k}$. We will denote this element by
$$
I_{v}  : \Fields^{\otimes H(v) } \to \C,
$$
where $H(v)$ is the set of germs of edges (or half-edges) of the graph $\gamma$ which are incident to $v$.

Every internal edge $e$ is labelled by the propagator
$$
P_e = P(\eps,L) \in \Fields^{\otimes H(e)}
$$
where $H(e) \subset H(\gamma)$ is the two-element set consisting of the half-edges forming $e$.

Then, we can contract 
$$
\otimes_{v \in V(\gamma)} I_v  : \Fields^{\otimes H(\gamma)} \to \C
$$
with 
$$
\otimes_{e \in E(\gamma)} P_e \in \Fields^{\otimes H(\gamma) \setminus T(\gamma)}
$$
to yield a linear map
$$
W_\gamma(P(\eps,L), I) : \Fields^{\otimes T(\gamma)} \to \C. 
$$
\subsection{} 

Let 
$$
\Oo^+(\E)[[\hbar]] \subset \Oo(\E)[[\hbar]]
$$
be the subspace consisting of those functionals $I$ which are at least cubic when reduced modulo $\hbar$ and modulo the nilpotent ideal $\I$ in our base ring $\A$. 
\begin{definition}
The renormalization group flow operator from scale $\eps$ to scale $L$ is the map
\begin{align*}
\Oo^+(\Fields)[[\hbar]] & \to \Oo^+(\Fields)[[\hbar]] \\
I &\mapsto \EA{P(\eps,L), I   } \defeq \sum_{\gamma}\frac{1}{\abs{\op{Aut}{\gamma}}}  W_\gamma ( P(\eps, L) , I )  \hbar^{g(\gamma)} .
\end{align*}
Thus, a collection
$$
\{I[L]  \in \Oo^+(\Fields)[[\hbar]] \mid L \in \R_{> 0} \}
$$
of functionals satisfies the renormalization group equation if, for all $\eps < L$, 
$$
I[L] = \sum_{\gamma}\frac{1}{\abs{\op{Aut}{\gamma}}}  W_\gamma ( P(\eps, L) , I[\eps]  )  \hbar^{l(\gamma)} ,
$$
where the sum is over all connected graphs as above. 
\end{definition}

\begin{figure}
\label{figure rg flow}
\includegraphics{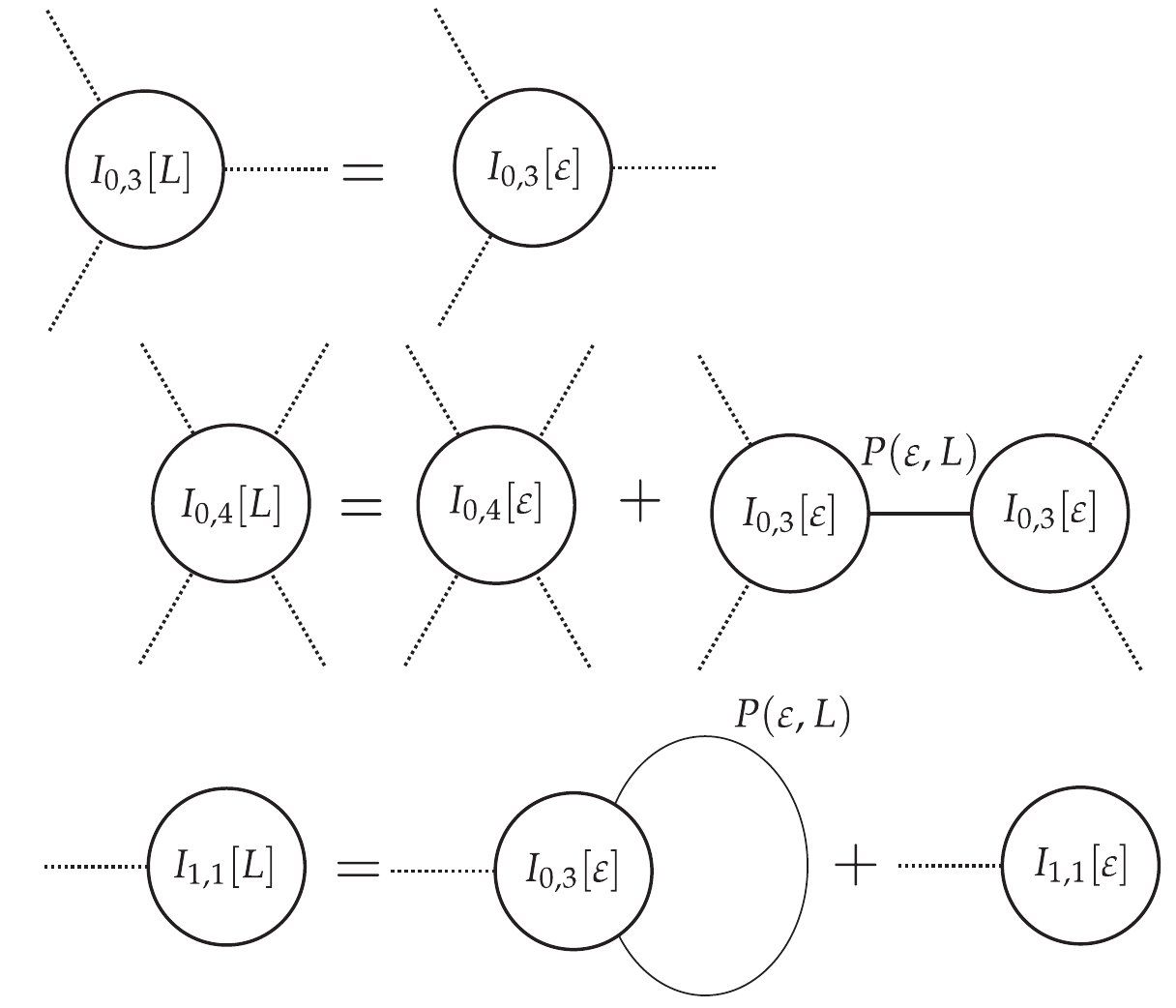}
\caption{The first few terms in the renormalization group flow from scale $\eps$ to scale $L$}
\end{figure}

Finally, I will explain the scale $L$ quantum master equation more precisely.  The heat kernel
$$
K_L \in \Fields \otimes \Fields
$$
is a symmetric element of cohomological degree $1$. Thus, we can define an operator
$$
\Delta_L : \Oo(\Fields) \to \Oo(\Fields)
$$
by contracting with $K_L$.  The operator $\Delta_L$ is an order two differential operator on the commutative algebra $\Oo(\Fields)$, and is the unique continuous order two differential operator with the property that it is zero on the subspace $\A \oplus \Fields^\vee$ of constant and linear functionals on $\Fields$, and, on the space $\Sym^2 \Fields^\vee$ of quadratic functionals, it is given by pairing with $K_L$.  

The operator $\Delta_L$ is of square zero, and commutes with the differential $Q:  \Oo(\Fields) \to \Oo(\Fields)$. 

Let us define a bracket
$$
\{-,-\}_L : \Oo(\Fields) \otimes \Oo(\Fields) \to \Oo(\Fields)
$$
by the formula
$$
\{\Phi, \Psi  \}_L = \Delta_L ( \Phi \Psi)  - (\Delta_L \Phi) \Psi - (-1)^{\abs{\Phi} } \Phi (\Delta_L \Psi) .
$$
This bracket is automatically a derivation in each factor. It follows from the facts that $\Delta_L$ has square zero and commutes with $Q$ that the bracket $\{-,-\}_L$ satisfies the graded Jacobi identity, and that both $Q$ and $\Delta_L$ are derivations for the bracket $\{-,-\}_L$. 

\begin{definition}
Let $\{I[L] \in \Oo^+(\E)[[\hbar]]\}$ be a collection of effective interactions which satisfy the renormalization group equation. We say that they satisfy the quantum master equation if, for all $L$,
$$
Q I [L] + \hbar \Delta_L I[L] + \tfrac{1}{2} \{I[L], I[L] \}_L = 0.
$$
\end{definition}
There is a compatibility between the quantum master equation and the renormalization group equation which implies that if $I[\eps]$ satisfies the scale $\eps$ quantum master equation, then $I[L]$ satisfies the scale $L$ quantum master equation, and conversely.  Thus, it suffices to check the quantum master equation at any scale.

%% file: g_symmetries.txt
\section{Symmetries of holomorphic Chern-Simons theory} \label{section_symmetries}
We are only interested in quantizations which preserve certain additional symmetries, which I will now describe.

Let us give the ring $\C[[\hbar]]$ a $\C^\times$ action, by giving $\hbar$ weight one.   Let the same $\C^\times$ act on the space $\Fields(E)$, by
$$
t \cdot (A + B ) = A + t^{-1} B
$$
if $A \in \Omega^{0,\ast}(E,\g[1])$ and $B \in \Omega^{0,\ast}(E, \g^\vee[-1])$.   

There is an induced action of $\C^\times$ on the space $\Oo(\Fields(E))$ of functionals on fields on $E$.  If $t \in \C^\times$, we will denote this action by
$$
\til{R}(t) : \Oo(\Fields(E) ) \to \Oo(\Fields(E)). 
$$

We will let
$$
R(t) : \Oo(\Fields(\C) ) \to \Oo(\Fields(\C))
$$
be a modified action defined by
$$
R(t) (\Phi) = t^{-1} \til{R}(t) (\Phi).
$$
Finally, let us extend the action $R(t)$ on $\Oo(\Fields(\C))$ to an action on $\Oo(\Fields(\C)) [[\hbar]]$, by giving $\hbar$ weight $1$.

\subsection{}

\begin{lemma}
The following operations are $\C^\times$ invariant.
\begin{enumerate}
\item The renormalization group flow operator of \cite{Cos11},
\begin{align*}
\Oo(\Fields(E)) [[\hbar]] &\to \Oo(\Fields(E) )[[\hbar]]  \\
\Phi &\to  \EA{P(\eps,L), \Phi}.
\end{align*}
\item The differential
$$
Q : \Oo(\Fields(E)) [[\hbar]] \to \Oo(\Fields(E)) [[\hbar]] ,
$$
as well as its quantized version
$$
\what{Q}_L = Q + \hbar \Delta_L.
$$
\item The BV bracket
$$
\{-,-\}_L : \Oo(\Fields(E)) [[\hbar]] \otimes_{\C[[\hbar]] } \Oo(\Fields(E)) [[\hbar]]  \to \Oo(\Fields(E)) [[\hbar]] .
$$
\end{enumerate}
\end{lemma}
A quantization of holomorphic Chern-Simons theory is a collection $\{I[L]\}$ of effective interactions satisfying the renormalization group equation and the quantum master equation.   Since both of these equations are compatible with the $\C^\times$ action, it is meaningful to ask that such a quantization be $\C^\times$-invariant.  This means that each effective interaction $I[L]$ is invariant.  We are only interested in such $\C^\times$-invariant quantizations.  

\begin{lemma}
Let 
$$
I[L] = \sum I^{(i)} [L] \hbar^i \in \Oo(\Fields(E)) [[\hbar]]
$$
be $\C^\times$ invariant.  Then $I^{(k)} = 0$ for $k > 1$.  Further, $I^{(1)} $ lies in the subspace
$$
\Oo (\Omega^{0,\ast} (E, \g[1] ) ) \subset \Oo(\Fields(E)) .
$$
\end{lemma}
\begin{proof}
Indeed, saying that $I[L]$ is $\C^\times$ invariant means that $I^{(i)}[L]$ is of weight $-i$ for the action $R(t)$ of $\C^\times$ on $\Oo(\Fields(E))$. This means that each $I^{(i)}[L]$ is of weight $1 - i$ for the action $\til{R}(t)$.  The action of $\til{R}(t)$ on $\Oo(\Fields(E))$ is induced from an action of $\C^\times$ on $\Fields(E)$, which has only negative weights.  It follows that there are no elements of $\Oo(\Fields(E))$ of negative weight for the action $\til{R}(t)$, and that the only elements of weight $0$ for the $\til{R}(t)$ action are elements functionals on the subspace $\Omega^{0,\ast}(E, \g[1] ) \subset \Fields(E)$. 
\end{proof}

This lemma implies that we only Feynman diagrams with one loop appear when considering $\C^\times$ invariant quantizations of holomorphic Chern-Simons theory.  This makes the task of analyzing possible quantizations far easier. 

\subsection{}
Next, we will define an action of the group $\op{Aff}(\C) = \C \ltimes \C^\times$ of affine linear automorphisms of $\C$ on the space $\Fields(\C)$ of holomorphic Chern-Simons theory.  

The action of $\op{Aff}(\C)$ on 
$$\Fields(\C) = \Omega^{0,\ast}(\C) \otimes \left( \mf g[1]\oplus \mf g^\vee[-1] \right)$$
is induced from the natural action of $\op{Aff}(\C)$ on $\Omega^{0,\ast}(\C)$, and a certain non-trivial action $\rho$ of $\op{Aff}(\C)$ on $\g[1] \oplus \g^\vee[-1]$.  The action $\rho$ is defined by
$$\rho (s,t)(A+ B)=A +  t^{-1}B,$$
if $(s,t) \in \C \ltimes \C^\times$, $A \in \g[1]$ and $B \in \g^\vee[-1]$.

The resulting $\op{Aff}(\C)$ action on $\Fields(\C)$ preserves the holomorphic Chern-Simons action functional, because every term in the action contains one $\d z$ and one field in $\Omega^{0,\ast}(\C, \g^\vee[-1])$.  By naturality, $\op{Aff}(\C)$ acts on $\Oo(\Fields(\C))[[\hbar]]$.

Let 
$$
\op{Isom}(\C) \subset \op{Aff}(\C)
$$
be the subgroup of isometries of $\C$.   This subgroup preserves the propagator
$$
P(\eps,L) \in \Fields(\C) \otimes \Fields(\C).
$$
Therefore, the action of $\op{Isom}(\C)$ on $\Oo(\Fields)^+(\C)[[\hbar]]$ commutes with the renormalization group flow, and also preserves the operator $\Delta_L$ and the bracket $\{-,-\}_L$. 
\begin{definition}
A collection of effective interactions
$$
\{I[L] \in \Oo(\Fields)^+[[\hbar]] \mid L \in \R_{> 0} \} 
$$
satisfying the renormalization group flow is invariant under $\op{Isom}(\C)$ if each $I[L]$ is invariant under the natural action of $\op{Isom}(\C)$ on $\Oo(\Fields)^+[[\hbar]]$.
\end{definition}

\subsection{}
We can write
$$
\op{Aff}(\C) = \op{Isom}(\C) \times \R_{> 0}
$$ 
as a product of the isometry group of $\C$ and the group $\R_{> 0}$, which acts by dilation.   Because the action of $\R_{> 0}$ does not preserve the propagator $P(\eps, L)$, it does not commute with the renormalization group flow.  Thus, it requires some more work to say whether a collection of effective interactions $\{I[L]\}$ is invariant under $\R_{> 0}$.

If $l \in \R_{> 0}$, let
$$
R_l : \Oo(\Fields) \to \Oo(\Fields)
$$
be the map induced from the $\R_{> 0}$ action.

In Chapter 4 of \cite{Cos11}, then 
\begin{lemma}
Suppose that 
$$\{I[L] \in \Oo^+(\Fields)[[\hbar]] \mid L \in \R_{> 0} \}$$
is a collection of effective interactions satisfying the renormalization group flow.  The collection of effective actions $I_l [L]$
$$
I_l[L] = R_l^\ast I [ l^{2} L].
$$
Further, if $\{I[L]\}$ satisfies the quantum master equation, then so does each $I_l [L]$.
\end{lemma}
Thus, the group $\R_{> 0}$ acts on the space of quantum field theories.  This action is called in \cite{Cos11} the local renormalization group flow.   Wilson's concept of renormalizability concerns this renormalization group flow: a theory is renormalizable if it converges to a fixed point under the local renormalization group flow as $l \to 0$.

We say that a collection of effective interactions $\{I[L]\}$ is invariant under $\R_{> 0}$ if, for all $l \in \R_{> 0}$, 
$$
I[L] = R_l^\ast I [ l^{2} L].
$$
Thus, a theory is $\R_{> 0}$ invariant if it is a fixed point under the local renormalization group flow.  

We now can say what it means for a theory to be invariant under the full group 
$$\op{Aff}(\C) = \op{Isom}(\C) \times \R_{> 0}$$
of affine automorphisms of $\C$.
\begin{definition}
Let 
$$\{I[L] \in \Oo^+(\Fields)[[\hbar]] \mid L \in \R_{> 0} \}$$
be a collection of effective interactions satisfying the renormalization group equation.  We say that $\{I[L]\}$ is invariant under $\op{Aff}(\C)$, if each $I[L]$ is fixed by the natural action of $\op{Isom}(\C)$ on $\Oo^+(\Fields)[[\hbar]]$, and if the collection $\{I[L]\}$ is a fixed point of the local renormalization group flow. 
\end{definition}

%% file: h_statement_of_the_main_theorem.txt
\section{Main theorem}
\label{section_main_theorem}
In this section we will precisely state the main theorems of this paper.  The first theorem identifies the space of possible quantizations of holomorphic Chern-Simons theory on $\C$, and thus on any elliptic curve.   The second part identifies the complex of global observables of the resulting field theory on an elliptic curve $E$ in terms of the Witten genus.

\begin{theorem}
The simplicial set of $\C^\times \times \op{Aff}(\C)$-invariant quantizations of holomorphic Chern-Simons theory on $\C$ is weakly equivalent to the simplicial set of trivializations of the cocycle
$$\op{ch}_2 (T B \g) \in  \Omega^2_{cl}(B \g) [-2].$$
\end{theorem}
The quantizations referred to in this theorem are collections of effective interactions $\{I[L]\}$, which satisfy the renormalization group flow and the quantum master equation, are invariant under the group $\C^\times \times \op{Aff}(\C)$ in the sense described earlier, and such that modulo $\hbar$, $I[L]$ converges to the holomorphic Chern-Simons interaction $I_{hCS}$ as $L \to 0$. 

\subsection{}
A quantization of holomorphic Chern-Simons theory on $\C$, invariant under $\C^\times \times \op{Aff}(\C)$, yields a quantization of holomorphic Chern-Simons theory on any elliptic curve $E$.  The next theorem states that the quantized effective action is related to the Witten genus.

Let $\Fields(E)$ be the space of fields for holomorphic Chern-Simons theory on an elliptic curve.  The fact that the effective interactions $I[L]$ satisfy the quantum master equation means that the map
\begin{align*}
\Oo(\Fields(E))[[\hbar]] & \to \Oo(\Fields(E)) [[\hbar]] \\ 
\alpha & \mapsto \what{Q}_L \alpha = Q \alpha + \{ I[L], \alpha \}_L + \hbar \Delta_L \alpha
\end{align*}
is a differential on $\Oo(\Fields(E)) [[\hbar]]$.    The renormalization group equation, which relates the effective interactions $I[L]$ for varying $L$, implies that the complexes $(\Oo(\Fields(E))[[\hbar]], \what{Q}_L)$ are homotopic for different values of $L$.

\subsection{}
Our main theorem shows how these complexes are related to the Witten genus. In order to state this theorem precisely, let us recall the definition of the Eisenstein series.

Let us represent our elliptic curve $E$ as a quotient of $\C$ by a lattice $\Lambda$, in such a way that the volume form $\omega$ on $E$ descends from the form $\d z$ on $\C$.

Let us define the Eisenstein series $E_{2k}(E,\omega)$ of our elliptic curve $E$ with volume element $\omega$  by
$$
E_{2k} (E,\omega) = \sum_{\lambda \in \Lambda \setminus 0} \lambda^{-2k}.
$$
This series is absolutely convergent if $k > 1$.    

Let us define a class
$$
\log \op{Wit}(E,\omega) \in \Omega^{-\ast}( B \g)
$$
by the formula
$$
\log \op{Wit}(X,E,\omega) =  \sum_{k \ge 2} \frac{(2k-1)!}{ (2 \pi i)^{2 k}}  E_{2k} (E,\omega) \op{ch}_{2k} (T B \g).
$$
\begin{theorem}
For any $L$,  there is a quasi-isomorphism of cochain complexes
$$
\left( \Oo(\Fields(E))[[\hbar]] , \what{Q}_L \right) \simeq \left( \Omega^{-\ast}(B \g) [[\hbar]], \hbar \Delta + \hbar \{ \log \op{Wit}(X,E,\omega)   , - \}  \right)
$$
\end{theorem}
This theorem is proved by explicitly calculating $I[\infty]$.

%% file: j_counterterms.txt
\section{First attempt at quantization}
\label{section_counterterms}
Let me outline the strategy for constructing quantum field theories given in \cite{Cos11}.

Given a classical interaction functional -- such as the holomorphic Chern-Simons interaction functional $I_{hCS}$ -- one can try to construction quantum effective interactions $I[L]$ by applying the renormalization group flow from scale $0$ to scale $L$.  That is, one can try to define
$$
I[L] = \EA{P(0,L), I_{hCS}}  = \lim_{\eps \to 0} \EA{P(\eps, L) , I_{hCS} }
$$
However, the famous ultraviolet divergences of quantum field theory say that this limit does not always exist.   

The the technique of counter-terms allows one to solve this problem.  The counterterms will be elements
$$
I_{hCS}^{CT} (\eps) \in \Ool (\E) \otimes_{alg} \cinfty( (0,\infty)_\eps )_{<0} \otimes \C[[\hbar]].
$$
In this expression, $\otimes_{alg}$ denotes the algebraic tensor product, which allows only finite sums.  The space $\Ool(\E)$ is the space of local action functionals, as before, and $\cinfty( (0,\infty)_\eps )_{<0}$ is the space of smooth functions on $(0,\infty)$ which are ``purely singular''.  There is a choice of what it means to be purely singular; in \cite{Cos11} this choice is referred to as the choice of a renormalization scheme.

The results of \cite{Cos11} imply that there is a unique set of counterterms $I_{hCS}^{CT}(\eps)$ with the property that the limit
$$
\lim_{\eps \to 0} \EA{P(\eps, L) , I_{hCS} - I_{hCS}^{CT}(\eps) } \in \Oo(\Fields)[[\hbar]]
$$
exists.

\subsection{}
Once one has constructed the counter-terms $I_{hCS}^{CT}(\eps)$, one defines the first approximation to the quantum effective interaction by
$$
I_{naive}[L] = \lim_{\eps \to 0} \EA{P(\eps, L) , I_{hCS} - I_{hCS}^{CT}(\eps) }.
$$
The sequence of functionals  $I_{naive}[L]$ automatically satisfies the renormalization group equation and the locality axiom.  But, in general, $I_{naive}[L]$ may not satisfy the quantum master equation, and thus may not define a quantum field theory.    In order to turn $I_{naive}[L]$ into a solution to the quantum master equation, one analyzes the possible cohomological obstructions to solving the QME, term by term in $\hbar$.

\subsection{}
For holomorphic Chern-Simons theory, the obstruction analysis will be performed in section \ref{section_obstruction}.  For now, we will only consider $I_{naive}[L]$.   The main result of this section is the following.
\begin{proposition}
On $\C$ or an elliptic curve $E$, the counter-terms $I_{hCS}^{CT}(\eps)$ vanish. Thus, the limit
$$
\lim_{\eps \to 0} \EA{P(\eps, L) , I_{hCS} }
$$
exists. The value of this limit will be denoted by
$$
I_{naive}[L] \in \Oo^+(\E) [[\hbar]].
$$
\end{proposition}
\begin{corollary}
If we work on $\C$, then $I_{naive}[L]$ is invariant under $\C^\times \times \op{Afff}(\C)$.
\end{corollary}
\begin{proof}
Invariance under $\C^\times \times \op{Isom}(\C)$ is immediate, because $\EA{P(\eps,L), I_{hCS}}$ is invariant under this group.

It remains to check invariance under the dilation subgroup $\R_{> 0} \subset \op{Aff}(\C)$.   Invariance under this group follows from the fact that all counterterms are zero. Indeed, it  is shown in chapter 4 of \cite{Cos11} that a theory with no counterterms is a fixed point of the local renormalization group flow. 
\end{proof}

\subsection{}
Thus, to complete the proof of the proposition, we need to verify that the counterterms for holomorphic Chern-Simons theory vanish. As we have seen, the $\C^\times$ symmetry implies that we need only consider one-loop counterterms.  The counterterms are defined by
$$
I_{hCS}^{CT}(\eps) =  \hbar \op{Sing}_{\eps} \sum_{\gamma} W_{\gamma}( P(\eps,L), I_{hCS}) 
$$
where the sum is over all connected graphs $\gamma$ with one loop.  The symbol $\op{Sing}_{\eps}$ refers to the singular part in $\eps$ of the expression on the right hand side. Of course, one needs a choice -- called a \emph{renormalization scheme} in \cite{Cos11} -- to define the singular part.  In this paper, however, the renormalization scheme plays no role, because the counterterms vanish.

We need to show the following.
\begin{lemma}
For all graphs $\gamma$ with one loop, 
$$
\lim_{\eps \to 0} W_{\gamma}( P(\eps,L), I_{hCS}) 
$$
exists. 
\end{lemma}
It follows from this that all counter-terms vanish.  

\subsection{} Let us now turn to the proof of this lemma.   The local nature of counter-terms (proved in \cite{Cos11}) allows us to restrict attention to the case $E = \C$.

Thus, let $\gamma$ be any connected one-loop graph. For each tail $t$ of $\gamma$, let us choose an element
$$
f_t \otimes X_t \in \Omega_c^{0,\ast}(\C) \otimes \mf g[1]
$$
where the space $\Omega_c^{0,\ast}(\C)$ refers to the Dolbeaut complex with compact support.  We need to check that
$$
\lim_{\eps \to 0} W_\gamma (P(\eps,L), I_{hCS}) ( \otimes_{t\in T(\gamma)} f_t \otimes X_t ) 
$$
exists.

The weight $W_\gamma(P(\eps,L) , I_{hCS})$ is constructed from contracting tensors in 
\begin{align*}
\Fields(\C) &=  \Omega^{0,\ast}(\C) \otimes \left( \g[1] \oplus \mf g^\vee[-1]\right) \\ 
&=  \cinfty(\C)  \otimes \left\{ \C[\d \zbar]  \g[1] \oplus \C [\d \zbar] \mf g^\vee[-1]\right\}. 
\end{align*}
Thus, we can write the weight as a product of a combinatorial factor $W_\gamma^{\mf g}(I)$, which arises from contracting tensors in the Lie algebra $\mf g$; with an analytic factor, $W_\gamma^{an}(P(\eps,L), I)$ which arises from contracting tensors in  $\cinfty(\C)$.  The combinatorial factor $W_\gamma^{\g}$ is independent of $\eps$. Thus, in order to check that the $\eps \to 0$ limit exists, we can focus our attention on the analytic factor. 

We say a one-loop graph is a \emph{wheel} if it cannot be disconnected by the removal of a single edge.  Any one-loop graph is a wheel with trees attached to some of the tails.  Since trees can not contribute any singularities, to show that the $\eps \to 0$ limit exists for any graph, it suffices to show that it exists for wheels.  Further, without loss of generality, we need only consider trivalent wheels; it is easy to check that showing the $\eps \to 0$ limit exists for trivalent wheels implies that the limit exist for all wheels. 
\begin{figure}
\label{wheel diagram}
\includegraphics{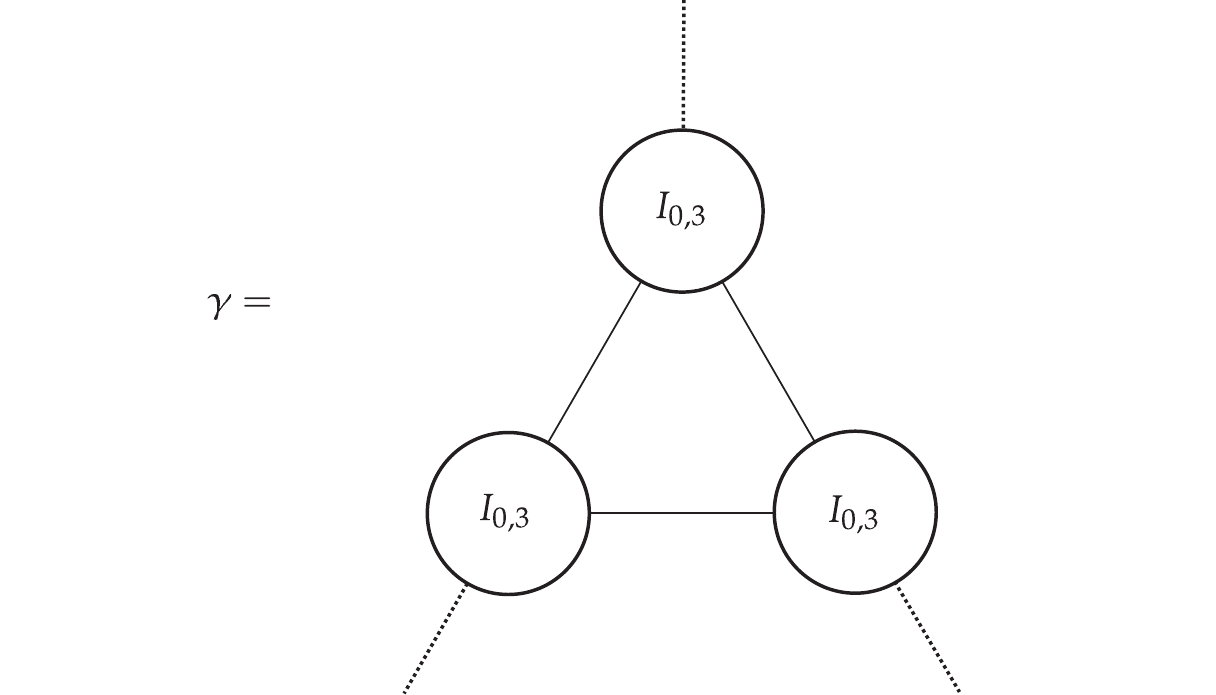}
\caption{The trivalent wheel with three vertices}
\end{figure}

Thus, let us assume that our graph is a trivalent wheel $\gamma_{n}$.   If $f_1,\ldots, f_n \in \cinfty_c(\C)$, we write $W_{\gamma_n}^{an}(P(\eps,L), I_{hCS})(f_1,\ldots,f_n)$ as an explicit integral: 
\begin{multline*}
W^{an}_{\gamma_n} (P(\eps,L), I_{hCS})  (f_1,\ldots, f_n) \\=
\int_{z_1,\ldots, z_n \in \C} \prod_{i = 1}^n \left( f_i(z_i,\zbar_i) P(\eps,L) (z_i, z_{i+1 \op{mod} n}) \prod \d z_i \d \zbar_i  \right). 
\end{multline*}
Here,
$$
P(\eps,L) = \int_{t = \eps}^L \frac{\d}{\d z} K_t \d t
$$
and 
$$
K_t \in \cinfty(\C \times \C)
$$
is the heat kernel for the standard Laplacian operator
$$
D = -\frac{\d^2} { \d z \d \zbar} .
$$
Note that, up to constants,
$$
K_t(z,w) = t^{-1} e^{-\abs{z - w}^2 / t}.
$$ 

To show that the $\eps \to 0$ limit of this integral exists, it suffices to show that for all compactly supported smooth functions $\phi \in \cinfty_c(\C^{n-1} )$, the limit
$$
\lim_{\eps \to 0} \int_{t_i = \eps}^L  \int_{\substack{z_1,\ldots,z_n \in \C \\ \sum z_i = 0 }  } \phi(z,\zbar)  \prod_{i = 1}^n t_i^{-1} \frac{\d}{\d z_i} e^{-\abs{z_i - z_{i+1 \mod n} }^2 / t_i } \d t_i \prod_{i = 1}^{n-1} \d z_i \d \zbar_i
$$ 
exists.  Note that for $n = 1$ the integrand is zero.   Thus, we will assume $n > 1$. 

There are two cases which will be treated separately: $n > 2$ and $n = 2$.  
\subsection{}
We will first consider the case $n > 2$.   It suffices to consider the integral when $t_1 < t_2 < \cdots < t_n$.   Let $u_i = (z_i - z_{i+1})$ for $i = 1,2,\ldots,n-1$.     After this change of coordinates, our integral becomes
$$
 \int_{0 < t_1 < \cdots < t_n  < L} \int_{\substack{u_1,\ldots,u_{n-1} \in \C }  } \psi ( u_i , \br{u}_i  ) \left(  \prod_{i = 1}^{n-1} t_i^{-2} \br{u}_i   e^{-\abs{u_i}^2 / t_i } \d t_i  \d u_i \d \br{u}_i \right)  t_n^{-2} \sum \br{u}_i   e^{- \abs{\sum u_i}^2 / t_n} \d t_n .
$$
We will show that the integral converges absolutely. 

The integral is bounded, in absolute value, by
$$
 \int_{0 < t_1 < \cdots < t_n  < L} \int_{\substack{u_1,\ldots,u_{n-1} \in \C }  } \left( \prod_{i = 1}^{n-1} t_i^{-2} \abs{u_i }  e^{-\abs{u_i}^2 / t_i } \d t_i  \d u_i \d \br{u}_i \right)  t_n^{-2} \sum \abs{u_i}  \d t_n.
$$
Let us further change coordinates, and let $v_i = t_i^{-1/2} u_i$ for $i = 1, \ldots, n-1$. 

After this change of coordinates, we see that integral becomes
$$
 \int_{0 < t_1 < \cdots < t_n  < L} \int_{\substack{v_1,\ldots,v_{n-1} \in \C }  } \left( \prod_{i = 1}^{n-1} t_i^{-1/2} \abs{v_i }  e^{-\abs{v_i}^2  } \d t_i  \d v_i \d \br{v}_i \right)  t_n^{-2} \sum_{i = 1}^{n-1} t_i^{1/2}\abs{ v_i }  \d t_n.
$$
Using the fact that $t_i < t_n$ for $i = 1,\ldots,n-1$, we see that the integral is bounded by
$$
\left(  \int_{0 < t_1 < \cdots < t_n  < L}  \prod_{i=1}^{n-1} t_i^{-1/2} \d t_i t_n^{-3/2} \d t_n \right) \left( \int_{\substack{v_1,\ldots,v_{n-1} \in \C }  }  P( \abs{v_i} ) e^{-\sum \abs{v_i}^2  } \d v_i \d \br{v}_i\right),
$$
where $P$ is some polynomial in the variables $\abs{v_i}$. 

Both integrals in this expression converge absolutely if $n > 2$.

\subsection{}
Let us next consider the case $n = 2$.  Then, we aim to show that limit
\begin{multline*}
\lim_{\eps \to 0} \int_{z_1 + z_2 = 0 } \int_{t_1, t_2 = \eps}^{L} \phi t_1^{-2} t_2^{-2} (\zbar_1 - \zbar_2)^2  e^{ - \abs{z_1 - z_2}^2 \left( t_1^{-1} + t_2^{-1} \right) } \d t_1 \d t_2 \d z_1 \d \zbar_1 \\
= \lim_{\eps \to 0} \int_{u \in \C } \int_{t_1, t_2 = \eps}^{L} \phi(u,\br{u}) t_1^{-2} t_2^{-2} \br{u}^2   e^{ - u \br{u}\left( t_1^{-1} + t_2^{-1} \right) } \d t_1 \d t_2 \d u \d \ubar 
\end{multline*}
exists.

To keep the notation simple, let
$$
\mu = \left(t_1^{-1} + t_2^{-1} \right)^{-1} = \frac{t_1 t_2}{t_1 + t_2}.
$$
We can evaluate
$$
\int_{u \in \C} \phi(u,\ubar) \ubar^2 e^{- u \ubar  \mu^{-1} } \d u \d \ubar
$$
by parts, by observing that
$$
\phi(u,\ubar)  \br{u}^2   e^{ - u \br{u} \mu^{-1}  }  = 
\mu^2 \left( \tfrac{\d^2}{(\d u)^2} \phi(u,\ubar) \right) e^{ - u \br{u} \mu^{-1}  }  + \text{ total derivatives in } u.
$$
If we let 
$$
\phi^{(2)}(u,\ubar) = \frac{\d^2}{(\d u)^2} \phi(u,\ubar)
$$
then we find that we need to show the limit
\begin{multline*}
 \lim_{\eps \to 0} \int_{u \in \C } \int_{t_1, t_2 = \eps}^{L} \phi^{(2)}(u,\ubar) \mu^2 t_1^{-2} t_2^{-2}  e^{ - u \br{u} \mu^{-1} } \d t_1 \d t_2 \d u \d \ubar  \\
=
 \lim_{\eps \to 0} \int_{u \in \C } \int_{t_1, t_2 = \eps}^{L} \phi^{(2)}(u,\ubar) (t_1 + t_2)^{-2}  e^{ - u \br{u} \mu^{-1} } \d t_1 \d t_2 \d u \d \ubar  
\end{multline*}
exists.

We can perform the integral over $u$  using Wick's lemma, to find
$$
\phi^{(2)}(0)  (t_1 + t_2)^{-2} \mu + (t_1 + t_2)^{-2} O (\mu^2) = \phi^{(2)}(0) \frac{t_1 t_2}{ (t_1 + t_2)^3} + (t_1 + t_2)^{-2} O( \mu^2),
$$
where $O(\mu^2)$ indicates an expression tending to zero as fast as $\mu^2$.
 
The limit
$$
\lim_{\eps \to 0} \int_{t_1, t_2 = \eps}^{L} \frac{t_1 t_2}{ (t_1 + t_2)^3} \d t_1 \d t_2 
$$
is easily seen to exist.

%% file: l_obstruction_group.txt
\section{The obstruction complex}
\label{section_obstruction}
In this section we will analyze the complex containing possible obstructions to quantizing holomorphic Chern-Simons theory.  

We have seen that there is an extra $\C^\times$ symmetry present on holomorphic Chern-Simons theory.  The $\C^\times$ invariant effective actions are of the form
$$
I[L] = I^{(0)}[L] + \hbar I^{(1)}[L],
$$
where $I^{(1)}[L]$ is a functional on the summand $\Omega^{0,\ast}(\C,\mf g[1])$ of $\Fields(\C)$.  

It follows that only one-loop obstructions to quantizations can appear, and that the obstruction-deformation complex consists of functionals only on $\Omega^{0,\ast}(\C, \mf g[1])$.   We are only interested in quantizations which are not only $\C^\times$-invariant, but also invariant under the group $\op{Aff}(\C)$ of affine symmetries of $\C$.

We have seen that there is a quantization 
$$
I_{naive}[L] = I^{(0)}[L] + \hbar I_{naive}^{(1)}[L]
$$
which is invariant under all these symmetries, and which satisfies the renormalization group equation, but which may fail to satisfy the quantum master equation.  

The quantum master equation automatically holds modulo $\hbar$. In addition, $\Delta_L I_{naive}^{(1)}[L] = 0$ for all $L$.  The failure of $I_{naive}[L]$ to satisfy the quantum master equation is thus encoded in an obstruction
$$
O[L] =  \Delta_L I^{(0)}[L] + Q I_{naive}^{(1)}[L] + \{ I^{(0)}[L], I_{naive}^{(1)}[L] \} _L.
$$
It was shown in Chapter 5 of \cite{Cos11}  that the family of obstructions $O[L]$ satisfy a renormalization group equation and a locality axiom.  If $\delta$ is a parameter of cohomological degree $-1$, these properties can be summarized by saying that the collection of effective interactions $\{I^{(0)}[L] + \delta O[L]\}$ satisfies the renormalization group equation, the quantum master equation, and the locality axiom, all modulo $\hbar$. 

As explained in Chapter 5 of \cite{Cos11},  these properties imply that the $L \to 0$ limit of $O[L]$ exists, and is a local action functional. We will denote this $L \to 0$ limit by
$$
O \in \Ool( \Omega^{0,\ast} \otimes \g[1] ) ^{\op{Aff}(\C)}.
$$
This obstruction is an element of cohomological degree $1$, and satisfies
$$
Q O + \{I,O\} = 0.
$$
Further, we can replace $I[L]$ by a collection of effective interactions which do solve the quantum master equation if and only if we can make $O$ exact; that is, if and only if we can find some 
$$
J \in \Ool( \Omega^{0,\ast} \otimes \g[1] ) ^{\op{Aff}(\C)}
$$
of cohomological degree $0$, which satisfies the equation
$$
Q J + \{I,J\} = O.
$$

\subsection{}
Thus, in order to construct the quantum theory, we need to first compute the cohomology of the complex $\Ool(\Omega^{0,\ast}(\C, \g[1]))$ of local functionals on $\Omega^{0,\ast}(\C,\g[1])$.  

The main result of this section is the following. 
\begin{proposition}
There is a quasi-isomorphism of cochain complexes
$$
\left( \Ool( \Omega^{0,\ast} (\C) \otimes \g[1] ) ^{\op{Aff}(\C)} , Q + \{I,-\} \right) \simeq \Omega^2_{cl}( B \g) [1]
$$
between the obstruction-deformation complex and the complex of closed $2$-forms on $B \g$, with a shift of one.

Further, the map
$$
\Omega^2_{cl}(B \g)[1] \to \Ool( \Omega^{0,\ast}(\C, \g[1] ) )
$$
arises by a transgression using the holomorphic volume form on $\C$, as explained in the next subsection. 
\end{proposition}
In the next section, we will see that the obstruction class corresponds to a non-zero multiple of 
$$
\op{ch}_2 (T B \g ) \in H^2 \Omega^2_{cl}( B \g).
$$

\subsection{}
Before I prove this proposition, I will describe, geometrically, a map
$$
\Omega^2_{cl}( B \g) [1] \to \Ool( \Omega^{0,\ast} (\C) \otimes \g[1] ).
$$
This map leads to the quasi-isomorphism of the proposition. 

In fact, this map is somewhat simpler to describe if we work on an elliptic curve $E$ rather than on $\C$.

If $E$ is an elliptic curve $\Ool(\Omega^{0,\ast}(E,\g[1]))$ is a subcomplex of the reduced Chevalley-Eilenberg complex of the curved $L_\infty$ algebra $\Omega^{0,\ast}(E,\g[1])$.     We will denote the classifying space of this curved $L_\infty$ algebra  by $(B \g)^{E_{\dbar}}$.

Thus,
$$
\Ool(\Omega^{0,\ast}(\C, \g[1])) \subset \Oo ( (B \g)^{E_{\dbar} } ) / \A = \Omega^1_{cl} ( (B \g)^{E_{\dbar}}).
$$  
We quotient by the ground ring $\A$ because we are considering the reduced Chevalley-Eilenberg complex.  The de Rham differential identifies functions modulo constants with closed one-forms.

There is a natural map
$$
\Omega^2_{cl}(B \g)[1] \to \Omega^1_{cl} ( (B \g)^{E_{\dbar}},
$$
given by transgressing a closed two-form on $B \g$ to a closed one-form on $(B \g)^{E_{\dbar}}$.  The transgression uses the volume element on $E_{\dbar}$, which is of cohomological degree $-1$.

Since this transgression is given by an integral along $E$, it is easy to see that the map factors through the subcomplex
$$
\Ool( \Omega^{0,\ast}(E, \g[1] ) \subset \Omega^1_{cl} ( B \g)^{E_{\dbar}} ).
$$
In fact, it factors through the subcomplex of translation-invariant local functionals in $
\Ool( \Omega^{0,\ast}(E, \g[1] )$, which can be identified with translation invariant local functionals on $\C$.

\subsection{}
Let us now turn to the proof of the proposition.

In \cite{Cos11},  Chapter 5, Section 6,  it was shown how complexes of local action functionals, like $\Ool(\Omega^{0,\ast}(\C) \otimes \g[1])$, can be rewritten in the language of $D$-modules. Let me explain how this applies to the present situation.

We can identify the space of jets of sections of $\Omega^{0,\ast}(\mf g)$ at $0 \in \C$ with the differential graded Lie algebra space
$$
\g[[z,\zbar, \d \zbar]] 
$$
where $\d \zbar$ has cohomological degree $1$, and the differential is the $\dbar$ operator. 

The Lie algebra $\g[[z,\zbar, \d \zbar]]$ is acted on by the abelian Lie algebra $\C\{ \frac{\d}{\d z}, \frac{\d}{\d \zbar} \}$, in the obvious way.  This action is by Lie algebra derivations.  Thus, it extends to an action of $\C\{ \frac{\d}{\d z}, \frac{\d}{\d \zbar} \}$ on the reduced Lie algebra cochain complex $C^\ast_{red}(\g[[z,\zbar,\d \zbar]]$.  (Note that we need to use continuous duals and completed symmetric products in the definition of this Lie algebra cochain complex.  Lie algebra cohomology groups of this form are often called Gel'fand-Fuks cohomology).

Lemma 6.7.1 in chapter 6 implies the following.
\begin{lemma}
There is a quasi-isomorphism of cochain complexes
\begin{multline*}
\left\{ C^\ast_{red} (\mf g[[z,\zbar,\d \zbar]] ) \otimes^{\mbb L}_{\C\left[ \frac{\d}{\d z},\frac{\d}{\d z} \right] } \C \d z \d \zbar \right\}^{\C^\times} \\ 
\simeq \left( \Ool( \Omega^{0,\ast} (\C) \otimes \g[1] ) ^{\op{Aff}(\C)} , Q + \{I,-\} \right) 
\end{multline*}
On the left hand side, we are taking the fixed point for the subgroup $\C^\times \subset \op{Aff}(\C)$.   Here $\C \d z \d \zbar$ refers to the one-dimensional vector space with the trivial action of the Lie algebra $\C\{ \frac{\d}{\d z}, \frac{\d}{\d \zbar} \}$, and equipped with the natural action of the group $\C^\times \subset \op{Aff}(\C)$.
\end{lemma}

\subsection{}
Let us now turn to the computation of the relevant part of the Lie algebra cohomology of $\g[[z,\zbar, \d \zbar]]$. Note that there is a quasi-isomorphism of differential graded Lie algebras
$$
\g[[z,\zbar, \d \zbar]] \simeq \g[[z]].
$$

Thus, we only need to compute
$$
\left\{C^\ast_{red} (\g[[z]] ) \otimes^{\mbb L}_{\C\left[ \frac{\d}{\d z},\frac{\d}{\d z} \right] } \C \d z \d \zbar \right\}^{\C^\times}.
$$
In order to complete the proof of the proposition, it remains to show the following.
\begin{lemma}
There is a quasi-isomorphism
$$
\left\{C^\ast_{red} (\g[[z]] ) \otimes^{\mbb L}_{\C\left[ \frac{\d}{\d z},\frac{\d}{\d \zbar} \right] } \C \d z \d \zbar \right\}^{\C^\times}
\simeq \Omega^2_{cl}(B \g)[1].
$$
\end{lemma}
\begin{proof}

We can compute the derived tensor product  in the complex on the right hand side using a Koszul resolution.   The Koszul resolution of the trivial module $\C$ for $\C\left[ \frac{\d}{\d z},\frac{\d}{\d z} \right]$ is the differential graded algebra
$$
\C\left[ \frac{\d}{\d z},\frac{\d}{\d \zbar}, \eps, \br{\eps} \right]
$$
where $\eps, \br{\eps}$ are in cohomological degree $-1$
with differential
\begin{align*}
\d \eps &= \frac{\d}{\d z} \\
\d \br{\eps} &= \frac{\d}{\d \zbar}
\end{align*}
The generators $\eps, \br{\eps}$ are acted on by $\C^\times$ in the obvious way: $\eps \to \lambda \eps$, $\br{\eps} \to \br{\lambda}\br{\eps}$. 

We find that there is a $\C^\times$ equivariant isomorphism
\begin{multline*}
C^\ast_{red} (\g[[z]] ) \otimes^{\mbb L}_{\C\left[ \frac{\d}{\d z},\frac{\d}{\d z} \right] } \C \d z \d \zbar 
\simeq\\
C^\ast_{red} (\g[[z]] ) \eps \br{\eps} \d z \d \zbar  
\to C^\ast_{red} (\g[[z]] ) \eps  \d z \d \zbar  \oplus C^\ast_{red} (\g[[z]] )\br{\eps} \d z \d \zbar  
\\ \to C^\ast_{red} (\g[[z]] ) \d z \d \zbar.
\end{multline*}
The differential arises from the action of the Lie algebra $\C\{\frac{\d}{\d z} , \frac{\d}{\d \zbar}\}$ on $C^\ast_{red}(\g[[z]])$.  

If we take $\C^\times$ invariants of this Koszul resolution, we find that only the terms with precisely one $\br{\eps}$ remain.  Thus,  we find that we need only compute the $\C^\times$ invariants of 
$$
C^\ast_{red} (\g[[z]])  \eps  \d z  [1] \to C^\ast_{red} (\g[[z]] ) \d z  [1]
$$
(where we have removed the $\C^\times$ invariant expression  $\br{\eps} \d \zbar$ from the notation). 

Note that
\begin{align*}
\left(C^\ast_{red} (\g[[z]] )\eps  \d z \right)^{\C^\times} &= C^\ast_{red}(\g) \\
\left( C^\ast_{red} (\g[[z]] ) \d z \right)^{\C^\times}  &= C^\ast(\g, z^\vee \g^\vee ).
\end{align*}
Further, we can identify $C^\ast_{red}(\g)$ with the two-term complex
$$
\A[1] \to \Oo(B \g)
$$
where $\A$ is our base ring.  Also, we can identify
$$
C^\ast(\g, z^\vee \g^\vee ) = \Omega^1( B \g).
$$
The map
$$
C^\ast_{red}(\g) \to C^\ast(\g, z^\vee \g^\vee )
$$
is the de Rham differential $\Oo(B \g) \to \Omega^1(B \g)$. 

Thus, we have shown that there is a quasi-isomorphism
$$
\left\{C^\ast_{red} (\g[[z]] ) \otimes^{\mbb L}_{\C\left[ \frac{\d}{\d z},\frac{\d}{\d z} \right] } \C \d z \d \zbar \right\}^{\C^\times} 
\simeq \A[3] \to \Oo(B \g) [2] \to \Omega^1(B \g) [1].
$$
The complex on the right hand side is quasi-isomorphic, via the de Rham differential, to the complex
$$
\Omega^2_{cl}(B \g)[1] \simeq \Omega^2(B \g)[1] \to \Omega^3(B \g) \to \Omega^4(B \g) [-1] \to \cdots
$$
as desired.

It is not hard to check that the resulting map
$$
\Omega^2_{cl}( B \g) [1] \to \Ool(\Omega^{0,\ast}(\C, \g) [1] ) 
$$
is the transgression map described earlier. 
\end{proof}

%% file: m_obstruction_graph.txt
\section{Calculation of the obstruction}
\label{section_obstruction_class}
So far we have seen that the obstruction-deformation complex for our field theory is $\Omega^2_{cl}(B \g)[1]$, the complex of closed $2$-forms on $B \g$, with a shift of one.

It remains to identify the actual obstruction.   

For very formal reasons, one can tell that the obstruction must be a multiple of $ch_2( T_{B\g})$. Indeed, the obstruction is additive under direct sum of Lie algebras, and lives in
$$
H^2 ( \Omega^2_{cl}(B \g) ).
$$
The only characteristic classes with these properties are multiples $ch_2(T_{B \g})$.

However, this is not the approach we will take: instead we will compute the obstruction directly, using Feynman diagrams.  The calculation is not hard.  However, as with all diagrammatic calculations, it is difficult to explain to those who have not worked with these techniques.  
\begin{theorem}
The obstruction 
$$
O \in H^2( \Omega^2_{cl}(B \g)) 
$$
is a non-zero multiple of $ch_2 (T_{B \g})$, the second Chern character of the tangent bundle to $B \g$. 
\end{theorem}

\begin{proof}
Recall that we defined $I_{naive}[L]$ as the limit
$$
I_{naive}[L] = \lim_{\eps \to 0} \EA{P(\eps,L) , I_{hCS} }
$$
where $I_{hCS}$ is the classical holomorphic Chern-Simons theory interaction.  

The obstruction to solving the quantum master equation, at scale $L$, is 
$$
O[L] = \hbar^{-1} \left\{ Q I_{naive}[L] + \tfrac{1}{2} \{ I_{naive}[L], I_{naive}[L] \}_L + \hbar \Delta_L I_{naive}[L] \right\}.
$$
(The obstruction automatically vanishes modulo $\hbar$, and has no $\hbar^2$ or higher contributions).

The obstruction satisfies the classical master equation, and has an $L \to 0$ limit which we denote by
$$
O = \lim_{L \to 0} O[L] \in \Oo_l (\E(\C)). 
$$
\begin{lemma}
The obstruction $O[L]$ is
$$
O [L] =  \lim_{\eps \to 0} 
\frac{\d}{\d \delta}  \EA { P(\eps,L), I_{hCS}  + \delta\left\{ \tfrac{1}{2} \{ I_{hCS}, I_{hCS} \}_{\eps} - \tfrac{1}{2} \{I_{hCS},I_{hCS}\}_0  \right\} }.  $$
Here $\delta$ is a parameter of cohomological degree $-1$ square zero. 
\end{lemma}
\begin{proof}
The compatibility between the renormalization group flow and the quantum master equation implies that, for all functionals $I$,
\begin{multline*}
 Q \EA{P(\eps, L),I } + \tfrac{1}{2}  \{ \EA{P(\eps, L),I } , \EA{P(\eps, L),I } \}_L +  \hbar \Delta_L \EA{P(\eps, L),I }  =  \\
\frac{\d}{\d \delta} \EA { P(\eps,L) , I + \delta\left\{ Q I + \tfrac{1}{2} \{ I, I \}_\eps + \hbar \Delta_\eps I  \right\} } 
\end{multline*}
where $\delta$ is a parameter of cohomological degree $-1$ (and square zero).

Thus, the obstruction $O[L]$ at scale $L$ satisfies 
$$
 O[L] = \hbar^{-1} \lim_{\eps \to 0}\frac{\d}{\d \delta}  \EA { P(\eps,L), I_{hCS} + \delta\left\{ Q I + \tfrac{1}{2} \{ I_{hCS}, I_{hCS} \}_\eps + \hbar \Delta_\eps I_{hCS}  \right\} }.
$$

Also, for all $\eps > 0$,
$$
\Delta_\eps I_{hCS} = 0.
$$
This follows from the expression
$$
K_t  = t^{-1} e^{-\norm{z-w}^2 / t} (\d \br{z} \otimes 1 - 1 \otimes \d \br{w} ) \otimes C_\g
$$
where $C_\g \in (\left( \g[1] \oplus \g^\vee[-1]\right)^{\otimes 2}$. 

In addition, the classical master equation asserts that
$$
Q I_{hCS} = - \tfrac{1}{2} \{ I_{hCS}, I_{hCS} \}_0
$$
where $\{-,-\}_0$ denotes the scale $0$ bracket.

Thus, we see that
$$
 Q I_{hCS}  + \tfrac{1}{2} \{ I_{hCS}, I_{hCS} \}_\eps + \hbar \Delta_\eps I_{hCS} = \tfrac{1}{2} \{ I_{hCS}, I_{hCS} \}_\eps - \tfrac{1}{2} \{ I_{hCS}, I_{hCS} \}_0.
$$
\end{proof}

Let $\gamma$ be a graph, and let $e$ be an edge of $\gamma$. Let us assume that the edge $e$ is not a loop. Let
$$
\EAs{\gamma,e} { P(\eps,L) ,K_\eps - K_0,  I_{hCS}  } \in \Oo(\E(\C))
$$
be obtained by putting the propagator $P(\eps,L)$ at all edges of $\gamma$ except for $e$, and by putting $K_\eps - K_0$ at the edge $e$.

We can write 
\begin{multline*}
 \hbar^{-1} \frac{\d}{\d \delta}  \EA { P(\eps,L), I_{hCS}  + \delta\left\{ \tfrac{1}{2} \{ I_{hCS}, I_{hCS} \}_{\eps} - \tfrac{1}{2} \{I_{hCS},I_{hCS}\}_0  \right\} }
\\ = \sum_{\gamma} \frac{1}{\op{Aut}(\gamma)} 
\EAs{\gamma,e} { P(\eps,L) ,K_\eps - K_0,  I_{hCS}  }
\end{multline*}
as a sum over one-loop graphs $\gamma$, equipped with an edge $e$ which is not a loop.

The obstruction is the limit of this sum as $\eps \to 0$.  
\begin{lemma}
Let $\gamma$ be a one-loop graph, and $e$ an edge of $\gamma$ which is not a loop.   Then,
$$
\lim_{\eps \to 0} \EAs{\gamma,e} { P(\eps,L) ,K_\eps - K_0,  I_{hCS}  } = 0
$$
unless the edge $e$ is contained in a wheel with precisely two vertices (in other words, unless the two vertices $v_1,v_2$ which are connected by $e$ are also connected by a single other edge).
\end{lemma}
\begin{proof}
This is a direct computation.

First, let us suppose that the edge $e$ is separating. Then 
$$
\lim_{\delta \to 0} \lim_{\eps \to 0} \EAs{\gamma,e} { P(\eps,L) ,K_\delta ,  I_{hCS}  } = \lim_{\eps \to 0} \lim_{\delta \to 0} \EAs{\gamma,e} { P(\eps,L) ,K_\delta ,  I_{hCS}  }.
$$
This is simply because the edge $e$ is part of a tree, and trees never contribute anything singular.

This means that
$$
\lim_{\eps \to 0} \EAs{\gamma,e} { P(\eps,L) ,K_\eps ,  I_{hCS}  } = \lim_{\eps \to 0} \EAs{\gamma,e} { P(\eps,L) , K_0,  I_{hCS}  } = 0
$$
so that the desired limit is zero.

Next, let us assume that $e$ is part of a wheel with at least three vertices.  Since $\gamma$ is a one-loop graph, it is a wheel with some trees attached.  We can, as usual, ignore the contributions of this trees.  Thus, let us assume that $\gamma$ is a wheel with three or more vertices.

Then, one can verify by an easy direct computation that
$$
\lim_{\delta \to 0} \lim_{\eps \to 0} \EAs{\gamma,e} { P(\eps,L) ,K_\delta ,  I_{hCS}  } = \lim_{\eps \to 0} \lim_{\delta \to 0} \EAs{\gamma,e} { P(\eps,L) ,K_\delta ,  I_{hCS}  }.
$$
Again, this implies that
$$
\lim_{\eps \to 0} \EAs{\gamma,e} { P(\eps,L) ,K_\eps - K_0,  I_{hCS}  } = 0.
$$
\end{proof}

\begin{corollary}
The obstruction $O = \lim_{L \to 0} O[L]$ can be written as a sum
$$
O = \sum_{\gamma,e} \lim_{\eps \to 0} \EAs{\gamma,e} { P(\eps,1) ,K_\eps ,  I_{hCS}  } 
$$
where the sum is over one-loop graphs which are wheels with two vertices.
\end{corollary}
\begin{proof}
Indeed, the previous lemma implies that the contribution to the obstruction coming from graphs which contain a wheel with more than three vertices vanishes.  

Thus, the obstruction $O[L]$ is 
$$
O[L] = \sum_{\gamma,e} \lim_{\eps \to 0} \EAs{\gamma,e} { P(\eps,L) ,K_\eps ,  I_{hCS}  } 
$$
where the sum is over one-loop graphs which contain a wheel with two vertices, possibly with trees attached on the outside.

Now, if $\gamma$ is a wheel with two vertices, one can check that
$$
\lim_{\eps \to 0} \EAs{\gamma,e} { P(\eps,L) ,K_\eps ,  I_{hCS}  } 
$$
is independent of $L$.   If $\gamma$ is a wheel with two vertices and some trees attached, the limit
$$
\lim_{L \to 0} \lim_{\eps \to 0} \EAs{\gamma,e} { P(\eps,L) ,K_\eps ,  I_{hCS}  } 
$$
is easily seen to be zero.  
\end{proof}
If $\gamma$ is a wheel with two vertices, and $e$ is one of the two edges of $\gamma$, we will let $O_{\gamma,e}$ be the part of the obstruction coming from $\gamma$.  Thus, 
$$O_{\gamma,e} = \lim_{\eps \to 0} \EAs{\gamma,e} { P(\eps,1) ,K_\eps ,  I_{hCS}  }.$$ 
\begin{figure}
\includegraphics{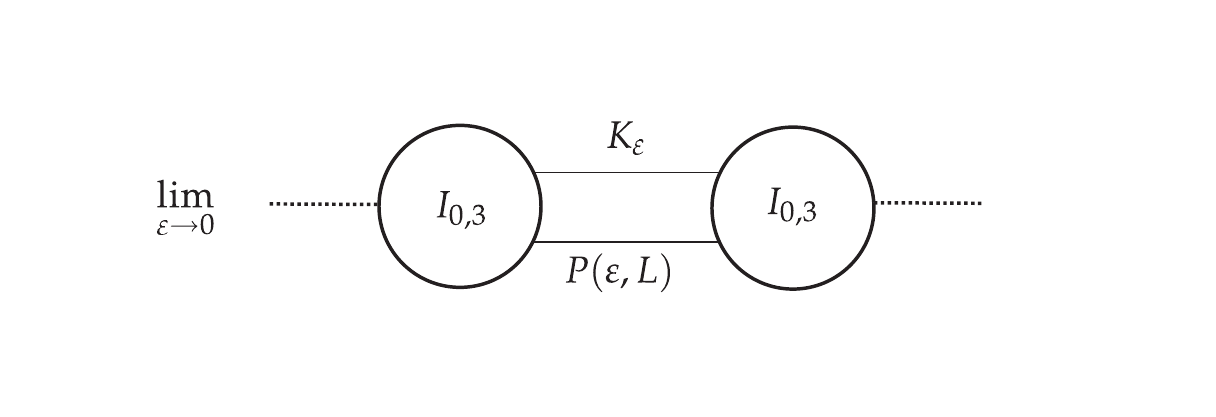}
\label{diagram obstruction}
\caption{The obstruction for the trivalent wheel with two vertices}
\end{figure}

We will view $O_{\gamma,e}$ as a linear map
$$
O_{\gamma,e} : \left( \Omega^{0,\ast}_c(\C) \otimes\g \right)^{\otimes T(\gamma)} \to \C.
$$
Thus, we can view $O_{\gamma,e}$ as a product of an analytic factor
$$
O_{\gamma,e}^{an} : \Omega^{0,\ast}_c(\C)^{\otimes T(\gamma)} \to \C
$$
and a Lie algebra factor
$$
O_{\gamma,e}^{\g} : \g^{\otimes T(\gamma)} \to \C.
$$
We will deal with the Lie algebra factor momentarily.  Let us first consider the analytic factor.
\begin{lemma}
Suppose $\gamma$ has $k+l$ tails, with $k$ attached to one vertex and $l$ to the other. Let $f_1,\ldots, f_k$ and $g_1,\ldots, g_l$ be compactly supported smooth functions on $\C$. Suppose that $f_1 \d \zbar, f_2, \ldots f_k$ are inputted at the first $k$ tails and that $g_1, g_2, \ldots g_l$ at the remaining $l$ tails. 

Then,
$$
O_{\gamma,e}^{an} (f_1 \d \zbar, f_2,\ldots, f_k, g_1, \ldots, g_l ) =c \int_{z \in \C} \left( \frac{\d}{\d z} \prod f_i\right) \prod g_j.$$
for some non-zero constant $c \in \C$.
\end{lemma}
\begin{proof}
Ignoring various factors of $\pi$ and combinatorial factors, 
\begin{multline*}
O_{\gamma,e}^{an} (f_1 \d \zbar, f_2,\ldots, f_k, g_1, \ldots, g_l ) = \\
 \lim_{\eps \to 0} \int_{t = \eps}^1 \int_{z, w \in \C}  \prod f_i(z) \prod g_j (w) \eps^{-1} e^{- \abs{z - w}^2 / \eps} \frac{\d}{\d z} t^{-1} e^{-\abs{z - w}^2 / t } \d t \d z \d \zbar \d w \d \br{w}. 
\end{multline*}
The integral over $z,w \in \C$ can be rewritten as
$$
- \int_{z,w \in \C}  \prod f_i ( z)  \prod g_j (w)( \br{z-w}) \eps^{-1}t^{-2} e^{ - \abs{z-w}^2 \eps^{-1} t^{-1} \left( \eps + t \right) } .
$$
Noting that 
$$
(\br{z-w} )\eps^{-1}t^{-2} e^{ - \abs{z - w}^2 \eps^{-1} t^{-1} \left( \eps + t \right) } = - \frac{1}{t (  \eps + t) } \frac{\d}{\d z } e^{ - \abs{z-w}^2 \eps^{-1} t^{-1} \left( \eps + t \right) } 
$$
allows us to rewrite this integral, using integration by parts, as 
$$
- \int_{z,w \in \C} \left( \frac{\d}{\d z}\prod f_i (z ) \right)  \prod g_j (w)  \frac{1}{t(\eps + t)}  e^{ - \abs{z-w}^2 \eps^{-1} t^{-1} \left( \eps + t \right) } .
$$
We can expand this integral using Wick's lemma: the leading term is (up to a non-zero factor)
$$
\int_{z \in \C} \left( \frac{\d}{\d z}\prod f_i (z ) \right)  \prod g_j (z)  \frac{\eps}{ (\eps + t)^2} .
$$
The final step is to verify that
$$
\lim_{\eps \to 0} \int_{\eps}^1 \frac{\eps}{(\eps + t)^2} \d t 
$$
is non-zero; which is immediate.
\end{proof}

As we will see in Section \ref{section_witten_genus}, we can identify the Lie algebra contribution to the weight of wheels with two vertices as being $\op{ch}_2$ of the tangent bundle to $B \g$.    Putting this observation together with the computation of the analytic component of $O_{\gamma,e}$ we see the obstruction is non-zero at the cochain level, and is constructed from $\op{ch}_2$.   It is not hard to check that the obstruction is a non-zero multiple of the image of $\op{ch}_2$ under the transgression map
$$
\Omega^2_{cl}(B \g)[1] \to \Ool(\Omega^{0,\ast}(\C) \otimes \g[1] ). 
$$ 
Since this transgression map is a quasi-isomorphism, the we see that the obstruction is a non-zero multiple of the second Chern character, thus proving the theorem. 
\end{proof}

\subsection{} 
This theorem implies that if we choose a trivialization of $ch_2(T_{B \g})$ then we find a quantization $I[L]$ of holomorphic Chern-Simons theory on $\C$ or any elliptic curve $E$,  invariant under $\C^\times \times \op{Aff}(\C)$.  This quantization is of the form
$$
I[L] = I_{naive}[L] + \hbar J[L]
$$
where
$$
J[L] \in \Oo(\Omega^{0,\ast}(E) \otimes \g[1] ).
$$
The term $J[L]$ is the correction to the failure of $I_{naive}[L]$ to satisfy the quantum master equation.

%% file: p_eisenstein.txt
\section{The Witten genus}
\label{section_witten_genus}
In this section, we will complete the calculation, and see that the Witten genus is encoded in $I[\infty]$. In order to state the precise calculation, we need to introduce a little notation.  Let 
$$
\mc H(E) \subset \Fields(E) = \Omega^{0,\ast}( E ) \otimes \left( \g[1] \oplus \g^\vee[-1] \right)
$$
denote the subspace of harmonic fields, that is, those in the kernel of both $\dbar$ and $\dbar^\ast$. 

\subsection{}

Let $(E,\omega)$ be an elliptic curve equipped with a holomorphic volume form. 

Serre duality gives rise to a trace map  
$$\op{Tr}_{\omega} : H^1(E, \Oo_E) \to \C.$$   
 We will identify  $H^1(E, \Oo_E)$ with the Dolbeaut cohomology group of $E$.  In these terms, the trace map arises from the map
\begin{align*}
\Omega^{0,1} (E) &\to \C \\
\alpha &\mapsto \int_{E} \omega \wedge \alpha.
\end{align*}
Let
$$
\omega^\vee \in H^1(E,\Oo_E)
$$ 
be the element such that 
$$
\op{Tr} (\omega^\vee) = 1.
$$

Let $\eps$ be a parameter of cohomological degree $1$.  Let us define an isomorphism of graded algebras
\begin{align*}
\C[\eps] & \to H^\ast( E, \Oo_E) \\
\eps  &\mapsto \const \omega^\vee.
\end{align*}

\subsection{}
Let $\mc H(E) \subset \Fields(E)$ be the subspace of harmonic fields.  Note that the pairing on $\Fields(E)$ restricts to an odd symplectic pairing on $\mc H(E)$.   The isomorphism $H^\ast(E,\Oo_E) \iso \C[\eps]$ leads to an isomorphism
$$
\mc H(E) \iso \C[\eps] \otimes \left( \g [1] \oplus \g^\vee[-1] \right) .
$$

The scale $\infty$ effective interaction $I[\infty]$ restricts to a solution to the quantum master equation on $\mc H(E)$.  Further, the one-loop part $I^{(1)}[\infty]$ comes from a functional on the space $\C[\eps] \otimes \g[1]$.

Note that
$$
C^\ast_{red}( \C[\eps] \otimes \g[1] )  = C^\ast_{red}(\g, \Sym^\ast \g^\vee ) = \Omega^{-\ast} (B \g ) ,
$$
where the right hand side is not equipped with the de Rham differential, just with the ``internal'' differential which preserves each space $\Omega^k(B \g)$. 

Thus, we can view the one-loop part of the scale $\infty$ effective interaction as
$$
I^{(1)} [\infty] \in \Omega^{-\ast} (B \g).
$$
\subsection{}
The finite dimensional space $\mc H(E)$ has an odd symplectic pairing. Thus, $\Oo(\mc H(E))$ has a BV operator $\Delta_{\mc H(E)}$ and a BV bracket $\{-,-\}$.   Further, if we equip $\Oo(\Fields(E))$ with the BV bracket $\Delta_\infty$ at scale $\infty$, the map
$$
\Oo(\Fields(E)) \to \Oo( \mc H(E))
$$
is a map of BV algebras, that is, it is compatible with the operators $\Delta_{\mc H(E))}$ on the right ant $\Delta_\infty$ on the left.  This is simply because the BV operator $\Delta_\infty$, when viewed as an element of $\Fields(E)^{\otimes 2}$, actually lies in the subspace $\mc H(E)^{\otimes 2}$.  

The inclusion $\C[\eps] \into \Omega^{0,\ast}(E)$ is a quasi-isomorphism. It follows that the map $\Oo(\Fields(E)) \to \Oo(\mc H(E))$ gives a quasi-isomorphism
\begin{multline*}
\left( \Oo\left( \C[\eps] \otimes \left( \g [1] \oplus \g^\vee[-1]   \right) \right)[[\hbar]],  \hbar \Delta_{\mc H(E)} + \{ I[\infty] \mid_{\mc H(E)}, -  \} \right)  \\
\simeq 
\left( \Oo\left( \Fields(E) \right)[[\hbar]], Q + \hbar \Delta_{\infty} + \{ I[\infty], -  \} \right).
\end{multline*}

\begin{lemma}
There is an isomorphism of cochain complexes
\begin{equation*}
\left( \Oo\left( \C[\eps] \otimes \left( \g [1] \oplus \g^\vee[-1]   \right) \right), \{ I^{(0)} [\infty] \mid_{\mc H(E)}, -  \} \right) \iso \Omega^{-\ast}( T^\ast_{B \g} ) \tag{$\dagger$} \label{eqn_dagger}
\end{equation*}
where $T^\ast _{B \g}$ refers to the formal completion at zero of the cotangent bundle to $B \g$, and the algebra of forms is equipped only with the internal differential, and not the de Rham differential.  

Further,  this isomorphism takes the BV operator $\Delta_{\mc H(E)}$ to the operator 
$$L_\pi : \Omega^{i} (T^\ast _{B\g} ) \to \Omega^{i-1} (T^\ast_{B \g} ) $$
given by Lie derivative with the Poisson tensor $\pi$ on $T^\ast_{B\g}$.
\end{lemma}
\begin{proof}
The solution of the classical master equation $I^{0}[\infty]$ in $ \Oo\left( \C[\eps] \otimes \left( \g [1] \oplus \g^\vee[-1]  \right) \right)$ can be interpreted as giving the space $\C[\eps] \otimes \left( \g \oplus \g^\vee [-2] \right)$ the structure of an $L_\infty$ algebra, compatible with the pairing of degree $-3$ defined by combining the obvious pairing on $\g \oplus \g^\vee[-2]$ with the trace map
\begin{align*}
\op{Tr} : \C[\eps] &\to  \C \\
\op{Tr} ( \eps ) &= 1.
\end{align*}
The $L_\infty$ structure given by $I^{0}[\infty]$ is obtained by transfer of structure from the $L_\infty$ structure on $\Omega^{0,\ast} (E) \otimes \left( \g \oplus \g^\vee[-2] \right)$.  Indeed, the Feynman diagrams describing $I^{0}[\infty]$ are precisely the trees appearing in the explicit formula \cite{Mer99,KonSoi01} for the homological perturbation lemma.  Let us call this $L_\infty$ structure $A$.

Further, the complex $\Omega^{-\ast} ( T^\ast B \g )$ is, by definition, the cochains of the $L_\infty$ algebra $\C[\eps] \otimes \left( \g \oplus \g^\vee[-2]\right)$, when this is endowed with the $L_\infty$ structure arising from the tensor product of the given $L_\infty$ structure on $\g \oplus \g^\vee[-2]$ with the commutative algebra structure on $\C[\eps]$.    Let us call this $L_\infty$ structure $B$.

Thus, in order to verify the equation ($\dagger$), we need to verify that the $L_\infty$ structures $A$ and $B$ on $\C[\eps] \otimes \left(\g \oplus \g^\vee[-2] \right)$ coincide.   Note that the $L_\infty$ structure $A$ is given by a sum over trees;  the terms in the sum given by trees with no edges yield $L_\infty$ structure $B$.

Thus, we need to verify that the terms in the expansion of $I^{0}[\infty]$ which involve trees with at least one edge all vanish.    For this computation, the Lie algebra will be irrelevant; we will check that the analytic part of the weight attached to any such tree vanishes.

 The sum-over-trees formula involves putting harmonic elements of $\Omega^{0,\ast}(E)$ at each tail of the tree, and putting the propagator 
$$P(0,\infty) = \int_{0}^\infty (\dbar^\ast \otimes 1) K_t$$
at each edge.  Let us isolate the contribution from a single vertex $v$ which has some non-zero number of tails, and a single internal edge.  All trees have at least one such vertex.  The  tails of $v$ are labelled by harmonic elements $h_1,\ldots, h_k \in \Omega^{0,\ast}(E)$.  We can express the weight of the tree as
$$
\int_{t = 0}^\infty \int_{z,w \in E}  \left( h_1(z)  \dots h_k (z) \right)  (\dbar^\ast \otimes 1) K_t(z,w) \Phi(w)
$$   
where $\Phi(w)$ captures the contribution from the rest of the tree.  

Note that the product of harmonic elements of $\Omega^{0,\ast}(E)$ remains harmonic.   Further, by integration by parts, we can (at the price of a sign) move the $\dbar^\ast$ in the above expression so that it acts on $h_1(z) \dots h_k(z)$.  Since $\dbar^\ast (h_1(z) \dots h_k(z) ) = 0$, the integral vanishes, as desired.  

The same argument shows that the weight of any $1$-loop diagram which contains a separating edge  also vanishes. 

Thus, we have checked the quasi-isomorphism ($\dagger$).

Next, we need to verify that the operator $\tr_{\mc{H}(E)}$ corresponds, under this isomorphism, to the operator $L_\pi$.  

Recall that we can identify the tangent bundle to $T^\ast B \g$ with the $\g \oplus \g^\vee[-2]$-module $\g [1] \oplus \g^\vee [-1]$.  The Poisson bivector on the symplectic manifold $T^\ast B \g$ is therefore some element
$$
\pi \in C^\ast ( \g \oplus \g^\vee[-2],  \wedge^2 ( \g[1] \oplus \g^\vee[-1] ) )
$$
of cohomological degree $0$.  This tensor $\pi$ is, in fact, in the subspace
$$
\pi \in \wedge^2 ( \g[1] \oplus \g^\vee[-1]).
$$
In terms of a local basis $V_i$ of sections of $\g$ (and corresponding dual basis $V_i^\vee$ of $\g^\vee$), $\pi$ is given by the formula
$$
\pi = V_i \otimes V_i^\vee + V_i^\vee \otimes V_i.
$$

Note that the BV operator $\tr_{\mc{H}(E)}$ is the order two differential operator on the dga $C^\ast \left( \C[\eps] \otimes \left( \g \oplus \g^\vee[-2] \right) \right)$ associated to the kernel
$$
K_\infty \in \Sym^2 \mc{H}(E)
$$
$K_\infty$ is simply the inverse to the natural non-degenerate pairing on $\mc{H}(E)$, and can be written, once we identify 
$$
\mc{H}(E) = \C[\eps] \otimes \left( \g[1] \oplus \g^\vee[-1] \right),
$$
as 
$$
K_\infty = \left( \eps \otimes 1 - 1 \otimes \eps \right) \left(  \sum V_i \otimes V_i^\vee + V_i^\vee \otimes V_i.   \right) 
$$
In other words, 
$$K_\infty = \left(  \eps \otimes 1 - 1 \otimes \eps \right) \pi. $$

In terms of the local basis $V_i$ as above, the operator $\tr_{\mc{H}(E)}$ is the constant-coefficient differential operator associated to $K_\infty$, and so is given by the formula
$$
\tr_{\mc{H}(E)} = \sum \dpa{ \eps V_i  } \dpa { V_i^\vee }  +  \dpa{  V_i  } \dpa { \eps V_i^\vee } .
$$

Next, we need to compare this to the operator $L_\pi$.  Recall that $\Omega^{-\ast}( T^\ast B \g)$ is generated by the dual of $ \g [1] \oplus \g^\vee[-1]$, which corresponds to the generators of $\Oo( B \g)$; and by the dual of $\eps \left( \g [1] \oplus \g^\vee[-1]\right)$, which corresponds to a basis over $\Oo(B \g)$ of $\Omega^{1} (T^\ast B \g)$.  

Now, the operator $L_\pi$ is the commutator of the contraction operator $\iota_\pi$ with the de Rham differential.    The operator $\iota_\pi$ is given by 
$$
\iota_\pi = \sum \dpa{ \eps V_i  } \dpa { \eps V_i^\vee }  .
$$
The de Rham differential $\d_{dR}$ is the operator associated to the $L_\infty$ derivation  map
$$
\dpa{\eps}  :  \C[\eps] \otimes \left( \g  \oplus \g^\vee[-2] \right)  \to \C[\eps] \otimes \left( \g  \oplus \g^\vee[-2] \right) .
$$
Thus,
\begin{align*}
\left[\d_{dR},   \dpa{ \eps V_i  }   \right] &= \dpa{V_i}\\
\left[\d_{dR},   \dpa{ \eps V^\vee_i  }   \right] &= \dpa{V^\vee_i}.
\end{align*}
It is immediate now that
$$
[\iota_\pi, \d_{dR} ] = \tr_{\mc{H}(E)}
$$
as desired. 
\end{proof}

\subsection{}
Thus, to complete the proof of the theorem, we need to show that $I^{(1)}[\infty]\mid_{\mc H(E)}$ corresponds to the Witten class.  

Recall that we defined the Atiyah class 
$$
\alpha = \alpha( T_{B \g}) \in \Omega^1 (B \g, \op{End} T_{B \g}).
$$
This is an element of cohomological degree $1$.   This Atiyah class is an avatar for the curvature. As usual, we can define the trace of powers of $\alpha$,
$$
\op{Tr} (\alpha)^k \in \Omega^k (B \g ).
$$
This is an element of cohomological degree $k$, and 
$$
\frac{1}{ k! (2 \pi i )^k  } \op{Tr} (\alpha)^k =  ch_{2k}  (T_{B \g}) \in H^k ( \Omega^k (B \g)).
$$
\begin{theorem}
The restriction of $I^{(1)}[\infty]$ to $\C[\eps] \otimes \g[1]$ is cohomologous to
\begin{align*}
 \sum_{k \ge 2} \frac{1}{2k (4 \pi^2)^{2k} } E_{2k} (E,\omega)\op{Tr} (\alpha^{2k} ) &= \\
 \sum_{k \ge 2} \frac{(2k-1)!}{(4 \pi^2)^{2k} } E_{2k} (E,\omega) ch_{2k}(T_{B \g} ) & \in \Omega^{-\ast}( B \g). 
\end{align*}
\end{theorem}
We have already seen that the second Chern character of $B \g$ is an obstruction to solving the quantum master equation. Thus, in order to have the quantum theory, the second Chern character must be exact.  

The precise calculation is the following.
\begin{proposition}
The one-loop part of the scale $\infty$ effective action  $I^{(1)}[\infty]$, when restricted to $\C[\eps] \otimes \g[1]$,  is equal to 
$$
 \tfrac{1}{32 \pi^4} E_2^{ren} (E,\omega) \Tr(\alpha^2) + \sum_{k \ge 2} \frac{1}{2k (4 \pi^2)^{2k} } E_{2k} (E,\omega)\op{Tr} (\alpha^{2k} ) \in \Omega^{-\ast}( B \g).
$$
where $E_2^{ren}(E,\omega)$ is a certain renormalized Eisenstein function.
\end{proposition}
Because $\Tr(\alpha^2)$ is a multiple of the second Chern character and therefore exact, this proposition implies the previous theorem.

\begin{proof}[Proof of proposition]
Recall that we can write our effective action as
$$
I^{(1)} [\infty] = I^{(1)}_{naive}[\infty] + J[\infty]
$$
where $J[L]$ corrects for the failure of $I_{naive}[L]$ to solve the quantum master equation.

Let us further decompose $I_{naive}^{(1)}[\infty]$ as 
$$
I_{naive}^{(1)}[\infty] = I_{wheels}[\infty]  + I_{other}[\infty]
$$
where 
$$
I_{wheels}[L] = \sum_{\gamma \text{ is a wheel}} \frac{1}{\abs{\op{Aut}(\gamma)} } W_\gamma( P(0,L), I_{hCS}) .
$$
(We say that a graph is a wheel if it is a connected graph with first Betti number $1$, with the property that we can not disconnect the graph by removing a single vertex. This implies that the vertex are arranged cyclically around a circle).

Also,
$$
I_{other}[L] = \sum_{\substack {\gamma \text{ is a one loop graph} \\ \text{which is not a wheel} } } \frac{1}{\abs{\op{Aut}(\gamma)} } W_\gamma( P(0,L), I_{hCS}) .
$$
\begin{lemma}
When restricted to the subspace
$$
\mc H(E) \subset \Fields(E)
$$
of harmonic fields, both $J[\infty]$ and $I_{other}[\infty]$ vanish.
\end{lemma}
\begin{proof}[Proof of lemma]
Let $\d \zbar$ denote a translation-invariant $(0,1)$ form on $E$.  We can identify
$$
\mc H(E) = \C[\d \zbar] \otimes \left( \mf g[1] \oplus \mf g^\vee[-1] \right).
$$

Let $\gamma$ be a tree with $k+1$ tails. If we choose a root tail for $\gamma$, we can interpret the weight of $\gamma$ as an operator
$$
W'_\gamma (P(\eps,L), I_{hCS} )  : \Fields^{\otimes k} \to \Fields.
$$
This operator is related to the previously-defined weight of $\gamma$ by
$$
W_\gamma ( P (\eps,L), I_{hCS}) (\alpha_1,\ldots,\alpha_{k+1} ) = \ip{ W'_\gamma (P(\eps,L), I_{hCS} ) (\alpha_1,\ldots,\alpha_k) , \alpha_{k+1}}.
$$
Here $\ip{-,-}$ is odd symplectic form on $\Fields(E)$.  

The operator $W'_\gamma( P(\eps,L), I_{hCS})$ is a composition of operators defined for each vertex and each internal edge. The operator at a vertex of valency $k+1$ arises from the $L_\infty$ structure map
$$l_k : \left\{ \Omega^{0,\ast}(E) \otimes \left( \mf g \oplus \mf g^\vee[-2] \right) \right\}^{\otimes k} \to  \Omega^{0,\ast}(E) \otimes \left( \mf g \oplus \mf g^\vee[-2] \right). $$

The operator for each edge is just the operator associated to the kernel $P(\eps,L)$, namely the operator
\begin{align*}
 \Omega^{0,\ast}(E) \otimes \left( \mf g \oplus \mf g^\vee[-2] \right)  & \to  \Omega^{0,\ast}(E) \otimes \left( \mf g \oplus \mf g^\vee[-2] \right) \\
\alpha & \mapsto \dbar^\ast\int_{\eps}^L  e^{-t [\dbar, \dbar^\ast] } \alpha \d t. 
\end{align*}
Let us suppose that the tree $\gamma$ has at least one internal edge.  Then, if each $\alpha_i \in \mc H(E)$, we must have
$$
W'_\gamma (P(\eps,L), I_{hCS} ) (\alpha_1,\ldots,\alpha_k) = 0.
$$
This is just because the $L_\infty$ structure on $ \Omega^{0,\ast}(E) \otimes \left( \mf g \oplus \mf g^\vee[-2] \right) $ preserves the subspace of harmonic Dolbeaut forms, and the operator $e^{- t [\dbar, \dbar^\ast] }$ annihilates this subspace.

Now, let $\gamma$ be a one-loop graph which is not a wheel.  Let $\gamma_{wheel} \subset \gamma$ be the largest wheel containing $\gamma$.  Then we can view $\gamma$ as obtained by grafting some trees onto $\gamma_{wheel}$.  The weight $W_\gamma(P(\eps,L), I)$ is obtained by composing the operators associated to these trees to $W_{\gamma_{wheel}} (P(\eps,L)$.   Since the operators associated to these trees are zero on $\mc H(E)$, it follows that $W_{\gamma}(P(\eps,L))$ is zero on $\mc H(E)$. 

Next, let us consider $J[L]$. We can write  
$$
J[L] = \sum_{\gamma} W_\gamma ( P(0,L), I_{hCS} + \hbar J)
$$
where the sum is over all trees $\gamma$, one of whose vertices is of genus $1$.  Also, $J$ denotes the local functional which kills the obstruction $O$.

The same argument as before shows that, if $\gamma$ is a tree with at least one internal edge which appears in this sum, then
$$
W_\gamma( P(0,L) , I_{hCS} + \hbar J ) \mid_{\mc H(E)} = 0.
$$
It remains to check that if $\gamma$ is the unique tree with no internal edges, then the weight is zero on $\mc H(E)$ .   Let $v_k$ denote the $k$-valent tree with one vertex of genus one and no other vertices.  We can identify
$$
W_{v_k} ( P (0,L) , I_{hCS} + \hbar J ) = J_k
$$
where
$$
J_k : \Fields(E)^{\otimes k} \to \C
$$
is the part of $J$ which is homogeneous of degree $k$.

Thus, it remains to show that $J \mid_{\mc H(E)}= 0$.  To see this, observe that the fact that the local action functional $J$ is invariant under $\op{Aff}(\C)$ implies that it contains at least one $z$-derivative, and is therefore zero when restricted to harmonic functions. 
\end{proof}

The following lemma will complete the proof. 
\begin{lemma}
When we identify 
$$
\Oo( \C[\eps] \otimes \g [1] )= \Omega^{-\ast} (B \g),
$$
then, for all $k > 1$,
$$
\sum_{\substack { \gamma \text{ is a wheel} \\ \text{with } 2k \text{ vertices} } } \frac{1}{\abs{\op{Aut}(\gamma)}} W_\gamma( P (0,\infty) ,    I_{hCS}) )  \mid_{ \C[\eps] \otimes \g[1] } 
$$
corresponds to 
$$
\tfrac{1}{2k (4 \pi^2)^{2k}} E_{2k} (E,\omega) \Tr (\alpha^{2k}) \in H^k \left( \Omega^{-k} (B \g) \right).
$$
Further, if $k \ge 1$, 
$$
\sum_{\substack { \gamma \text{ is a wheel} \\ \text{with } 2k-1 \text{ vertices} } } \frac{1}{\abs{\op{Aut}(\gamma)}} W_\gamma( P (0,\infty) ,    I_{hCS}) )  \mid_{ \C[\eps] \otimes \g[1] } = 0.
$$
\end{lemma}
\begin{proof}
We will only prove the first statement; verifying that the sum over wheels with $2k-1$ vertices yields zero is easy.

Recall that the Atiyah class $\alpha$ is an element
$$
\alpha \in C_{red}^\ast(\g, \g^\vee \otimes  \g^\vee \otimes \g).
$$
We will view $\alpha$ as an element
$$
\alpha \in   C_{red}^\ast(\g) \left( \eps  \g^\vee \right) \otimes \op{End}(\g) \subset  C_{red}^\ast( \g \otimes \C[\eps] ) \otimes \op{End}(\g).
$$

Let us restrict the classical holomorphic Chern-Simons action
$$
I_{hCS} : \Omega^{0,\ast}(E) \otimes \left( \g[1] \oplus \g^\vee[-1] \right)  \to \C
$$
to a functional
$$
\til{I}_{hCS} : \C[\eps] \otimes \left( \g[1] \oplus \g^\vee[-1] \right) \to \C.
$$
This functional is linear in $\g^\vee$ and linear in $\eps \g$.  Thus, we can view $\til{I}_{hCS}$ as an element
$$
\til{I}_{hCS} \in C_{red}^\ast(\g ) \otimes \eps \g^\vee  \otimes \g \subset C_{red}^\ast (\g \otimes \C[\eps] ) \otimes \g.
$$
Let
$$
\d_{dR} : C_{red}^\ast(\g) \to C^\ast(\g) \otimes \g^\vee
$$
be the de Rham differential.  Let us extend $\d_{dR}$ to a map
$$
\d_{dR} \otimes 1 : C^\ast_{red}(\g) \otimes \eps \g^\vee \otimes \g \to C^\ast_{red}(\g) \otimes \eps \g^\vee \otimes \op{End}(\g).
$$
Lemma \ref{lemma_atiyah_hcs} implies that
$$
(\d_{dR} \otimes 1)  I_{hCS} = \alpha \in C_{red}^\ast(\g) \otimes \eps \g^\vee \otimes \op{End}(\g) \subset C_{red}^\ast( \g \otimes \C[\eps]  ) \otimes \op{End}(\g).
$$

\subsection{}

 Let us identify
$$
E = \C / \Lambda
$$
where $\Lambda \subset \C$ is a lattice. We will make this identification in such a way that the volume form $\omega$ on $E$ pulls back to $\d z$.  

Via this identification, we can identify translation invariant geometric objects on $\C$ with geometric objects on $E$. Thus, we can talk about $\d \zbar \in \Omega^{0,1}(E)$, and the derivations $\tfrac{\d}{\d z}$ and $\tfrac{\d}{\d \zbar}$ of $\cinfty(E)$.

Let
$$
\mu = \constinv  \int_{E} \d z \wedge \d \zbar. 
$$
The isomorphism
$$
\C[\eps] \to H^\ast( E, \Oo_E) = H^\ast (\Omega^{0,\ast}(E)) 
$$
sends
$$
\eps \mapsto \mu^{-1} \d \zbar.
$$

\subsection{}
There is an isomorphism of commutative graded algebras
$$
\cinfty(E) \otimes \C[\eps] \iso \Omega^{0,\ast}(E) 
$$
which sends $\eps \to \mu^{-1} \d \zbar$, as before.  Under this isomorphism, the trace map $\Omega^{0,1}(E) \to \C$ corresponds to the map
\begin{align*}
\eps \cinfty(E) & \to \C \\
\eps f & \mapsto \mu^{-1}  \int_{E} f  \d z \d \zbar
\end{align*}

This isomorphism leads to an isomorphism
$$
\Fields(E) =  \Omega^{0,\ast}(E) \otimes \left( \g[1] \oplus \g^\vee[-1] \right) \iso \cinfty(E) \otimes \C[\eps] \otimes \left( \g[1] \oplus \g^\vee[-1] \right). 
$$

Let 
$$
C_\g \in \left( \g[1] \oplus \g^\vee[-1] \right)^{\otimes2}
$$
be the element which corresponds, using the pairing on $\g[1] \oplus \g^\vee[-1]$, to the identity map on $\g[1] \oplus \g^\vee[-1]$.  

Let
$$
K_t^{\text{scalar}} \in \cinfty(E) \otimes \cinfty(E)
$$
be the scalar heat kernel.  Thus, if $D : \cinfty(E) \to \cinfty(E)$ is the Laplacian, then $K_t^{\text{scalar}}$ has the property that for all $f \in \cinfty(E)$,
$$
(e^{-t D} f )(z) \int_{w \in E} K_t(z,w) f(w) \d w \d \wbar \mu^{-1}.
$$

The heat kernel $K_t \in \Fields(E) \otimes \Fields(E)$ can be identified with
$$
K_t^{scalar} \otimes (\eps \otimes 1 - 1 \otimes \eps) \otimes C_\g .
$$
The propagator is
$$
P(0,\infty) = \int_{0}^\infty (\dbar^\ast \otimes 1) K_t.
$$
Since
$$
\dbar^\ast ( f (z,\zbar) \d \zbar) = \frac{\d}{\d z} f ,
$$
and $\eps = \mu^{-1} \d \zbar$, we can identify the propagator as
$$
P(0,\infty) = \int_{0}^\infty \mu^{-1} \frac{\d}{\d z}  K_t^{scalar}  \otimes \C_\g.
$$
Note that the kernel $\int_{0}^\infty \mu^{-1} \frac{\d}{\d z}  K_t^{scalar}$ corresponds to the operator
$$
\mu^{-1} \frac{\d}{\d z} D^{-1} : \cinfty(E) \to \cinfty(E).
$$
\subsection{}
 We will view the element
$$
\alpha \otimes \mu^{-1} \frac{\d}{\d z} D^{-1} \in C_{red}^\ast(\g\otimes \C[\eps] ) \op{End} (  \g [1] \otimes \cinfty(E) )
$$
an endomorphism of the $C^\ast_{red}(\g \otimes \C[\eps])$ module $C^\ast_{red}(\g \otimes \C[\eps]) \otimes \g[1] \otimes \cinfty(E)$. 
	
The weight of any wheel can be viewed as the trace of a composition of operators.  A simple Feynman diagram computation shows that
\begin{multline*}
\sum_{ \substack { \gamma \text{ is a wheel} \\ \text{with } 2k \text{ vertices} } } \frac{1}{\abs{\op{Aut}(\gamma)}} W_\gamma( P (0,\infty) ,    I_{hCS}) )  \mid_{ \C[\eps] \otimes \g[1] } 
\\= \frac{1}{2k}\op{Tr} \left( \left( \alpha \otimes \mu^{-1} \tfrac{\d}{\d z} D^{-1}  \right) ^{2k } \right) \in C^\ast(\g \otimes \C[\eps] ).
\end{multline*}
Now, observe that
$$
\op{Tr} \left( \left( \alpha \otimes \mu^{-1} \tfrac{\d}{\d z} D^{-1}  \right) ^{2k } \right) = \op{Tr} (\alpha^{2k} ) \op{Tr} \left( \left(\mu^{-1}  \tfrac{\d}{\d z} D^{-1} \right)^{2k} \right).
$$

\subsection{}
To complete the proof, we need to show that
$$
\op{Tr} \left( \left(\mu^{-1} \tfrac{\d}{\d z} D^{-1} \right)^{2k} \right)  = \frac{1}{(4 \pi^2)^{2k} } E_{2k}.
$$
Recall that our elliptic curve $E$ can be written as
$$
E = \C / \Lambda
$$
where $\Lambda \subset \C$ is a lattice, and the pull-back of the volume form $\omega$ on $E$ to $\C$ is $\d z$.

Let $\alpha + i \beta \in \Lambda$ and $\delta + i \eps \in \Lambda$ be generators for the lattice, where $\alpha,\beta,\delta,\eps \in \R$. 

Recall that
$$
\mu = \pi i  \int_{E} \d z \d \zbar.
$$
Also,
$$
 \det \left(   \begin{matrix} 
      \alpha & \beta \\
      \delta & \eps \\
   \end{matrix}\right) = \frac{1}{2i} \int_{E} \d z \d \zbar = \frac{-1}{ 2 \pi } \mu
$$

If $n,m \in \Z$, let
\begin{align*}
F_{n,m} (x,y) &= \exp \left\{  (\alpha \eps - \beta \delta)^{-1}  2 \pi i \left( n \beta x - n \alpha y +  m \eps x - m \delta y \right)   \right\} \in \cinfty(\C) \\
&= \exp \left\{  -  \mu^{-1} 4 i \pi^2    \left( n \beta x - n \alpha y +  m \eps x - m \delta y \right)   \right\}
\end{align*}

Note that
\begin{align*}
F_{n,m} ( x + \alpha, y + \beta) &= F_{n,m}(x,y) \\
F_{n,m}( x + \delta, y + \eps ) &= F_{n,m}(x,y).
\end{align*}
Thus, $F_{n,m}(x,y)$ is invariant under $\Lambda$, and descends to a smooth function on $E$.  As $n,m$ range over $\Z \times \Z$, the functions $F_{n,m}(x,y)$ form a basis for the space of smooth functions on the elliptic curve $E$.  

Note that
\begin{align*}
\frac{\d}{\d z} F_{n,m} &=  - \mu^{-1}  4i \pi^2 \left(  n \beta +  m \eps + i n \alpha + i m \delta  \right) F_{n,m}  \\
\frac{\d}{\d \zbar} F_{n,m} &= - \mu^{-1} 4 i \pi^2  \left(  n \beta +  m \eps - i n \alpha - i m \delta  \right) F_{n,m}.
\end{align*}
Thus,
$$
\frac{\d}{\d z} D^{-1} F_{n,m} = -\frac{ \mu (4 i \pi^2)^{-1}  F_{n,m}   }{     n \beta +  m \eps - i n \alpha - i m \delta   } .
$$
It follows that
\begin{align*}
\op{Tr}_{\cinfty(E)}  \left( \left( \mu^{-1} \frac{\d}{\d z} D^{-1} \right)^{2k}  \right) &=  \sum_{(n,m) \in \Z \times \Z} \frac{(4 i \pi^2)^{-2k}}{ \left(n \beta - n i \alpha + m \eps - m i \eps\right)^{2k}  }   \\
&=\frac{1}{(4 i \pi^2)^{2k} } \sum_{\lambda \in \Lambda} \frac{1}{ (i \lambda)^{2k}   }\\
&=  \frac{1}{( 4 \pi^2)^{2k} }  E_{2k}.
\end{align*}

\end{proof}

\end{proof}

%% file: r_appendix.txt
\section{Appendix}
In this appendix, we will prove that maps from a dg ringed space to an $L_\infty$ space satisfy \v{C}ech descent. 

First, we will recall the definition of the \v{C}ech complex $\check{C}(\mf{U}, F)$ with coefficients in a simplicial presheaf $F$ on a topological space $X$, associated to an open cover $\mf{U} = \{U_i \mid i \in I \}$ of $X$.  If $U \subset X$, we will let $F(U)[k]$ denote the set of $k$-simplices of $F(U)$. 

Let $[k] = \{0,1,\dots, k\}$ be the set with $k+1$ elements.  If $\phi : [k] \to I$, we will let $U_\phi$ be $\cap_{i = 0}^k U_{\phi(i)}$.  

A $0$-simplex $\alpha$ of $\check{C}(\mf{U}, F ) $ is a function which assigns to each $\phi : [k] \to I$ an element
$$
\alpha(\phi)  \in F( U_{\phi} ) [k]
$$
satisfying certain incidence relations.  A non-decreasing map $f : [k] \to [l]$ induces a map 
$$f^\ast : F(U)[l] \to F(U)[k]$$ 
for each $U \subset X$.   Further, for each $\phi : [l] \to I$, 
$$U_{\phi} \subset U_{\phi \circ f }.$$
We require that, for all $\phi : [l] \to I$, and all non-decreasing maps $f : [k] \to [l]$, 
$$
f^\ast \alpha (\phi ) = \alpha ( \phi \circ f ) \mid_{U_{\phi }} \in F (U_\phi ) [k]. 
$$

An $n$-simplex of $\check{C}(\mf{U}, F)$ is defined to be a function $\alpha$, which to each $\phi : [l] \to I$ as above assigns a map of simplicial sets
$$
\alpha(\phi) : \tr^n \times \tr^l \to F( U_\phi ) ,
$$
satisfying the same incidence relation: if $f : [k] \to [l]$ is a non-decreasing map, then
$$
f^\ast \alpha(\phi) = \alpha ( \phi \circ f ) \mid_{U_\phi } 
$$
as a map $\tr^n \times \tr^k \to F(U_\phi)$.  

\subsection{}
Next, let us restate the theorem concerning maps of $L_\infty$ spaces. 
\begin{theorem}
Suppose that $\phi : (Y,\g_Y) \to (X,\g_X)$ is an equivalence of $L_\infty$ spaces. 

Then, for all $L_\infty$ spaces $(Z,\g_Z)$, the maps of simplicial sets
\begin{align*}
\op{Maps}( (Z, \g_Z) , (Y,\g_Y)) &\to \op{Maps}( (Z, \g_Z) , (X,\g_X)) \\
\op{Maps}( (X, \g_X) , (Z,\g_Z)) &\to \op{Maps}( (Y, \g_Y) , (Z,\g_Z))
\end{align*}
given by composing with $\phi$ are both weak homotopy equivalences.  

Further, for all $L_\infty$ spaces $X,Y$ the simplicial presheaf $\op{Maps}((X,\g_X), (Y,\g_Y))$ on $X$ which sends $U \to \op{Maps}((U, \g_X \mid_U ) , (Y, \g_Y) )$ is a homotopy sheaf: that is, for any open subset $U \subset X$,  and any open cover $\mf{U}$ of $U$, the natural map
$$
\Gamma (U,  \op{Maps}((X,\g_X), (Y,\g_Y)) ) \to  \check{C}(\mf{U},  \op{Maps}((X,\g_X), (Y,\g_Y)) )
$$
is a weak equivalence. 
\end{theorem}

\begin{proof}
Let us fix a smooth map $f : Z \to Y$.   We have three curved $L_\infty$ algebras $\g_Z$, $f^\ast \g_Y$, and $f^\ast \phi^\ast \g_X$ over $\Omega^\ast_Z$, and a map $f^\ast \g_Y \to f^\ast \phi^\ast \g_X$.  

Let $A_X$ denote the sheaf on $Z$ dga $C^\ast(f^\ast \phi^\ast \g_X)$, where cochains are of course taken over $\Omega^\ast_Z$.  Let $I_X \subset A_X$ denote the ideal generated by $C^{> 0}(f^\ast \phi^\ast \g_X)$ and $\Omega^{> 0}_Z$.  Define $(A_Y, I_Y)$ in the same way, using $f^\ast  \g_Y$ in place of $f^\ast \phi^\ast \g_X$.   Note that there is a map of commutative dgas over $\Omega^\ast_Z$ 
$$
A_X \to A_Y
$$
which takes the ideal $I_X$ to the ideal $I_Y$. 

Let $\op{MC}( A_X \otimes \g_Z)$ denote the simplicial presheaf on $Z$ whose  set of $n$-simplices,  on an open set $U \subset Z$, is the set of Maurer-Cartan elements
$$
\alpha \in A_X(U) \otimes_{\Omega^\ast(U)} \g_Z(U) \otimes_{\R} \Omega^\ast(\tr^n)
$$
which vanish modulo the ideal $I_X (U) \subset A_X(U)$. 

The simplicial set of lifts of the smooth map $f : Z \to Y$ to a map of $L_\infty$ spaces is (by definition) the simplicial set of global sections of the simplicial presheaf  $\op{MC}(A_Y \otimes \g_Z)$.

We need to show that the natural map of simplicial sets
$$
\Gamma(Z, \op{MC}( A_X \otimes \g_Z) ) \to  \Gamma(Z, \op{MC}(A_Y \otimes \g_Z) )
$$
is a weak homotopy equivalence.  

This is proved by Artinian induction.  Let $I_X^n \subset A_X$ denote the $n^{\mathrm{th}}$ power of the ideal $I_X$.  

The first lemma is the following.
\begin{lemma}
The map of simplicial sets
$$
\Gamma(Z, \op{MC}(A_X / I_X^n \otimes \g_Z) ) \to \Gamma(Z, \op{MC}(A_X / I_X^{n-1} \otimes \g_Z)  )
$$
is a fibration. 
\end{lemma}
\begin{proof}
There is a short exact sequence of sheaves of graded vector spaces
$$
I_X^{n-1} / I_X^{n} \otimes \g_Z \to A_X / I_X^n \otimes \g_Z \to A_X / I_X^{n-1} \otimes \g_Z .
$$
Note that $I_X^{n-1} / I_X^n \otimes \g_Z$ is a cochain complex.  So, this exact sequence expresses the curved $L_\infty$ algebra $A_X / I_X^n \otimes \g_Z$ as a central extension of $A_X / I_X^{n-1} \otimes \g_Z$ by the cochain complex $I_X^{n-1} / I_X^n \otimes \g_Z$.  

Because are working in the $\cinfty$ context, this  exact sequence of graded vector spaces splits.  

Now, suppose we have a Maurer-Cartan element 
$$\alpha \in \Gamma(Z, \Omega^\ast(\tr^n ) \otimes A_X / I_X^{n-1} \otimes \g_Z).$$
Let $\til{\alpha}$ be any lift to an element 
$$
\til{\alpha} \in \Gamma(Z, \Omega^\ast(\tr^n ) \otimes A_X / I_X^{n} \otimes \g_Z).
$$

The obstruction to $\til{\alpha}$ satisfying the Maurer-Cartan equation is a cohomology class in 
$$
\Gamma(Z,I_X^{n-1} / I_X^n \otimes \g_Z \otimes \Omega^\ast (\tr^n)).
$$
Now, suppose we know that $\til{\alpha}$ satisfies the Maurer-Cartan equation when restricted to some horn $H \subset \tr_n$.  Then, since the obstruction class vanishes when restricted to the horn.  Since the map $\Omega^\ast(\tr_n) \to \Omega^\ast(H)$ is a quasi-isomorphism, it follows that the obstruction to lifting $\alpha$ also vanishes. 
\end{proof}

\begin{lemma}
The natural map of simplicial presheaves
$$ 
\Gamma(U,\op{MC}(A_Y \otimes \g_Z) ) \to \Gamma(U, \op{MC}(A_X \otimes \g_Z) ))
$$
is a weak homotopy equivalence for all open subsets $U \subset Z$.
\end{lemma}
\begin{proof}
We can write 
$$
\Gamma(U, \op{MC}(A_X \otimes \g_Z) ) = \liminv \Gamma(U, \op{MC}(A_X / I_X^n \otimes \g_Z) )
$$
where the maps in the inverse limit are all fibrations; and similarly for $\Gamma(U, \op{MC}(A_Y \otimes \g_Z) )$.  The map is compatible with the inverse systems, so to check it's a weak equivalence we need only check that the maps
$$
\Gamma(U, \op{MC}(A_X / I_X^n\otimes \g_Z) )
\to \Gamma(U, \op{MC}(A_Y / I_Y^n \otimes \g_Z) )
$$
are weak equivalences.  By induction on $n$, it suffices to verify that we have weak equivalences on the fibres of the maps 
$$
\Gamma(U, \op{MC}(A_X / I_X^n\otimes \g_Z) )
\to \Gamma(U, \op{MC}(A_Y / I_X^{n-1} \otimes \g_Z) ).
$$
The fibres are the Dold-Kan simplicial sets associated to the cochain complexes $\Gamma(U, I_X^{n-1}/ I_X^n \otimes \g_Z)$ and $\Gamma(U, I_Y^{n-1}/ I_Y^n \otimes \g_Z)$.

Now, since the map of sheaves of cochain complexes 
$$
I_X^{n-1}/ I_X^n \to I_Y^{n-1}/ I_Y^n
$$
are, by assumption, homotopy equivalences, the result follows. 
\end{proof}
This proves the first statement of the theorem.  The second statement, that the map
$$
\op{Maps}( (Y, \g_Y) , (Z,\g_Z)) \to \op{Maps}( (X, \g_X) , (Z,\g_Z))
$$
is a weak equivalence, is proved by a similar argument.  

Thus, to complete the proof of the theorem, we need to show the following.
\begin{lemma}
The simplicial presheaf 
$$
U \mapsto \Gamma(U, \op{MC}(A_X \otimes \g_Z) ) 
$$
satisfies \v{C}ech descent.  That is, if $\mc{V} = \{V_i\}$ is an open cover of $U$, then the map
$$
\Gamma(U, \op{MC}(A_X \otimes \g_Z) ) \to \check{C} (\mc{V}, \op{MC}(A_X \otimes \g_Z) )
$$
is a weak homotopy equivalence.
\end{lemma}
\begin{proof}
We have seen that $\op{MC}(A_X \otimes \g_Z) $ is an inverse limit of $\op{MC}(A_X / I_X^n \otimes \g_Z)$, and that the maps in the inverse system are fibrations.  Thus, the \v{C}ech simplicial set $\check{C} (\mc{V}, \op{MC}(A_X \otimes \g_Z) )$ is an inverse limit
$$
\check{C} (\mc{V}, \op{MC}(A_X \otimes \g_Z) ) = \liminv \check{C} (\mc{V}, \op{MC}(A_X /  I_X^n \otimes \g_Z) ),
$$
and again the maps in the inverse system are again all fibrations.   Thus, to prove the lemma, it suffices to verify that the map
$$
\Gamma(U, \op{MC}(I_X^{n-1} / I_X^n \otimes \g_Z) ) \to \check{C} (\mc{V}, \op{MC}(I_X^{n-1}/ I_X^n \otimes \g_Z) )
$$
is a weak equivalence. But, $\op{MC}(I_X^{n-1}/ I_X^n \otimes \g_Z) )$ is the Dold-Kan simplicial presheaf associated to the sheaf of cochain complexes $I_X^{n-1} / I_X^n \otimes \g_Z$.  Since we are working in a $\cinfty$ context, partitions of unity allow one to show as usual that the cohomology with coefficients in the sheaf of complexes $I_X^{n-1}  / I_X^n \otimes \g_Z$ is the same as the cohomology of the global sections of $I_X^{n-1}/ I_X^n \otimes \g_Z$.
\end{proof}
This completes the proof of the theorem. 
\end{proof}

%% file: elliptic.bbl
\begin{thebibliography}{AKMP87}


\bibitem[AKMW87]{AlvKilMan87}
O.~Alvarez, T.~Killinback, M.~Mangano and P.~Windey, \textsl{ String theory and
  loop space index theorems},
\newblock Comm. Math. Phys. \textbf{ 111} (1987).

\bibitem[CG11]{CosGwi11}
K.~Costello and O.~Gwilliam, \textsl{ Factorization algebras in perturbative
  quantum field theory},
\newblock Available at \\
  \verb=http://math.northwestern.edu/~costello/renormalization.html=  (2011).

\bibitem[Cos10]{Cos10a}
K.~Costello,
\newblock A geometric construction of the {W}itten genus, I,
\newblock in \textsl{ Proceedings of the {I}nternational {C}ongress of
  {M}athematicians, ({H}yderabad, 2010)}, 2010.

\bibitem[Cos11a]{Cos11b}
K.~Costello, \textsl{ Notes on supersymmetric and holomorphic field theories in
  dimensions $2$ and $4$},
\newblock (2011).

\bibitem[Cos11b]{Cos11}
K.~Costello,
\newblock \textsl{ Renormalization and effective field theory},
\newblock Surveys and monographs, American Mathematical Society, 2011.

\bibitem[CVdB10]{CalBer10}
D.~Calaque and M.~Van~den Bergh, \textsl{ Hochschild cohomology and {A}tiyah
  classes},
\newblock Adv. Math. \textbf{ 224}(5), 1839--1889 (2010), {arXiv:0708.2725}.

\bibitem[GG11]{GraGwi11}
R.~Grady and O.~Gwilliam, \textsl{ One-dimensional Chern-Simons theory and the
  $\what{A}$ genus},
\newblock (2011), {arXiv:1110.3533}.


\bibitem[GMS00]{GorMalSch00}
V.~Gorbounov, F.~Malikov and V.~Schechtman, \textsl{ Gerbes of chiral
  differential operators},
\newblock Math. Res. Lett. \textbf{ 7}(1), 55--66 (2000).

\bibitem[Hin01]{Hin01}
V.~Hinich, \textsl{ D{G} coalgebras as formal stacks},
\newblock J. Pure Appl. Algebra \textbf{ 162}(2-3), 209--250 (2001).

\bibitem[Kap97]{Kap97}
M.~M. Kapranov, \textsl{ {R}ozansky-{W}itten invariants via the {A}tiyah
  class},
\newblock (1997), {arXiv:alg-geom/9704009}.

\bibitem[Kos85]{Kos85}
J.-L. Koszul,
\newblock Crochet de {S}chouten-{N}ijenhuis et cohomologie,
\newblock in \textsl{ Elie {C}artan et les mathematiques d'aujourd'hui}, pages
  257--271, Asterisque, 1985.

\bibitem[KS01]{KonSoi01}
M.~Kontsevich and Y.~Soibelman, \textsl{ Homological mirror symmetry and torus
  fibrations},
\newblock (2001), {math.SG/0011041}.

\bibitem[Lur09a]{Lur09b}
J.~Lurie, \textsl{ Derived Algebraic Geometry V: Structured Spaces},
\newblock (2009), {arXiv:0905.0459}.

\bibitem[Lur09b]{Lur09}
J.~Lurie, \textsl{ On the classification of topological field theories},
\newblock (2009), {arXiv:0905.0465}.

\bibitem[Lur10]{Lur10}
J.~Lurie, \textsl{ Moduli problems for ring spectra},
\newblock (2010).

\bibitem[Mer99]{Mer99}
S.~A. Merkulov, \textsl{ Strong homotopy algebras of a {K}\"ahler manifold},
\newblock Internat. Math. Res. Notices (3), 153--164 (1999), {math.AG/9809172}.

\bibitem[Qui69]{Qui69}
D.~Quillen, \textsl{ Rational homotopy theory},
\newblock Ann. of Math. (2) \textbf{ 90}(2), 205--295 (1969).

\bibitem[Roz11]{Roz11}
N.~Rozenblyum, \textsl{ Connections on conformal blocks},
\newblock MIT thesis  (2011).

\bibitem[To{\"e}06]{Toe06}
B.~To{\"e}n, \textsl{ Higher and derived stacks: a global overview},
\newblock (2006), {arXiv:math/0604504}.

\bibitem[Wit87]{Wit87}
E.~Witten, \textsl{ Elliptic genera and quantum field theory},
\newblock Comm. Math. Phys. \textbf{ 109}(4), 525--536 (1987).

\end{thebibliography}
